\documentclass[twoside,11pt]{article}
\pdfoutput=1

\usepackage{jmlr2e-fixed}
\usepackage{mathrsfs}
\usepackage{amsmath}
\usepackage{microtype}
\usepackage{enumerate}
\usepackage{tikz}
\usepackage{hyperref}
\usepackage{bookmark}

\newcommand{\arXiv}[1]{arXiv:\href{http://arXiv.org/abs/#1}{#1}}

\newtheorem{examp}[theorem]{Example}  

\renewenvironment{proof}[1][Proof]{\par\noindent{\bf #1\ }}{\hfill\BlackBox\\[2mm]}

\newcommand{\citel}[1]{\citefullauthor{#1}, \citeyear{#1}}

\def\arXiv#1{arXiv:\href{http://arXiv.org/abs/#1}{#1}}
\usepackage{doi}

\newcommand{\SymmDiff}{\mathop{\triangle}}

\newcommand{\E}{\mathbb{E}}
\newcommand{\N}{\mathbb{N}}
\newcommand{\Q}{\mathbb{Q}}

\newcommand{\R}{\mathbb{R}}
\renewcommand{\Pr}{\mathbb{P}}

\newcommand{\eps}{\varepsilon}
\newcommand{\1}{\mathbf{1}}

\def\cW{\mathcal{W}}
\def\cV{\mathcal{V}}
\def\cU{\mathcal{U}}

\def\cS{\mathcal{S}}

\def\cF{\mathcal{F}}

\def\cA{\mathcal{A}}

\newcommand{\BB}{\mathbb}
\newcommand{\ol}{\overline}

\newcommand{\op}{\operatorname}

\newcommand{\scr}{\mathscr}
\newcommand{\eqD}{\overset{d}{=}}
\newcommand{\ep}{\epsilon}
\newcommand{\rta}{\rightarrow}

\newcommand{\wt}{\widetilde}
\newcommand{\wh}{\widehat}
\newcommand{\mcl}{\mathcal}

\newcommand{\W}{\mcl{W}}

\renewcommand{\P}{\mcl{P}}

\renewcommand{\epsilon}{\varepsilon}

\DeclareMathOperator{\supp}{supp}

\DeclareMathAlphabet{\mathpzc}{OT1}{pzc}{m}{it}

\usepackage{lastpage}
\jmlrheading{18}{2018}{1--\pageref{LastPage}}{8/16; Revised 6/17}{5/18}{16-421}
{Christian Borgs, Jennifer T.\ Chayes, Henry Cohn, and Nina Holden}

\ShortHeadings{Limits of Sparse Exchangeable Graphs}{Borgs, Chayes, Cohn, and Holden}

\firstpageno{1}

\begin{document}

\title{Sparse Exchangeable Graphs and Their Limits\\ via Graphon Processes}

\author{\name Christian Borgs
	\email borgs@microsoft.com \\
    \addr Microsoft Research\\
       One Memorial Drive\\
	   Cambridge, MA 02142, USA
    \AND
	\name Jennifer T.\ Chayes
	\email jchayes@microsoft.com \\
	\addr Microsoft Research\\
	One Memorial Drive\\
	Cambridge, MA 02142, USA
	\AND
	\name Henry Cohn
	\email cohn@microsoft.com \\
	\addr Microsoft Research\\
	One Memorial Drive\\
	Cambridge, MA 02142, USA
	\AND
	\name Nina Holden
	\email ninah@math.mit.edu \\
	\addr Department of Mathematics\\
	Massachusetts Institute of Technology\\
	Cambridge, MA 02139, USA
    }

\editor{Edoardo M.\ Airoldi}

\maketitle

\begin{abstract} %
In a recent paper, Caron and Fox suggest a probabilistic model for sparse
graphs which are exchangeable when associating each vertex with a time
parameter in $\R_+$. Here we show that by generalizing the classical
definition of graphons as functions over probability spaces to functions
over $\sigma$-finite measure spaces, we can model a large family of
exchangeable graphs, including the Caron-Fox graphs and the traditional
exchangeable dense graphs as special cases. Explicitly, modelling the
underlying space of features by a $\sigma$-finite measure space
$(S,\cS,\mu)$ and the connection probabilities by an integrable function
$W\colon S\times S\to [0,1]$, we construct a random family  $(G_t)_{t\geq
0}$ of growing graphs such that the vertices of $G_t$ are given by a
Poisson point process on $S$ with intensity $t\mu$, with two points $x,y$
of the point process connected with probability $W(x,y)$. We call such a
random family a \emph{graphon process}. We prove that a graphon process has
convergent subgraph frequencies (with possibly infinite limits) and that,
in the natural extension of the cut metric to our setting, the sequence
converges to the generating graphon. We also show that the underlying
graphon is identifiable only as an equivalence class over graphons with cut
distance zero. More generally, we study metric convergence for arbitrary
(not necessarily random) sequences of graphs, and show that a sequence of
graphs has a convergent subsequence if and only if it has a subsequence
satisfying a property we call \emph{uniform regularity of tails}. Finally,
we prove that every graphon is equivalent to a graphon on $\R_+$ equipped
with Lebesgue measure.
\end{abstract}


\section{Introduction}
\label{sec1}

The theory of graphons has provided a powerful tool for sampling and studying
convergence properties of sequences of dense graphs. Graphons characterize
limiting properties of dense graph sequences, such as properties arising in
combinatorial optimization and statistical physics. Furthermore, sequences of
dense graphs sampled from a (possibly random) graphon are characterized by a
natural notion of exchangeability via the Aldous-Hoover theorem. This paper
presents an analogous theory for sparse graphs.

In the past few years, graphons have been used as non-parametric extensions
of stochastic block models, to model and learn large networks. There have
been several rigorous papers on the subject of consistent estimation using
graphons (see, for example, papers by \citel{bickelchen}, \citel{BCL11},
\citel{roheetal}, \citel{CWA}, \citel{WO}, \citel{rateopt}, \citel{C15},
\citel{oracle}, and \citel{w-estimation}, as well as references therein), and
graphons have also been used to estimate real-world networks, such as
Facebook and LinkedIn (E.~M.~Airoldi, private communication, 2015).  This
makes it especially useful to have graphon models for sparse networks with
unbounded degrees, which are the appropriate description of many large
real-world networks.

In the classical theory of graphons as studied by, for example,
\citet*{graph-hom}, \citet*{ls-graphlimits}, \citet*{denseconv1},
\citet*{br-09}, \citet*{BCL10}, and \citet*{janson-survey}, a graphon is a
symmetric $[0,1]$-valued function defined on a probability space. In our
generalized theory we let the underlying measure space of the graphon be a
$\sigma$-finite measure space; i.e., we allow the space to have infinite
total measure. More precisely, given a $\sigma$-finite measure space $\scr
S=(S,\cS,\mu)$ we define a graphon to be a pair $\W=(W,\scr S)$, where
$W\colon S\times S\to\R$ is a symmetric integrable function, with the special
case when $W$ is $[0,1]$-valued being most relevant for the random graphs
studied in the current paper. We present a random graph model associated with
these generalized graphons which has a number of properties making it
appropriate for modelling sparse networks, and we present a new theory for
convergence of graphs in which our generalized graphons arise naturally as
limits of sparse graphs.

Given a $[0,1]$-valued graphon $\W=(W,\scr S)$ with $\scr S=(S,\cS,\mu)$ a
$\sigma$-finite measure space, we will now define a random process which
generalizes the classical notion of $\W$-random graphs, introduced in the
statistics literature \citep*{HRH02} under the name latent position graphs,
in the context of graph limits \citep*{ls-graphlimits} as $\W$-random graphs,
and in the context of extensions of the classical random graph theory
\citep*{bjr07} as inhomogeneous random graphs. Recall that in the classical
setting where $\W$ is defined on a probability space, $\W$-random graphs are
generated by first choosing $n$ points $x_1,\dots,x_n$ i.i.d.\ from the
probability distribution $\mu$ over the feature space $S$, and then
connecting the vertices $i$ and $j$ with probability $W(x_i,x_j)$.  Here,
inspired by \citet*{caron-fox}, we generalize this to arbitrary
$\sigma$-finite measure spaces by first considering a Poisson point
process\footnote{We will make this construction more precise in Section
\ref{sec:W-random-results}; in particular, we will explain that we may
associate $\Gamma_t$ with a collection of random variables $x_i\in S$. The
same result holds for the Poisson point process $\Gamma$ considered in the
next paragraph.} $\Gamma_t$ with intensity $t\mu$ on $S$ for any fixed $t>0$,
and then connecting two points $x_i,x_j$ in $\Gamma_t$ with probability
$W(x_i,x_j)$. As explained in the next paragraph, this leads to a family of
graphs $(\wt G_t)_{t\geq 0}$ such that the graphs $\wt G_t$ have almost
surely at most countably infinitely many vertices and (assuming appropriate
integrability conditions on $W$, e.g., $W\in L^1$) a finite number of edges.
Removing all isolated vertices from $\wt G_t$, we obtain a family of graphs
$( G_t)_{t\geq 0}$ that are almost surely finite. We refer to the families
$(\wt G_t)_{t\geq 0}$ and $(G_t)_{t\geq 0}$ as \emph{graphon processes}; when
it is necessary to distinguish the two, we call them graphon processes with
or without isolated vertices, respectively.

\begin{figure}
\begin{center}
\includegraphics{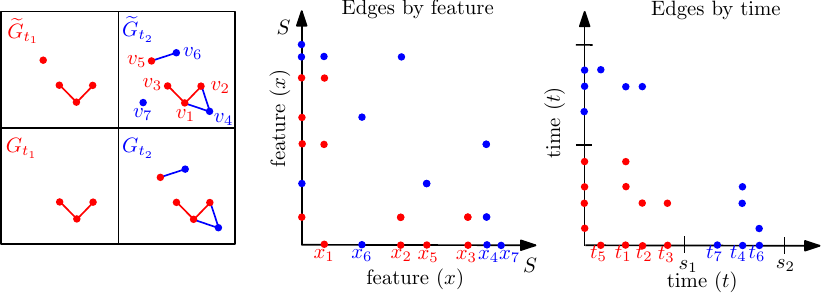}
\end{center}
\caption{This figure illustrates how we can generate
a graphon process $(G_t)_{t\geq 0}$ from a graphon $\W=(W,\scr S)$, where $\scr S=(S,\cS,\mu)$
is a $\sigma$-finite measure space. The two coordinate axes on the middle figure
represent our feature space $S$, where the red (resp.\ blue) dots on the axes represent
vertices born during $[0,s_1]$ (resp.\ $(s_1,s_2]$) for $0<s_1<s_2$, and the red
(resp.\ blue) dots in the interior of the first quadrant represent edges in $G_{t}$
for $t\geq s_1$ (resp.\ $t\geq s_2$). The graph $G_t$ is an induced subgraph of a graph
$\wt G_t$ with infinitely many vertices in the case $\mu(S)=\infty$, such that $G_t$ is
obtained from $\wt G_t$ by removing isolated vertices. At time $t\geq 0$ the marginal
law of the features of $V(\wt G_t)$ is a Poisson point process on $S$ with
intensity $t\mu$. Two distinct vertices with features $x$ and $x'$, respectively, are
connected to each other by an undirected edge with probability $W(x,x')$. The coordinate
axes on the right figure represent time $\BB R_+$. We get the graph $G_t$ by considering
the edges restricted to $[0,t]^2$. Note that the coordinate axes in the right figure and
the graphs $\wt G_t$ in the left figure are slightly inaccurate if we assume $\mu(S)=\infty$,
since in this case there are infinitely many isolated vertices in $\wt G_t$ for each $t>0$.
We have chosen to label the vertices by the order in which they appear in $G_t$,
where ties are resolved by considering the time the vertices were born, i.e., by
considering the time they appeared in $\wt G_t$.}
\label{fig1}
\end{figure}

To interpret the graphon process $(G_t)_{t\geq 0}$ as a family of growing
graphs we will need to couple the graphs $G_t$ for different times $t\geq 0$.
To this end, we consider a Poisson point process $\Gamma$ on $\R_+\times S$
(with $\R_+:=[0,\infty)$ being equipped with the Borel $\sigma$-algebra and
Lebesgue measure). Each point $v=(t,x)$ of $\Gamma$ corresponds to a vertex
of an infinite graph $\wt G$, where the coordinate $t$ is interpreted as the
time the vertex is born and the coordinate $x$ describes a feature of the
vertex. Two distinct vertices $v=(t,x)$ and $v'=(t',x')$ are connected by an
undirected edge with probability $W(x,x')$, independently for each possible
pair of distinct vertices. For each fixed time $t\geq 0$ define a graph $G_t$
by considering the induced subgraph of $\wt G$ corresponding to vertices
which are born at time $t$ or earlier, where we do not include vertices which
would be isolated in $G_t$. See Figure~\ref{fig1} for an illustration. The
family of growing graphs $(G_t)_{t\geq 0}$ just described  includes classical
dense $\W$-random graphs (up to isolated vertices) and the sparse graphs
studied by \citet*{caron-fox} and \citet*{kallenberg-block} as special cases,
and is (except for minor technical differences) identical to the family of
random graphs studied by \citet*{veitchroy}, a paper which was written in
parallel with our paper; see our remark at the end of this introduction.

The graphon process $(\wt G_t)_{t\geq 0}$ satisfies a natural notion of
exchangeability. Roughly speaking, in our setting this means that the
features of newly born vertices are homogeneous in time. More precisely, it
can be defined as joint exchangeability of a random measure in $\R_+^2$,
where the two coordinates correspond to time, and each edge of the graph
corresponds to a point mass. We will prove that graphon processes as defined
above, with $W$ integrable and possibly random, are characterized by
exchangeability of the random measure in $\R_+^2$ along with a certain
regularity condition we call \emph{uniform regularity of tails}. See
Proposition~\ref{prop5} in Section~\ref{sec:W-random-results}. This result is
an analogue in the setting of possibly sparse graphs satisfying the
aforementioned regularity condition of the Aldous-Hoover theorem
\citep*{aldous,hoover}, which characterizes $\W$-random graphs over
probability spaces as graphs that are invariant in law under permutation of
their vertices.

The graphon processes defined above also have a number of other properties
making them particularly natural to model sparse graphs or networks. They are
suitable for modelling networks which grow over time since no additional
rescaling parameters (like the explicitly given density dependence on the
number of vertices specified by \citel{br-09}, and \citel{lp1}) are
necessary; all information about the random graph model is encoded by the
graphon alone. The graphs are \emph{projective} in the sense that if $s<t$
the graph $G_s$ is an induced subgraph of $G_t$. Finally, a closely related
family of weighted graphs is proven by \citet*{caron-fox} to have power law
degree distribution for certain $\W$, and our graphon processes are expected
to behave similarly. The graphon processes studied in this paper have a
different qualitative behavior than the sparse $\W$-random graphs studied by
\citet*{br-09} and \citet*{lp1,lp2} (see Figure~\ref{fig3}), with the only
overlap of the two theories occurring when the graphs are dense.  If the
sparsity of the graphs is caused by the degrees of the vertices being scaled
down approximately uniformly over time, then the model studied by
\citet*{br-09} and \citet*{lp1,lp2} is most natural. If the sparsity is
caused by later vertices typically having lower connectivity probabilities
than earlier vertices, then the model presented in this paper is most
natural. The sampling method we will use in our forthcoming paper
\citep*{multiscale2} generalizes both of these methods.

\begin{figure}
\begin{center}
\begin{tikzpicture}[line width = 0.2pt]
\draw (0,0)--(5.00000,0);
\draw (5.00000,0)--(5.00000,5.00000);
\draw (5.00000,5.00000)--(0,5.00000);
\draw (0,5.00000)--(0,0);
\fill[black] (0,0.100000)--(0,0.200000)--(0.100000,0.200000)--(0.100000,0.100000)--cycle;
\fill[black] (0,0.400000)--(0,0.500000)--(0.100000,0.500000)--(0.100000,0.400000)--cycle;
\fill[black] (0,0.500000)--(0,0.600000)--(0.100000,0.600000)--(0.100000,0.500000)--cycle;
\fill[black] (0,0.900000)--(0,1.00000)--(0.100000,1.00000)--(0.100000,0.900000)--cycle;
\fill[black] (0,1.00000)--(0,1.10000)--(0.100000,1.10000)--(0.100000,1.00000)--cycle;
\fill[black] (0,1.50000)--(0,1.60000)--(0.100000,1.60000)--(0.100000,1.50000)--cycle;
\fill[black] (0,2.00000)--(0,2.10000)--(0.100000,2.10000)--(0.100000,2.00000)--cycle;
\fill[black] (0,2.10000)--(0,2.20000)--(0.100000,2.20000)--(0.100000,2.10000)--cycle;
\fill[black] (0,2.20000)--(0,2.30000)--(0.100000,2.30000)--(0.100000,2.20000)--cycle;
\fill[black] (0,2.50000)--(0,2.60000)--(0.100000,2.60000)--(0.100000,2.50000)--cycle;
\fill[black] (0,2.60000)--(0,2.70000)--(0.100000,2.70000)--(0.100000,2.60000)--cycle;
\fill[black] (0,3.20000)--(0,3.30000)--(0.100000,3.30000)--(0.100000,3.20000)--cycle;
\fill[black] (0,3.50000)--(0,3.60000)--(0.100000,3.60000)--(0.100000,3.50000)--cycle;
\fill[black] (0,3.70000)--(0,3.80000)--(0.100000,3.80000)--(0.100000,3.70000)--cycle;
\fill[black] (0,3.80000)--(0,3.90000)--(0.100000,3.90000)--(0.100000,3.80000)--cycle;
\fill[black] (0,4.00000)--(0,4.10000)--(0.100000,4.10000)--(0.100000,4.00000)--cycle;
\fill[black] (0,4.10000)--(0,4.20000)--(0.100000,4.20000)--(0.100000,4.10000)--cycle;
\fill[black] (0,4.30000)--(0,4.40000)--(0.100000,4.40000)--(0.100000,4.30000)--cycle;
\fill[black] (0,4.70000)--(0,4.80000)--(0.100000,4.80000)--(0.100000,4.70000)--cycle;
\fill[black] (0,4.80000)--(0,4.90000)--(0.100000,4.90000)--(0.100000,4.80000)--cycle;
\fill[black] (0.100000,0)--(0.100000,0.100000)--(0.200000,0.100000)--(0.200000,0)--cycle;
\fill[black] (0.100000,0.200000)--(0.100000,0.300000)--(0.200000,0.300000)--(0.200000,0.200000)--cycle;
\fill[black] (0.100000,0.500000)--(0.100000,0.600000)--(0.200000,0.600000)--(0.200000,0.500000)--cycle;
\fill[black] (0.100000,0.700000)--(0.100000,0.800000)--(0.200000,0.800000)--(0.200000,0.700000)--cycle;
\fill[black] (0.100000,0.800000)--(0.100000,0.900000)--(0.200000,0.900000)--(0.200000,0.800000)--cycle;
\fill[black] (0.100000,0.900000)--(0.100000,1.00000)--(0.200000,1.00000)--(0.200000,0.900000)--cycle;
\fill[black] (0.100000,1.00000)--(0.100000,1.10000)--(0.200000,1.10000)--(0.200000,1.00000)--cycle;
\fill[black] (0.100000,1.20000)--(0.100000,1.30000)--(0.200000,1.30000)--(0.200000,1.20000)--cycle;
\fill[black] (0.100000,1.30000)--(0.100000,1.40000)--(0.200000,1.40000)--(0.200000,1.30000)--cycle;
\fill[black] (0.100000,1.80000)--(0.100000,1.90000)--(0.200000,1.90000)--(0.200000,1.80000)--cycle;
\fill[black] (0.100000,2.30000)--(0.100000,2.40000)--(0.200000,2.40000)--(0.200000,2.30000)--cycle;
\fill[black] (0.100000,2.50000)--(0.100000,2.60000)--(0.200000,2.60000)--(0.200000,2.50000)--cycle;
\fill[black] (0.100000,3.50000)--(0.100000,3.60000)--(0.200000,3.60000)--(0.200000,3.50000)--cycle;
\fill[black] (0.200000,0.100000)--(0.200000,0.200000)--(0.300000,0.200000)--(0.300000,0.100000)--cycle;
\fill[black] (0.200000,0.300000)--(0.200000,0.400000)--(0.300000,0.400000)--(0.300000,0.300000)--cycle;
\fill[black] (0.200000,0.500000)--(0.200000,0.600000)--(0.300000,0.600000)--(0.300000,0.500000)--cycle;
\fill[black] (0.200000,2.60000)--(0.200000,2.70000)--(0.300000,2.70000)--(0.300000,2.60000)--cycle;
\fill[black] (0.200000,3.40000)--(0.200000,3.50000)--(0.300000,3.50000)--(0.300000,3.40000)--cycle;
\fill[black] (0.200000,4.20000)--(0.200000,4.30000)--(0.300000,4.30000)--(0.300000,4.20000)--cycle;
\fill[black] (0.200000,4.70000)--(0.200000,4.80000)--(0.300000,4.80000)--(0.300000,4.70000)--cycle;
\fill[black] (0.200000,4.80000)--(0.200000,4.90000)--(0.300000,4.90000)--(0.300000,4.80000)--cycle;
\fill[black] (0.300000,0.200000)--(0.300000,0.300000)--(0.400000,0.300000)--(0.400000,0.200000)--cycle;
\fill[black] (0.300000,0.400000)--(0.300000,0.500000)--(0.400000,0.500000)--(0.400000,0.400000)--cycle;
\fill[black] (0.300000,0.800000)--(0.300000,0.900000)--(0.400000,0.900000)--(0.400000,0.800000)--cycle;
\fill[black] (0.300000,1.40000)--(0.300000,1.50000)--(0.400000,1.50000)--(0.400000,1.40000)--cycle;
\fill[black] (0.300000,2.00000)--(0.300000,2.10000)--(0.400000,2.10000)--(0.400000,2.00000)--cycle;
\fill[black] (0.300000,2.10000)--(0.300000,2.20000)--(0.400000,2.20000)--(0.400000,2.10000)--cycle;
\fill[black] (0.300000,2.30000)--(0.300000,2.40000)--(0.400000,2.40000)--(0.400000,2.30000)--cycle;
\fill[black] (0.300000,3.10000)--(0.300000,3.20000)--(0.400000,3.20000)--(0.400000,3.10000)--cycle;
\fill[black] (0.300000,3.60000)--(0.300000,3.70000)--(0.400000,3.70000)--(0.400000,3.60000)--cycle;
\fill[black] (0.400000,0)--(0.400000,0.100000)--(0.500000,0.100000)--(0.500000,0)--cycle;
\fill[black] (0.400000,0.300000)--(0.400000,0.400000)--(0.500000,0.400000)--(0.500000,0.300000)--cycle;
\fill[black] (0.400000,1.10000)--(0.400000,1.20000)--(0.500000,1.20000)--(0.500000,1.10000)--cycle;
\fill[black] (0.400000,1.40000)--(0.400000,1.50000)--(0.500000,1.50000)--(0.500000,1.40000)--cycle;
\fill[black] (0.400000,2.70000)--(0.400000,2.80000)--(0.500000,2.80000)--(0.500000,2.70000)--cycle;
\fill[black] (0.400000,3.50000)--(0.400000,3.60000)--(0.500000,3.60000)--(0.500000,3.50000)--cycle;
\fill[black] (0.500000,0)--(0.500000,0.100000)--(0.600000,0.100000)--(0.600000,0)--cycle;
\fill[black] (0.500000,0.100000)--(0.500000,0.200000)--(0.600000,0.200000)--(0.600000,0.100000)--cycle;
\fill[black] (0.500000,0.200000)--(0.500000,0.300000)--(0.600000,0.300000)--(0.600000,0.200000)--cycle;
\fill[black] (0.500000,0.700000)--(0.500000,0.800000)--(0.600000,0.800000)--(0.600000,0.700000)--cycle;
\fill[black] (0.500000,0.900000)--(0.500000,1.00000)--(0.600000,1.00000)--(0.600000,0.900000)--cycle;
\fill[black] (0.500000,1.20000)--(0.500000,1.30000)--(0.600000,1.30000)--(0.600000,1.20000)--cycle;
\fill[black] (0.500000,2.00000)--(0.500000,2.10000)--(0.600000,2.10000)--(0.600000,2.00000)--cycle;
\fill[black] (0.500000,2.10000)--(0.500000,2.20000)--(0.600000,2.20000)--(0.600000,2.10000)--cycle;
\fill[black] (0.500000,3.60000)--(0.500000,3.70000)--(0.600000,3.70000)--(0.600000,3.60000)--cycle;
\fill[black] (0.600000,0.700000)--(0.600000,0.800000)--(0.700000,0.800000)--(0.700000,0.700000)--cycle;
\fill[black] (0.600000,0.900000)--(0.600000,1.00000)--(0.700000,1.00000)--(0.700000,0.900000)--cycle;
\fill[black] (0.600000,1.10000)--(0.600000,1.20000)--(0.700000,1.20000)--(0.700000,1.10000)--cycle;
\fill[black] (0.600000,3.00000)--(0.600000,3.10000)--(0.700000,3.10000)--(0.700000,3.00000)--cycle;
\fill[black] (0.600000,3.40000)--(0.600000,3.50000)--(0.700000,3.50000)--(0.700000,3.40000)--cycle;
\fill[black] (0.600000,4.80000)--(0.600000,4.90000)--(0.700000,4.90000)--(0.700000,4.80000)--cycle;
\fill[black] (0.700000,0.100000)--(0.700000,0.200000)--(0.800000,0.200000)--(0.800000,0.100000)--cycle;
\fill[black] (0.700000,0.500000)--(0.700000,0.600000)--(0.800000,0.600000)--(0.800000,0.500000)--cycle;
\fill[black] (0.700000,0.600000)--(0.700000,0.700000)--(0.800000,0.700000)--(0.800000,0.600000)--cycle;
\fill[black] (0.700000,1.30000)--(0.700000,1.40000)--(0.800000,1.40000)--(0.800000,1.30000)--cycle;
\fill[black] (0.700000,1.60000)--(0.700000,1.70000)--(0.800000,1.70000)--(0.800000,1.60000)--cycle;
\fill[black] (0.700000,2.00000)--(0.700000,2.10000)--(0.800000,2.10000)--(0.800000,2.00000)--cycle;
\fill[black] (0.700000,3.10000)--(0.700000,3.20000)--(0.800000,3.20000)--(0.800000,3.10000)--cycle;
\fill[black] (0.700000,4.00000)--(0.700000,4.10000)--(0.800000,4.10000)--(0.800000,4.00000)--cycle;
\fill[black] (0.700000,4.80000)--(0.700000,4.90000)--(0.800000,4.90000)--(0.800000,4.80000)--cycle;
\fill[black] (0.800000,0.100000)--(0.800000,0.200000)--(0.900000,0.200000)--(0.900000,0.100000)--cycle;
\fill[black] (0.800000,0.300000)--(0.800000,0.400000)--(0.900000,0.400000)--(0.900000,0.300000)--cycle;
\fill[black] (0.800000,2.30000)--(0.800000,2.40000)--(0.900000,2.40000)--(0.900000,2.30000)--cycle;
\fill[black] (0.800000,2.70000)--(0.800000,2.80000)--(0.900000,2.80000)--(0.900000,2.70000)--cycle;
\fill[black] (0.800000,4.10000)--(0.800000,4.20000)--(0.900000,4.20000)--(0.900000,4.10000)--cycle;
\fill[black] (0.800000,4.80000)--(0.800000,4.90000)--(0.900000,4.90000)--(0.900000,4.80000)--cycle;
\fill[black] (0.900000,0)--(0.900000,0.100000)--(1.00000,0.100000)--(1.00000,0)--cycle;
\fill[black] (0.900000,0.100000)--(0.900000,0.200000)--(1.00000,0.200000)--(1.00000,0.100000)--cycle;
\fill[black] (0.900000,0.500000)--(0.900000,0.600000)--(1.00000,0.600000)--(1.00000,0.500000)--cycle;
\fill[black] (0.900000,0.600000)--(0.900000,0.700000)--(1.00000,0.700000)--(1.00000,0.600000)--cycle;
\fill[black] (0.900000,1.70000)--(0.900000,1.80000)--(1.00000,1.80000)--(1.00000,1.70000)--cycle;
\fill[black] (0.900000,3.60000)--(0.900000,3.70000)--(1.00000,3.70000)--(1.00000,3.60000)--cycle;
\fill[black] (0.900000,4.10000)--(0.900000,4.20000)--(1.00000,4.20000)--(1.00000,4.10000)--cycle;
\fill[black] (1.00000,0)--(1.00000,0.100000)--(1.10000,0.100000)--(1.10000,0)--cycle;
\fill[black] (1.00000,0.100000)--(1.00000,0.200000)--(1.10000,0.200000)--(1.10000,0.100000)--cycle;
\fill[black] (1.00000,2.10000)--(1.00000,2.20000)--(1.10000,2.20000)--(1.10000,2.10000)--cycle;
\fill[black] (1.00000,3.40000)--(1.00000,3.50000)--(1.10000,3.50000)--(1.10000,3.40000)--cycle;
\fill[black] (1.10000,0.400000)--(1.10000,0.500000)--(1.20000,0.500000)--(1.20000,0.400000)--cycle;
\fill[black] (1.10000,0.600000)--(1.10000,0.700000)--(1.20000,0.700000)--(1.20000,0.600000)--cycle;
\fill[black] (1.20000,0.100000)--(1.20000,0.200000)--(1.30000,0.200000)--(1.30000,0.100000)--cycle;
\fill[black] (1.20000,0.500000)--(1.20000,0.600000)--(1.30000,0.600000)--(1.30000,0.500000)--cycle;
\fill[black] (1.20000,1.30000)--(1.20000,1.40000)--(1.30000,1.40000)--(1.30000,1.30000)--cycle;
\fill[black] (1.20000,4.90000)--(1.20000,5.00000)--(1.30000,5.00000)--(1.30000,4.90000)--cycle;
\fill[black] (1.30000,0.100000)--(1.30000,0.200000)--(1.40000,0.200000)--(1.40000,0.100000)--cycle;
\fill[black] (1.30000,0.700000)--(1.30000,0.800000)--(1.40000,0.800000)--(1.40000,0.700000)--cycle;
\fill[black] (1.30000,1.20000)--(1.30000,1.30000)--(1.40000,1.30000)--(1.40000,1.20000)--cycle;
\fill[black] (1.30000,3.30000)--(1.30000,3.40000)--(1.40000,3.40000)--(1.40000,3.30000)--cycle;
\fill[black] (1.30000,4.90000)--(1.30000,5.00000)--(1.40000,5.00000)--(1.40000,4.90000)--cycle;
\fill[black] (1.40000,0.300000)--(1.40000,0.400000)--(1.50000,0.400000)--(1.50000,0.300000)--cycle;
\fill[black] (1.40000,0.400000)--(1.40000,0.500000)--(1.50000,0.500000)--(1.50000,0.400000)--cycle;
\fill[black] (1.40000,3.80000)--(1.40000,3.90000)--(1.50000,3.90000)--(1.50000,3.80000)--cycle;
\fill[black] (1.40000,4.40000)--(1.40000,4.50000)--(1.50000,4.50000)--(1.50000,4.40000)--cycle;
\fill[black] (1.50000,0)--(1.50000,0.100000)--(1.60000,0.100000)--(1.60000,0)--cycle;
\fill[black] (1.50000,1.80000)--(1.50000,1.90000)--(1.60000,1.90000)--(1.60000,1.80000)--cycle;
\fill[black] (1.60000,0.700000)--(1.60000,0.800000)--(1.70000,0.800000)--(1.70000,0.700000)--cycle;
\fill[black] (1.60000,2.60000)--(1.60000,2.70000)--(1.70000,2.70000)--(1.70000,2.60000)--cycle;
\fill[black] (1.60000,2.70000)--(1.60000,2.80000)--(1.70000,2.80000)--(1.70000,2.70000)--cycle;
\fill[black] (1.60000,3.10000)--(1.60000,3.20000)--(1.70000,3.20000)--(1.70000,3.10000)--cycle;
\fill[black] (1.70000,0.900000)--(1.70000,1.00000)--(1.80000,1.00000)--(1.80000,0.900000)--cycle;
\fill[black] (1.70000,2.00000)--(1.70000,2.10000)--(1.80000,2.10000)--(1.80000,2.00000)--cycle;
\fill[black] (1.70000,2.70000)--(1.70000,2.80000)--(1.80000,2.80000)--(1.80000,2.70000)--cycle;
\fill[black] (1.70000,3.30000)--(1.70000,3.40000)--(1.80000,3.40000)--(1.80000,3.30000)--cycle;
\fill[black] (1.80000,0.100000)--(1.80000,0.200000)--(1.90000,0.200000)--(1.90000,0.100000)--cycle;
\fill[black] (1.80000,1.50000)--(1.80000,1.60000)--(1.90000,1.60000)--(1.90000,1.50000)--cycle;
\fill[black] (1.80000,3.90000)--(1.80000,4.00000)--(1.90000,4.00000)--(1.90000,3.90000)--cycle;
\fill[black] (1.90000,2.00000)--(1.90000,2.10000)--(2.00000,2.10000)--(2.00000,2.00000)--cycle;
\fill[black] (1.90000,2.80000)--(1.90000,2.90000)--(2.00000,2.90000)--(2.00000,2.80000)--cycle;
\fill[black] (1.90000,2.90000)--(1.90000,3.00000)--(2.00000,3.00000)--(2.00000,2.90000)--cycle;
\fill[black] (2.00000,0)--(2.00000,0.100000)--(2.10000,0.100000)--(2.10000,0)--cycle;
\fill[black] (2.00000,0.300000)--(2.00000,0.400000)--(2.10000,0.400000)--(2.10000,0.300000)--cycle;
\fill[black] (2.00000,0.500000)--(2.00000,0.600000)--(2.10000,0.600000)--(2.10000,0.500000)--cycle;
\fill[black] (2.00000,0.700000)--(2.00000,0.800000)--(2.10000,0.800000)--(2.10000,0.700000)--cycle;
\fill[black] (2.00000,1.70000)--(2.00000,1.80000)--(2.10000,1.80000)--(2.10000,1.70000)--cycle;
\fill[black] (2.00000,1.90000)--(2.00000,2.00000)--(2.10000,2.00000)--(2.10000,1.90000)--cycle;
\fill[black] (2.00000,3.70000)--(2.00000,3.80000)--(2.10000,3.80000)--(2.10000,3.70000)--cycle;
\fill[black] (2.00000,4.10000)--(2.00000,4.20000)--(2.10000,4.20000)--(2.10000,4.10000)--cycle;
\fill[black] (2.10000,0)--(2.10000,0.100000)--(2.20000,0.100000)--(2.20000,0)--cycle;
\fill[black] (2.10000,0.300000)--(2.10000,0.400000)--(2.20000,0.400000)--(2.20000,0.300000)--cycle;
\fill[black] (2.10000,0.500000)--(2.10000,0.600000)--(2.20000,0.600000)--(2.20000,0.500000)--cycle;
\fill[black] (2.10000,1.00000)--(2.10000,1.10000)--(2.20000,1.10000)--(2.20000,1.00000)--cycle;
\fill[black] (2.10000,2.30000)--(2.10000,2.40000)--(2.20000,2.40000)--(2.20000,2.30000)--cycle;
\fill[black] (2.10000,2.50000)--(2.10000,2.60000)--(2.20000,2.60000)--(2.20000,2.50000)--cycle;
\fill[black] (2.10000,3.60000)--(2.10000,3.70000)--(2.20000,3.70000)--(2.20000,3.60000)--cycle;
\fill[black] (2.10000,4.30000)--(2.10000,4.40000)--(2.20000,4.40000)--(2.20000,4.30000)--cycle;
\fill[black] (2.10000,4.80000)--(2.10000,4.90000)--(2.20000,4.90000)--(2.20000,4.80000)--cycle;
\fill[black] (2.20000,0)--(2.20000,0.100000)--(2.30000,0.100000)--(2.30000,0)--cycle;
\fill[black] (2.30000,0.100000)--(2.30000,0.200000)--(2.40000,0.200000)--(2.40000,0.100000)--cycle;
\fill[black] (2.30000,0.300000)--(2.30000,0.400000)--(2.40000,0.400000)--(2.40000,0.300000)--cycle;
\fill[black] (2.30000,0.800000)--(2.30000,0.900000)--(2.40000,0.900000)--(2.40000,0.800000)--cycle;
\fill[black] (2.30000,2.10000)--(2.30000,2.20000)--(2.40000,2.20000)--(2.40000,2.10000)--cycle;
\fill[black] (2.30000,2.90000)--(2.30000,3.00000)--(2.40000,3.00000)--(2.40000,2.90000)--cycle;
\fill[black] (2.30000,4.60000)--(2.30000,4.70000)--(2.40000,4.70000)--(2.40000,4.60000)--cycle;
\fill[black] (2.40000,3.40000)--(2.40000,3.50000)--(2.50000,3.50000)--(2.50000,3.40000)--cycle;
\fill[black] (2.40000,4.10000)--(2.40000,4.20000)--(2.50000,4.20000)--(2.50000,4.10000)--cycle;
\fill[black] (2.50000,0)--(2.50000,0.100000)--(2.60000,0.100000)--(2.60000,0)--cycle;
\fill[black] (2.50000,0.100000)--(2.50000,0.200000)--(2.60000,0.200000)--(2.60000,0.100000)--cycle;
\fill[black] (2.50000,2.10000)--(2.50000,2.20000)--(2.60000,2.20000)--(2.60000,2.10000)--cycle;
\fill[black] (2.50000,3.70000)--(2.50000,3.80000)--(2.60000,3.80000)--(2.60000,3.70000)--cycle;
\fill[black] (2.60000,0)--(2.60000,0.100000)--(2.70000,0.100000)--(2.70000,0)--cycle;
\fill[black] (2.60000,0.200000)--(2.60000,0.300000)--(2.70000,0.300000)--(2.70000,0.200000)--cycle;
\fill[black] (2.60000,1.60000)--(2.60000,1.70000)--(2.70000,1.70000)--(2.70000,1.60000)--cycle;
\fill[black] (2.70000,0.400000)--(2.70000,0.500000)--(2.80000,0.500000)--(2.80000,0.400000)--cycle;
\fill[black] (2.70000,0.800000)--(2.70000,0.900000)--(2.80000,0.900000)--(2.80000,0.800000)--cycle;
\fill[black] (2.70000,1.60000)--(2.70000,1.70000)--(2.80000,1.70000)--(2.80000,1.60000)--cycle;
\fill[black] (2.70000,1.70000)--(2.70000,1.80000)--(2.80000,1.80000)--(2.80000,1.70000)--cycle;
\fill[black] (2.80000,1.90000)--(2.80000,2.00000)--(2.90000,2.00000)--(2.90000,1.90000)--cycle;
\fill[black] (2.80000,3.70000)--(2.80000,3.80000)--(2.90000,3.80000)--(2.90000,3.70000)--cycle;
\fill[black] (2.90000,1.90000)--(2.90000,2.00000)--(3.00000,2.00000)--(3.00000,1.90000)--cycle;
\fill[black] (2.90000,2.30000)--(2.90000,2.40000)--(3.00000,2.40000)--(3.00000,2.30000)--cycle;
\fill[black] (2.90000,4.30000)--(2.90000,4.40000)--(3.00000,4.40000)--(3.00000,4.30000)--cycle;
\fill[black] (3.00000,0.600000)--(3.00000,0.700000)--(3.10000,0.700000)--(3.10000,0.600000)--cycle;
\fill[black] (3.10000,0.300000)--(3.10000,0.400000)--(3.20000,0.400000)--(3.20000,0.300000)--cycle;
\fill[black] (3.10000,0.700000)--(3.10000,0.800000)--(3.20000,0.800000)--(3.20000,0.700000)--cycle;
\fill[black] (3.10000,1.60000)--(3.10000,1.70000)--(3.20000,1.70000)--(3.20000,1.60000)--cycle;
\fill[black] (3.20000,0)--(3.20000,0.100000)--(3.30000,0.100000)--(3.30000,0)--cycle;
\fill[black] (3.30000,1.30000)--(3.30000,1.40000)--(3.40000,1.40000)--(3.40000,1.30000)--cycle;
\fill[black] (3.30000,1.70000)--(3.30000,1.80000)--(3.40000,1.80000)--(3.40000,1.70000)--cycle;
\fill[black] (3.30000,3.80000)--(3.30000,3.90000)--(3.40000,3.90000)--(3.40000,3.80000)--cycle;
\fill[black] (3.30000,4.50000)--(3.30000,4.60000)--(3.40000,4.60000)--(3.40000,4.50000)--cycle;
\fill[black] (3.40000,0.200000)--(3.40000,0.300000)--(3.50000,0.300000)--(3.50000,0.200000)--cycle;
\fill[black] (3.40000,0.600000)--(3.40000,0.700000)--(3.50000,0.700000)--(3.50000,0.600000)--cycle;
\fill[black] (3.40000,1.00000)--(3.40000,1.10000)--(3.50000,1.10000)--(3.50000,1.00000)--cycle;
\fill[black] (3.40000,2.40000)--(3.40000,2.50000)--(3.50000,2.50000)--(3.50000,2.40000)--cycle;
\fill[black] (3.50000,0)--(3.50000,0.100000)--(3.60000,0.100000)--(3.60000,0)--cycle;
\fill[black] (3.50000,0.100000)--(3.50000,0.200000)--(3.60000,0.200000)--(3.60000,0.100000)--cycle;
\fill[black] (3.50000,0.400000)--(3.50000,0.500000)--(3.60000,0.500000)--(3.60000,0.400000)--cycle;
\fill[black] (3.50000,3.70000)--(3.50000,3.80000)--(3.60000,3.80000)--(3.60000,3.70000)--cycle;
\fill[black] (3.50000,4.10000)--(3.50000,4.20000)--(3.60000,4.20000)--(3.60000,4.10000)--cycle;
\fill[black] (3.60000,0.300000)--(3.60000,0.400000)--(3.70000,0.400000)--(3.70000,0.300000)--cycle;
\fill[black] (3.60000,0.500000)--(3.60000,0.600000)--(3.70000,0.600000)--(3.70000,0.500000)--cycle;
\fill[black] (3.60000,0.900000)--(3.60000,1.00000)--(3.70000,1.00000)--(3.70000,0.900000)--cycle;
\fill[black] (3.60000,2.10000)--(3.60000,2.20000)--(3.70000,2.20000)--(3.70000,2.10000)--cycle;
\fill[black] (3.60000,3.90000)--(3.60000,4.00000)--(3.70000,4.00000)--(3.70000,3.90000)--cycle;
\fill[black] (3.70000,0)--(3.70000,0.100000)--(3.80000,0.100000)--(3.80000,0)--cycle;
\fill[black] (3.70000,2.00000)--(3.70000,2.10000)--(3.80000,2.10000)--(3.80000,2.00000)--cycle;
\fill[black] (3.70000,2.50000)--(3.70000,2.60000)--(3.80000,2.60000)--(3.80000,2.50000)--cycle;
\fill[black] (3.70000,2.80000)--(3.70000,2.90000)--(3.80000,2.90000)--(3.80000,2.80000)--cycle;
\fill[black] (3.70000,3.50000)--(3.70000,3.60000)--(3.80000,3.60000)--(3.80000,3.50000)--cycle;
\fill[black] (3.80000,0)--(3.80000,0.100000)--(3.90000,0.100000)--(3.90000,0)--cycle;
\fill[black] (3.80000,1.40000)--(3.80000,1.50000)--(3.90000,1.50000)--(3.90000,1.40000)--cycle;
\fill[black] (3.80000,3.30000)--(3.80000,3.40000)--(3.90000,3.40000)--(3.90000,3.30000)--cycle;
\fill[black] (3.90000,1.80000)--(3.90000,1.90000)--(4.00000,1.90000)--(4.00000,1.80000)--cycle;
\fill[black] (3.90000,3.60000)--(3.90000,3.70000)--(4.00000,3.70000)--(4.00000,3.60000)--cycle;
\fill[black] (3.90000,4.50000)--(3.90000,4.60000)--(4.00000,4.60000)--(4.00000,4.50000)--cycle;
\fill[black] (4.00000,0)--(4.00000,0.100000)--(4.10000,0.100000)--(4.10000,0)--cycle;
\fill[black] (4.00000,0.700000)--(4.00000,0.800000)--(4.10000,0.800000)--(4.10000,0.700000)--cycle;
\fill[black] (4.10000,0)--(4.10000,0.100000)--(4.20000,0.100000)--(4.20000,0)--cycle;
\fill[black] (4.10000,0.800000)--(4.10000,0.900000)--(4.20000,0.900000)--(4.20000,0.800000)--cycle;
\fill[black] (4.10000,0.900000)--(4.10000,1.00000)--(4.20000,1.00000)--(4.20000,0.900000)--cycle;
\fill[black] (4.10000,2.00000)--(4.10000,2.10000)--(4.20000,2.10000)--(4.20000,2.00000)--cycle;
\fill[black] (4.10000,2.40000)--(4.10000,2.50000)--(4.20000,2.50000)--(4.20000,2.40000)--cycle;
\fill[black] (4.10000,3.50000)--(4.10000,3.60000)--(4.20000,3.60000)--(4.20000,3.50000)--cycle;
\fill[black] (4.20000,0.200000)--(4.20000,0.300000)--(4.30000,0.300000)--(4.30000,0.200000)--cycle;
\fill[black] (4.30000,0)--(4.30000,0.100000)--(4.40000,0.100000)--(4.40000,0)--cycle;
\fill[black] (4.30000,2.10000)--(4.30000,2.20000)--(4.40000,2.20000)--(4.40000,2.10000)--cycle;
\fill[black] (4.30000,2.90000)--(4.30000,3.00000)--(4.40000,3.00000)--(4.40000,2.90000)--cycle;
\fill[black] (4.40000,1.40000)--(4.40000,1.50000)--(4.50000,1.50000)--(4.50000,1.40000)--cycle;
\fill[black] (4.50000,3.30000)--(4.50000,3.40000)--(4.60000,3.40000)--(4.60000,3.30000)--cycle;
\fill[black] (4.50000,3.90000)--(4.50000,4.00000)--(4.60000,4.00000)--(4.60000,3.90000)--cycle;
\fill[black] (4.60000,2.30000)--(4.60000,2.40000)--(4.70000,2.40000)--(4.70000,2.30000)--cycle;
\fill[black] (4.70000,0)--(4.70000,0.100000)--(4.80000,0.100000)--(4.80000,0)--cycle;
\fill[black] (4.70000,0.200000)--(4.70000,0.300000)--(4.80000,0.300000)--(4.80000,0.200000)--cycle;
\fill[black] (4.80000,0)--(4.80000,0.100000)--(4.90000,0.100000)--(4.90000,0)--cycle;
\fill[black] (4.80000,0.200000)--(4.80000,0.300000)--(4.90000,0.300000)--(4.90000,0.200000)--cycle;
\fill[black] (4.80000,0.600000)--(4.80000,0.700000)--(4.90000,0.700000)--(4.90000,0.600000)--cycle;
\fill[black] (4.80000,0.700000)--(4.80000,0.800000)--(4.90000,0.800000)--(4.90000,0.700000)--cycle;
\fill[black] (4.80000,0.800000)--(4.80000,0.900000)--(4.90000,0.900000)--(4.90000,0.800000)--cycle;
\fill[black] (4.80000,2.10000)--(4.80000,2.20000)--(4.90000,2.20000)--(4.90000,2.10000)--cycle;
\fill[black] (4.90000,1.20000)--(4.90000,1.30000)--(5.00000,1.30000)--(5.00000,1.20000)--cycle;
\fill[black] (4.90000,1.30000)--(4.90000,1.40000)--(5.00000,1.40000)--(5.00000,1.30000)--cycle;
\draw (2.50000,5.50000) node {rescaled graphon on probability space};
\draw (7,0)--(12.0000,0);
\draw (12.0000,0)--(12.0000,5.00000);
\draw (12.0000,5.00000)--(7,5.00000);
\draw (7,5.00000)--(7,0);
\fill[black] (7,0.188679)--(7,0.283019)--(7.09434,0.283019)--(7.09434,0.188679)--cycle;
\fill[black] (7,0.283019)--(7,0.377358)--(7.09434,0.377358)--(7.09434,0.283019)--cycle;
\fill[black] (7,0.377358)--(7,0.471698)--(7.09434,0.471698)--(7.09434,0.377358)--cycle;
\fill[black] (7,0.471698)--(7,0.566038)--(7.09434,0.566038)--(7.09434,0.471698)--cycle;
\fill[black] (7,0.566038)--(7,0.660377)--(7.09434,0.660377)--(7.09434,0.566038)--cycle;
\fill[black] (7,0.660377)--(7,0.754717)--(7.09434,0.754717)--(7.09434,0.660377)--cycle;
\fill[black] (7,0.754717)--(7,0.849057)--(7.09434,0.849057)--(7.09434,0.754717)--cycle;
\fill[black] (7,0.849057)--(7,0.943396)--(7.09434,0.943396)--(7.09434,0.849057)--cycle;
\fill[black] (7,0.943396)--(7,1.03774)--(7.09434,1.03774)--(7.09434,0.943396)--cycle;
\fill[black] (7,1.03774)--(7,1.13208)--(7.09434,1.13208)--(7.09434,1.03774)--cycle;
\fill[black] (7,1.13208)--(7,1.22642)--(7.09434,1.22642)--(7.09434,1.13208)--cycle;
\fill[black] (7,1.22642)--(7,1.32075)--(7.09434,1.32075)--(7.09434,1.22642)--cycle;
\fill[black] (7,1.32075)--(7,1.41509)--(7.09434,1.41509)--(7.09434,1.32075)--cycle;
\fill[black] (7,2.16981)--(7,2.26415)--(7.09434,2.26415)--(7.09434,2.16981)--cycle;
\fill[black] (7,2.64151)--(7,2.73585)--(7.09434,2.73585)--(7.09434,2.64151)--cycle;
\fill[black] (7,3.01887)--(7,3.11321)--(7.09434,3.11321)--(7.09434,3.01887)--cycle;
\fill[black] (7,4.43396)--(7,4.52830)--(7.09434,4.52830)--(7.09434,4.43396)--cycle;
\fill[black] (7,4.52830)--(7,4.62264)--(7.09434,4.62264)--(7.09434,4.52830)--cycle;
\fill[black] (7.09434,0.188679)--(7.09434,0.283019)--(7.18868,0.283019)--(7.18868,0.188679)--cycle;
\fill[black] (7.09434,0.283019)--(7.09434,0.377358)--(7.18868,0.377358)--(7.18868,0.283019)--cycle;
\fill[black] (7.09434,0.377358)--(7.09434,0.471698)--(7.18868,0.471698)--(7.18868,0.377358)--cycle;
\fill[black] (7.09434,0.660377)--(7.09434,0.754717)--(7.18868,0.754717)--(7.18868,0.660377)--cycle;
\fill[black] (7.09434,0.754717)--(7.09434,0.849057)--(7.18868,0.849057)--(7.18868,0.754717)--cycle;
\fill[black] (7.09434,0.849057)--(7.09434,0.943396)--(7.18868,0.943396)--(7.18868,0.849057)--cycle;
\fill[black] (7.09434,0.943396)--(7.09434,1.03774)--(7.18868,1.03774)--(7.18868,0.943396)--cycle;
\fill[black] (7.09434,1.13208)--(7.09434,1.22642)--(7.18868,1.22642)--(7.18868,1.13208)--cycle;
\fill[black] (7.09434,1.32075)--(7.09434,1.41509)--(7.18868,1.41509)--(7.18868,1.32075)--cycle;
\fill[black] (7.09434,1.41509)--(7.09434,1.50943)--(7.18868,1.50943)--(7.18868,1.41509)--cycle;
\fill[black] (7.09434,1.69811)--(7.09434,1.79245)--(7.18868,1.79245)--(7.18868,1.69811)--cycle;
\fill[black] (7.09434,1.88679)--(7.09434,1.98113)--(7.18868,1.98113)--(7.18868,1.88679)--cycle;
\fill[black] (7.09434,2.26415)--(7.09434,2.35849)--(7.18868,2.35849)--(7.18868,2.26415)--cycle;
\fill[black] (7.09434,2.92453)--(7.09434,3.01887)--(7.18868,3.01887)--(7.18868,2.92453)--cycle;
\fill[black] (7.09434,4.15094)--(7.09434,4.24528)--(7.18868,4.24528)--(7.18868,4.15094)--cycle;
\fill[black] (7.09434,4.33962)--(7.09434,4.43396)--(7.18868,4.43396)--(7.18868,4.33962)--cycle;
\fill[black] (7.09434,4.62264)--(7.09434,4.71698)--(7.18868,4.71698)--(7.18868,4.62264)--cycle;
\fill[black] (7.09434,4.71698)--(7.09434,4.81132)--(7.18868,4.81132)--(7.18868,4.71698)--cycle;
\fill[black] (7.18868,0)--(7.18868,0.0943396)--(7.28302,0.0943396)--(7.28302,0)--cycle;
\fill[black] (7.18868,0.0943396)--(7.18868,0.188679)--(7.28302,0.188679)--(7.28302,0.0943396)--cycle;
\fill[black] (7.18868,0.283019)--(7.18868,0.377358)--(7.28302,0.377358)--(7.28302,0.283019)--cycle;
\fill[black] (7.18868,0.471698)--(7.18868,0.566038)--(7.28302,0.566038)--(7.28302,0.471698)--cycle;
\fill[black] (7.18868,0.566038)--(7.18868,0.660377)--(7.28302,0.660377)--(7.28302,0.566038)--cycle;
\fill[black] (7.18868,0.660377)--(7.18868,0.754717)--(7.28302,0.754717)--(7.28302,0.660377)--cycle;
\fill[black] (7.18868,0.943396)--(7.18868,1.03774)--(7.28302,1.03774)--(7.28302,0.943396)--cycle;
\fill[black] (7.18868,1.03774)--(7.18868,1.13208)--(7.28302,1.13208)--(7.28302,1.03774)--cycle;
\fill[black] (7.18868,1.32075)--(7.18868,1.41509)--(7.28302,1.41509)--(7.28302,1.32075)--cycle;
\fill[black] (7.18868,1.41509)--(7.18868,1.50943)--(7.28302,1.50943)--(7.28302,1.41509)--cycle;
\fill[black] (7.18868,1.60377)--(7.18868,1.69811)--(7.28302,1.69811)--(7.28302,1.60377)--cycle;
\fill[black] (7.18868,1.69811)--(7.18868,1.79245)--(7.28302,1.79245)--(7.28302,1.69811)--cycle;
\fill[black] (7.18868,2.07547)--(7.18868,2.16981)--(7.28302,2.16981)--(7.28302,2.07547)--cycle;
\fill[black] (7.18868,2.54717)--(7.18868,2.64151)--(7.28302,2.64151)--(7.28302,2.54717)--cycle;
\fill[black] (7.18868,2.73585)--(7.18868,2.83019)--(7.28302,2.83019)--(7.28302,2.73585)--cycle;
\fill[black] (7.18868,3.77358)--(7.18868,3.86792)--(7.28302,3.86792)--(7.28302,3.77358)--cycle;
\fill[black] (7.18868,3.96226)--(7.18868,4.05660)--(7.28302,4.05660)--(7.28302,3.96226)--cycle;
\fill[black] (7.18868,4.05660)--(7.18868,4.15094)--(7.28302,4.15094)--(7.28302,4.05660)--cycle;
\fill[black] (7.28302,0)--(7.28302,0.0943396)--(7.37736,0.0943396)--(7.37736,0)--cycle;
\fill[black] (7.28302,0.0943396)--(7.28302,0.188679)--(7.37736,0.188679)--(7.37736,0.0943396)--cycle;
\fill[black] (7.28302,0.188679)--(7.28302,0.283019)--(7.37736,0.283019)--(7.37736,0.188679)--cycle;
\fill[black] (7.28302,0.377358)--(7.28302,0.471698)--(7.37736,0.471698)--(7.37736,0.377358)--cycle;
\fill[black] (7.28302,0.471698)--(7.28302,0.566038)--(7.37736,0.566038)--(7.37736,0.471698)--cycle;
\fill[black] (7.28302,0.566038)--(7.28302,0.660377)--(7.37736,0.660377)--(7.37736,0.566038)--cycle;
\fill[black] (7.28302,0.660377)--(7.28302,0.754717)--(7.37736,0.754717)--(7.37736,0.660377)--cycle;
\fill[black] (7.28302,0.754717)--(7.28302,0.849057)--(7.37736,0.849057)--(7.37736,0.754717)--cycle;
\fill[black] (7.28302,0.849057)--(7.28302,0.943396)--(7.37736,0.943396)--(7.37736,0.849057)--cycle;
\fill[black] (7.28302,0.943396)--(7.28302,1.03774)--(7.37736,1.03774)--(7.37736,0.943396)--cycle;
\fill[black] (7.28302,1.03774)--(7.28302,1.13208)--(7.37736,1.13208)--(7.37736,1.03774)--cycle;
\fill[black] (7.28302,1.22642)--(7.28302,1.32075)--(7.37736,1.32075)--(7.37736,1.22642)--cycle;
\fill[black] (7.28302,1.32075)--(7.28302,1.41509)--(7.37736,1.41509)--(7.37736,1.32075)--cycle;
\fill[black] (7.28302,2.07547)--(7.28302,2.16981)--(7.37736,2.16981)--(7.37736,2.07547)--cycle;
\fill[black] (7.28302,2.16981)--(7.28302,2.26415)--(7.37736,2.26415)--(7.37736,2.16981)--cycle;
\fill[black] (7.28302,2.54717)--(7.28302,2.64151)--(7.37736,2.64151)--(7.37736,2.54717)--cycle;
\fill[black] (7.28302,3.11321)--(7.28302,3.20755)--(7.37736,3.20755)--(7.37736,3.11321)--cycle;
\fill[black] (7.28302,3.39623)--(7.28302,3.49057)--(7.37736,3.49057)--(7.37736,3.39623)--cycle;
\fill[black] (7.28302,4.24528)--(7.28302,4.33962)--(7.37736,4.33962)--(7.37736,4.24528)--cycle;
\fill[black] (7.28302,4.90566)--(7.28302,5.00000)--(7.37736,5.00000)--(7.37736,4.90566)--cycle;
\fill[black] (7.37736,0)--(7.37736,0.0943396)--(7.47170,0.0943396)--(7.47170,0)--cycle;
\fill[black] (7.37736,0.0943396)--(7.37736,0.188679)--(7.47170,0.188679)--(7.47170,0.0943396)--cycle;
\fill[black] (7.37736,0.283019)--(7.37736,0.377358)--(7.47170,0.377358)--(7.47170,0.283019)--cycle;
\fill[black] (7.37736,0.471698)--(7.37736,0.566038)--(7.47170,0.566038)--(7.47170,0.471698)--cycle;
\fill[black] (7.37736,0.566038)--(7.37736,0.660377)--(7.47170,0.660377)--(7.47170,0.566038)--cycle;
\fill[black] (7.37736,0.660377)--(7.37736,0.754717)--(7.47170,0.754717)--(7.47170,0.660377)--cycle;
\fill[black] (7.37736,0.754717)--(7.37736,0.849057)--(7.47170,0.849057)--(7.47170,0.754717)--cycle;
\fill[black] (7.37736,0.849057)--(7.37736,0.943396)--(7.47170,0.943396)--(7.47170,0.849057)--cycle;
\fill[black] (7.37736,0.943396)--(7.37736,1.03774)--(7.47170,1.03774)--(7.47170,0.943396)--cycle;
\fill[black] (7.37736,1.32075)--(7.37736,1.41509)--(7.47170,1.41509)--(7.47170,1.32075)--cycle;
\fill[black] (7.37736,1.41509)--(7.37736,1.50943)--(7.47170,1.50943)--(7.47170,1.41509)--cycle;
\fill[black] (7.37736,1.50943)--(7.37736,1.60377)--(7.47170,1.60377)--(7.47170,1.50943)--cycle;
\fill[black] (7.37736,1.60377)--(7.37736,1.69811)--(7.47170,1.69811)--(7.47170,1.60377)--cycle;
\fill[black] (7.37736,1.88679)--(7.37736,1.98113)--(7.47170,1.98113)--(7.47170,1.88679)--cycle;
\fill[black] (7.37736,1.98113)--(7.37736,2.07547)--(7.47170,2.07547)--(7.47170,1.98113)--cycle;
\fill[black] (7.37736,2.07547)--(7.37736,2.16981)--(7.47170,2.16981)--(7.47170,2.07547)--cycle;
\fill[black] (7.37736,2.45283)--(7.37736,2.54717)--(7.47170,2.54717)--(7.47170,2.45283)--cycle;
\fill[black] (7.37736,2.83019)--(7.37736,2.92453)--(7.47170,2.92453)--(7.47170,2.83019)--cycle;
\fill[black] (7.37736,3.30189)--(7.37736,3.39623)--(7.47170,3.39623)--(7.47170,3.30189)--cycle;
\fill[black] (7.37736,4.81132)--(7.37736,4.90566)--(7.47170,4.90566)--(7.47170,4.81132)--cycle;
\fill[black] (7.47170,0)--(7.47170,0.0943396)--(7.56604,0.0943396)--(7.56604,0)--cycle;
\fill[black] (7.47170,0.188679)--(7.47170,0.283019)--(7.56604,0.283019)--(7.56604,0.188679)--cycle;
\fill[black] (7.47170,0.283019)--(7.47170,0.377358)--(7.56604,0.377358)--(7.56604,0.283019)--cycle;
\fill[black] (7.47170,0.377358)--(7.47170,0.471698)--(7.56604,0.471698)--(7.56604,0.377358)--cycle;
\fill[black] (7.47170,0.660377)--(7.47170,0.754717)--(7.56604,0.754717)--(7.56604,0.660377)--cycle;
\fill[black] (7.47170,3.01887)--(7.47170,3.11321)--(7.56604,3.11321)--(7.56604,3.01887)--cycle;
\fill[black] (7.56604,0)--(7.56604,0.0943396)--(7.66038,0.0943396)--(7.66038,0)--cycle;
\fill[black] (7.56604,0.188679)--(7.56604,0.283019)--(7.66038,0.283019)--(7.66038,0.188679)--cycle;
\fill[black] (7.56604,0.283019)--(7.56604,0.377358)--(7.66038,0.377358)--(7.66038,0.283019)--cycle;
\fill[black] (7.56604,0.377358)--(7.56604,0.471698)--(7.66038,0.471698)--(7.66038,0.377358)--cycle;
\fill[black] (7.56604,0.849057)--(7.56604,0.943396)--(7.66038,0.943396)--(7.66038,0.849057)--cycle;
\fill[black] (7.56604,1.79245)--(7.56604,1.88679)--(7.66038,1.88679)--(7.66038,1.79245)--cycle;
\fill[black] (7.56604,3.11321)--(7.56604,3.20755)--(7.66038,3.20755)--(7.66038,3.11321)--cycle;
\fill[black] (7.56604,3.20755)--(7.56604,3.30189)--(7.66038,3.30189)--(7.66038,3.20755)--cycle;
\fill[black] (7.56604,3.86792)--(7.56604,3.96226)--(7.66038,3.96226)--(7.66038,3.86792)--cycle;
\fill[black] (7.66038,0)--(7.66038,0.0943396)--(7.75472,0.0943396)--(7.75472,0)--cycle;
\fill[black] (7.66038,0.0943396)--(7.66038,0.188679)--(7.75472,0.188679)--(7.75472,0.0943396)--cycle;
\fill[black] (7.66038,0.188679)--(7.66038,0.283019)--(7.75472,0.283019)--(7.75472,0.188679)--cycle;
\fill[black] (7.66038,0.283019)--(7.66038,0.377358)--(7.75472,0.377358)--(7.75472,0.283019)--cycle;
\fill[black] (7.66038,0.377358)--(7.66038,0.471698)--(7.75472,0.471698)--(7.75472,0.377358)--cycle;
\fill[black] (7.66038,0.471698)--(7.66038,0.566038)--(7.75472,0.566038)--(7.75472,0.471698)--cycle;
\fill[black] (7.66038,1.69811)--(7.66038,1.79245)--(7.75472,1.79245)--(7.75472,1.69811)--cycle;
\fill[black] (7.75472,0)--(7.75472,0.0943396)--(7.84906,0.0943396)--(7.84906,0)--cycle;
\fill[black] (7.75472,0.0943396)--(7.75472,0.188679)--(7.84906,0.188679)--(7.84906,0.0943396)--cycle;
\fill[black] (7.75472,0.283019)--(7.75472,0.377358)--(7.84906,0.377358)--(7.84906,0.283019)--cycle;
\fill[black] (7.75472,0.377358)--(7.75472,0.471698)--(7.84906,0.471698)--(7.84906,0.377358)--cycle;
\fill[black] (7.75472,1.88679)--(7.75472,1.98113)--(7.84906,1.98113)--(7.84906,1.88679)--cycle;
\fill[black] (7.75472,2.45283)--(7.75472,2.54717)--(7.84906,2.54717)--(7.84906,2.45283)--cycle;
\fill[black] (7.75472,2.54717)--(7.75472,2.64151)--(7.84906,2.64151)--(7.84906,2.54717)--cycle;
\fill[black] (7.84906,0)--(7.84906,0.0943396)--(7.94340,0.0943396)--(7.94340,0)--cycle;
\fill[black] (7.84906,0.0943396)--(7.84906,0.188679)--(7.94340,0.188679)--(7.94340,0.0943396)--cycle;
\fill[black] (7.84906,0.283019)--(7.84906,0.377358)--(7.94340,0.377358)--(7.94340,0.283019)--cycle;
\fill[black] (7.84906,0.377358)--(7.84906,0.471698)--(7.94340,0.471698)--(7.94340,0.377358)--cycle;
\fill[black] (7.84906,0.566038)--(7.84906,0.660377)--(7.94340,0.660377)--(7.94340,0.566038)--cycle;
\fill[black] (7.84906,1.03774)--(7.84906,1.13208)--(7.94340,1.13208)--(7.94340,1.03774)--cycle;
\fill[black] (7.94340,0)--(7.94340,0.0943396)--(8.03774,0.0943396)--(8.03774,0)--cycle;
\fill[black] (7.94340,0.0943396)--(7.94340,0.188679)--(8.03774,0.188679)--(8.03774,0.0943396)--cycle;
\fill[black] (7.94340,0.188679)--(7.94340,0.283019)--(8.03774,0.283019)--(8.03774,0.188679)--cycle;
\fill[black] (7.94340,0.283019)--(7.94340,0.377358)--(8.03774,0.377358)--(8.03774,0.283019)--cycle;
\fill[black] (7.94340,0.377358)--(7.94340,0.471698)--(8.03774,0.471698)--(8.03774,0.377358)--cycle;
\fill[black] (7.94340,3.67925)--(7.94340,3.77358)--(8.03774,3.77358)--(8.03774,3.67925)--cycle;
\fill[black] (8.03774,0)--(8.03774,0.0943396)--(8.13208,0.0943396)--(8.13208,0)--cycle;
\fill[black] (8.03774,0.188679)--(8.03774,0.283019)--(8.13208,0.283019)--(8.13208,0.188679)--cycle;
\fill[black] (8.03774,0.283019)--(8.03774,0.377358)--(8.13208,0.377358)--(8.13208,0.283019)--cycle;
\fill[black] (8.03774,0.849057)--(8.03774,0.943396)--(8.13208,0.943396)--(8.13208,0.849057)--cycle;
\fill[black] (8.13208,0)--(8.13208,0.0943396)--(8.22642,0.0943396)--(8.22642,0)--cycle;
\fill[black] (8.13208,0.0943396)--(8.13208,0.188679)--(8.22642,0.188679)--(8.22642,0.0943396)--cycle;
\fill[black] (8.13208,2.64151)--(8.13208,2.73585)--(8.22642,2.73585)--(8.22642,2.64151)--cycle;
\fill[black] (8.22642,0)--(8.22642,0.0943396)--(8.32075,0.0943396)--(8.32075,0)--cycle;
\fill[black] (8.22642,0.283019)--(8.22642,0.377358)--(8.32075,0.377358)--(8.32075,0.283019)--cycle;
\fill[black] (8.32075,0)--(8.32075,0.0943396)--(8.41509,0.0943396)--(8.41509,0)--cycle;
\fill[black] (8.32075,0.0943396)--(8.32075,0.188679)--(8.41509,0.188679)--(8.41509,0.0943396)--cycle;
\fill[black] (8.32075,0.188679)--(8.32075,0.283019)--(8.41509,0.283019)--(8.41509,0.188679)--cycle;
\fill[black] (8.32075,0.283019)--(8.32075,0.377358)--(8.41509,0.377358)--(8.41509,0.283019)--cycle;
\fill[black] (8.32075,0.377358)--(8.32075,0.471698)--(8.41509,0.471698)--(8.41509,0.377358)--cycle;
\fill[black] (8.32075,2.35849)--(8.32075,2.45283)--(8.41509,2.45283)--(8.41509,2.35849)--cycle;
\fill[black] (8.32075,2.45283)--(8.32075,2.54717)--(8.41509,2.54717)--(8.41509,2.45283)--cycle;
\fill[black] (8.41509,0.0943396)--(8.41509,0.188679)--(8.50943,0.188679)--(8.50943,0.0943396)--cycle;
\fill[black] (8.41509,0.188679)--(8.41509,0.283019)--(8.50943,0.283019)--(8.50943,0.188679)--cycle;
\fill[black] (8.41509,0.377358)--(8.41509,0.471698)--(8.50943,0.471698)--(8.50943,0.377358)--cycle;
\fill[black] (8.41509,3.58491)--(8.41509,3.67925)--(8.50943,3.67925)--(8.50943,3.58491)--cycle;
\fill[black] (8.50943,0.377358)--(8.50943,0.471698)--(8.60377,0.471698)--(8.60377,0.377358)--cycle;
\fill[black] (8.60377,0.188679)--(8.60377,0.283019)--(8.69811,0.283019)--(8.69811,0.188679)--cycle;
\fill[black] (8.60377,0.377358)--(8.60377,0.471698)--(8.69811,0.471698)--(8.69811,0.377358)--cycle;
\fill[black] (8.69811,0.0943396)--(8.69811,0.188679)--(8.79245,0.188679)--(8.79245,0.0943396)--cycle;
\fill[black] (8.69811,0.188679)--(8.69811,0.283019)--(8.79245,0.283019)--(8.79245,0.188679)--cycle;
\fill[black] (8.69811,0.660377)--(8.69811,0.754717)--(8.79245,0.754717)--(8.79245,0.660377)--cycle;
\fill[black] (8.69811,3.49057)--(8.69811,3.58491)--(8.79245,3.58491)--(8.79245,3.49057)--cycle;
\fill[black] (8.79245,0.566038)--(8.79245,0.660377)--(8.88679,0.660377)--(8.88679,0.566038)--cycle;
\fill[black] (8.88679,0.0943396)--(8.88679,0.188679)--(8.98113,0.188679)--(8.98113,0.0943396)--cycle;
\fill[black] (8.88679,0.377358)--(8.88679,0.471698)--(8.98113,0.471698)--(8.98113,0.377358)--cycle;
\fill[black] (8.88679,0.754717)--(8.88679,0.849057)--(8.98113,0.849057)--(8.98113,0.754717)--cycle;
\fill[black] (8.98113,0.377358)--(8.98113,0.471698)--(9.07547,0.471698)--(9.07547,0.377358)--cycle;
\fill[black] (8.98113,3.20755)--(8.98113,3.30189)--(9.07547,3.30189)--(9.07547,3.20755)--cycle;
\fill[black] (9.07547,0.188679)--(9.07547,0.283019)--(9.16981,0.283019)--(9.16981,0.188679)--cycle;
\fill[black] (9.07547,0.283019)--(9.07547,0.377358)--(9.16981,0.377358)--(9.16981,0.283019)--cycle;
\fill[black] (9.07547,0.377358)--(9.07547,0.471698)--(9.16981,0.471698)--(9.16981,0.377358)--cycle;
\fill[black] (9.16981,0)--(9.16981,0.0943396)--(9.26415,0.0943396)--(9.26415,0)--cycle;
\fill[black] (9.16981,0.283019)--(9.16981,0.377358)--(9.26415,0.377358)--(9.26415,0.283019)--cycle;
\fill[black] (9.26415,0.0943396)--(9.26415,0.188679)--(9.35849,0.188679)--(9.35849,0.0943396)--cycle;
\fill[black] (9.35849,1.32075)--(9.35849,1.41509)--(9.45283,1.41509)--(9.45283,1.32075)--cycle;
\fill[black] (9.45283,0.377358)--(9.45283,0.471698)--(9.54717,0.471698)--(9.54717,0.377358)--cycle;
\fill[black] (9.45283,0.754717)--(9.45283,0.849057)--(9.54717,0.849057)--(9.54717,0.754717)--cycle;
\fill[black] (9.45283,1.32075)--(9.45283,1.41509)--(9.54717,1.41509)--(9.54717,1.32075)--cycle;
\fill[black] (9.54717,0.188679)--(9.54717,0.283019)--(9.64151,0.283019)--(9.64151,0.188679)--cycle;
\fill[black] (9.54717,0.283019)--(9.54717,0.377358)--(9.64151,0.377358)--(9.64151,0.283019)--cycle;
\fill[black] (9.54717,0.754717)--(9.54717,0.849057)--(9.64151,0.849057)--(9.64151,0.754717)--cycle;
\fill[black] (9.64151,0)--(9.64151,0.0943396)--(9.73585,0.0943396)--(9.73585,0)--cycle;
\fill[black] (9.64151,1.13208)--(9.64151,1.22642)--(9.73585,1.22642)--(9.73585,1.13208)--cycle;
\fill[black] (9.73585,0.188679)--(9.73585,0.283019)--(9.83019,0.283019)--(9.83019,0.188679)--cycle;
\fill[black] (9.83019,0.377358)--(9.83019,0.471698)--(9.92453,0.471698)--(9.92453,0.377358)--cycle;
\fill[black] (9.92453,0.0943396)--(9.92453,0.188679)--(10.0189,0.188679)--(10.0189,0.0943396)--cycle;
\fill[black] (10.0189,0)--(10.0189,0.0943396)--(10.1132,0.0943396)--(10.1132,0)--cycle;
\fill[black] (10.0189,0.471698)--(10.0189,0.566038)--(10.1132,0.566038)--(10.1132,0.471698)--cycle;
\fill[black] (10.1132,0.283019)--(10.1132,0.377358)--(10.2075,0.377358)--(10.2075,0.283019)--cycle;
\fill[black] (10.1132,0.566038)--(10.1132,0.660377)--(10.2075,0.660377)--(10.2075,0.566038)--cycle;
\fill[black] (10.2075,0.566038)--(10.2075,0.660377)--(10.3019,0.660377)--(10.3019,0.566038)--cycle;
\fill[black] (10.2075,1.98113)--(10.2075,2.07547)--(10.3019,2.07547)--(10.3019,1.98113)--cycle;
\fill[black] (10.3019,0.377358)--(10.3019,0.471698)--(10.3962,0.471698)--(10.3962,0.377358)--cycle;
\fill[black] (10.3962,0.283019)--(10.3962,0.377358)--(10.4906,0.377358)--(10.4906,0.283019)--cycle;
\fill[black] (10.4906,1.69811)--(10.4906,1.79245)--(10.5849,1.79245)--(10.5849,1.69811)--cycle;
\fill[black] (10.5849,1.41509)--(10.5849,1.50943)--(10.6792,1.50943)--(10.6792,1.41509)--cycle;
\fill[black] (10.6792,0.943396)--(10.6792,1.03774)--(10.7736,1.03774)--(10.7736,0.943396)--cycle;
\fill[black] (10.7736,0.188679)--(10.7736,0.283019)--(10.8679,0.283019)--(10.8679,0.188679)--cycle;
\fill[black] (10.8679,0.566038)--(10.8679,0.660377)--(10.9623,0.660377)--(10.9623,0.566038)--cycle;
\fill[black] (10.9623,0.188679)--(10.9623,0.283019)--(11.0566,0.283019)--(11.0566,0.188679)--cycle;
\fill[black] (11.0566,0.188679)--(11.0566,0.283019)--(11.1509,0.283019)--(11.1509,0.188679)--cycle;
\fill[black] (11.1509,0.0943396)--(11.1509,0.188679)--(11.2453,0.188679)--(11.2453,0.0943396)--cycle;
\fill[black] (11.2453,0.283019)--(11.2453,0.377358)--(11.3396,0.377358)--(11.3396,0.283019)--cycle;
\fill[black] (11.3396,0.0943396)--(11.3396,0.188679)--(11.4340,0.188679)--(11.4340,0.0943396)--cycle;
\fill[black] (11.4340,0)--(11.4340,0.0943396)--(11.5283,0.0943396)--(11.5283,0)--cycle;
\fill[black] (11.5283,0)--(11.5283,0.0943396)--(11.6226,0.0943396)--(11.6226,0)--cycle;
\fill[black] (11.6226,0.0943396)--(11.6226,0.188679)--(11.7170,0.188679)--(11.7170,0.0943396)--cycle;
\fill[black] (11.7170,0.0943396)--(11.7170,0.188679)--(11.8113,0.188679)--(11.8113,0.0943396)--cycle;
\fill[black] (11.8113,0.377358)--(11.8113,0.471698)--(11.9057,0.471698)--(11.9057,0.377358)--cycle;
\fill[black] (11.9057,0.283019)--(11.9057,0.377358)--(12.0000,0.377358)--(12.0000,0.283019)--cycle;
\draw (9.50000,5.50000) node {graphon with non-compact support};
\end{tikzpicture}
\end{center}
\caption{The adjacency matrices of graphs sampled as described by
\citet*{lp1} (left) and in this paper (right), where we used the
graphon $\mcl W_1=(W_1,[0,1])$ (left) and the graphon $\mcl W_2=(W_2,\R_+)$ (right),
with $W_1(x_1,x_2)=x_1^{-1/2}x_2^{-1/2}$ for $x_1,x_2\in[0,1]$ and
$W_2(x_1,x_2)=\min(0.8,7\min(1,x_1^{-2})\min(1,x_2^{-2}))$ for $x_1,x_2\in\R_+$.  Black
(resp.\ white) indicates that there is (resp.\ is not) an edge.
We rescaled the height of the graphon by $\rho:=1/40$ on the left figure. As
described by \citet*{lp1,lp2} the type of each vertex is sampled independently
and uniformly from $[0,1]$, and each pair of vertices is connected with
probability $\min(\rho W_1,1)$. In the right figure the vertices were
sampled by a Poisson point process on $\R_+$ of intensity $t=4$, and two
vertices were connected independently with a probability given by $W_2$; see
Section~\ref{sec:W-random-results} and the main text of this introduction.
The two graphs have very different qualitative properties. In the left graph
most vertices have a degree close to the average degree, where the average
degree depends on our scaling factor $\rho$. In the right graph the edges
are distributed more inhomogeneously: most of the edges are contained in
induced subgraphs of constant density, and the sparsity is caused by a large
number of vertices with very low degree. \label{fig3} }
\end{figure}

To compare different models, and to discuss notions of convergence, we
introduce the following natural generalization of the cut metric for graphons
on probability spaces to our setting. For two graphons $\W_1=(W_1,\scr S_1)$
and $\W_2=(W_2,\scr S_2)$, this metric is easiest to define when the two
graphons are defined over the same space. However, for applications we want
to compare graphons over different spaces, say two Borel spaces $\scr S_1$
and $\scr S_2$. Assuming that both Borel spaces have infinite total measure,
the cut distance between $\W_1$ and $\W_2$ can then be defined as
\begin{equation}
\label{def:delta-Borel}
\delta_\square(\mcl W_1,\mcl W_2)=
\inf_{\psi_1,\psi_2}
\sup_{U,V\subseteq \R_+^2}
\left|\int_{U\times V} \big(W_1^{\psi_1}-W_2^{\psi_2}\big)\,d\mu \,d\mu\right|,
\end{equation}
where we take the infimum over measure-preserving maps $\psi_j\colon \R_+\to
S_j$ for $j=1,2$, $W_j^{\psi_j}(x,y):=W_j(\psi_j(x),\psi_j(y))$ for $x,y \in
\R_+$, and the supremum is over measurable sets $U,V\subseteq \R_+^2$. (See
Definition~\ref{defn1} below for the definition of the cut distance for
graphons over general spaces, including the case where one or both spaces
have finite total mass.)  We call two graphons \emph{equivalent} if they have
cut distance zero. As we will see, two graphons are equivalent if and only if
the random families $(G_t)_{t\geq 0}$ generated from these graphons have the
same distribution; see Theorem~\ref{prop11} below.

To compare graphs and graphons, we embed a graph on $n$ vertices into the set
of step functions over $[0,1]^2$ in the usual way by decomposing $[0,1]$ into
adjacent intervals $I_1,\dots,I_n$ of lengths $1/n$, and define a step
function $W^G$ as the function which is equal to $1$ on $I_i\times I_j$ if
$i$ and $j$ are connected in $G$, and equal to $0$ otherwise. Extending $W^G$
to a function on $\R_+^2$ by setting it to zero outside of $[0,1]^2$, we can
then compare graphs to graphons on measure spaces of infinite mass, and in
particular we get a notion of convergence in metric of a sequence of graphs
$(G_n)_{n\in\N}$ to a graphon $\W$.

In the classical theory of graph convergence, such a sequence will  converge
to the zero graphon whenever the sequence
is sparse.%
\footnote{Here, as usual, a sequence of simple graphs is considered sparse if
the number of edges divided by the square of the number of vertices goes to
zero.} We resolve this difficulty by rescaling the input arguments of the
step function $W^G$ so as to get a ``stretched graphon''
$\W^{G,s}=(W^{G,s},\R_+)$ satisfying $\|W^{G,s}\|_1=1$. Equivalently, we may
interpret $\W^{G,s}$ as a graphon where the measure of the underlying measure
space is rescaled. See Figure~\ref{fig2} for an illustration, which also
compares the rescaling in the current paper with the rescaling considered by
\citet*{lp1}. We say that $(G_n)_{n\in\N}$ converges to a graphon $\W$ (with
$L^1$ norm equal to 1) for the stretched cut metric if
$\lim_{n\rta\infty}\delta_\square(\W^{G_n,s},\W)=0$. Graphons on
$\sigma$-finite measure spaces of infinite total measure may therefore be
considered as limiting objects for sequences of sparse graphs, similarly as
graphons on probability spaces are considered limits of dense graphs. We
prove that graphon processes converge to the generating graphon in the
stretched cut metric; see Proposition~\ref{prop4} in
Section~\ref{sec:W-random-results}. We will also consider another family of
random sparse graphs associated with a graphon $\W$ over a $\sigma$-finite
measure space, and prove that these graphs are also converging for the
stretched cut metric.

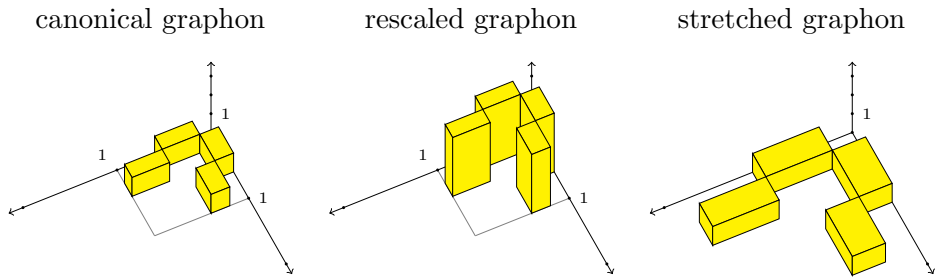
\begin{figure}
\begin{center}
\begin{tikzpicture}[line width = 0.2pt]
\draw[gray] (-1.2500,-0.50000)--(-0.75000,-1.3750);
\draw[gray] (0.50000,-0.87500)--(-0.75000,-1.3750);
\draw[->] (0,0)--(0,0.93750);
\draw[->] (0,0)--(-2.6875,-1.0750);
\draw[->] (0,0)--(1.0750,-1.8813);
\draw (0,0.25000) node[right] {\tiny 1};
\draw (0.50000,-0.87500) node[right] {\tiny 1};
\draw (-1.2500,-0.50000) node[above left] {\tiny 1};
\fill (0,0) circle [radius=0.02];
\fill (0,0.25000) circle [radius=0.02];
\fill (0,0.50000) circle [radius=0.02];
\fill (0,0.75000) circle [radius=0.02];
\fill (-1.2500,-0.50000) circle [radius=0.02];
\fill (-2.5000,-1.0000) circle [radius=0.02];
\fill (0.50000,-0.87500) circle [radius=0.02];
\fill (1.0000,-1.7500) circle [radius=0.02];
\fill[yellow] (-0.25000,0.15000)--(-0.75000,-0.050000)--(-0.65000,-0.22500)--(-0.15000,-0.025000)--cycle;
\draw (-0.25000,0.15000)--(-0.75000,-0.050000)--(-0.65000,-0.22500)--(-0.15000,-0.025000)--cycle;
\fill[yellow] (-0.75000,-0.30000)--(-0.75000,-0.050000)--(-0.65000,-0.22500)--(-0.65000,-0.47500)--cycle;
\draw (-0.75000,-0.30000)--(-0.75000,-0.050000)--(-0.65000,-0.22500)--(-0.65000,-0.47500)--cycle;
\fill[yellow] (-0.15000,-0.27500)--(-0.15000,-0.025000)--(-0.65000,-0.22500)--(-0.65000,-0.47500)--cycle;
\draw (-0.15000,-0.27500)--(-0.15000,-0.025000)--(-0.65000,-0.22500)--(-0.65000,-0.47500)--cycle;
\fill[yellow] (0.10000,0.075000)--(-0.15000,-0.025000)--(0.050000,-0.37500)--(0.30000,-0.27500)--cycle;
\draw (0.10000,0.075000)--(-0.15000,-0.025000)--(0.050000,-0.37500)--(0.30000,-0.27500)--cycle;
\fill[yellow] (-0.15000,-0.27500)--(-0.15000,-0.025000)--(0.050000,-0.37500)--(0.050000,-0.62500)--cycle;
\draw (-0.15000,-0.27500)--(-0.15000,-0.025000)--(0.050000,-0.37500)--(0.050000,-0.62500)--cycle;
\fill[yellow] (0.30000,-0.52500)--(0.30000,-0.27500)--(0.050000,-0.37500)--(0.050000,-0.62500)--cycle;
\draw (0.30000,-0.52500)--(0.30000,-0.27500)--(0.050000,-0.37500)--(0.050000,-0.62500)--cycle;
\fill[yellow] (-0.65000,-0.22500)--(-1.1500,-0.42500)--(-1.0500,-0.60000)--(-0.55000,-0.40000)--cycle;
\draw (-0.65000,-0.22500)--(-1.1500,-0.42500)--(-1.0500,-0.60000)--(-0.55000,-0.40000)--cycle;
\fill[yellow] (-1.1500,-0.67500)--(-1.1500,-0.42500)--(-1.0500,-0.60000)--(-1.0500,-0.85000)--cycle;
\draw (-1.1500,-0.67500)--(-1.1500,-0.42500)--(-1.0500,-0.60000)--(-1.0500,-0.85000)--cycle;
\fill[yellow] (-0.55000,-0.65000)--(-0.55000,-0.40000)--(-1.0500,-0.60000)--(-1.0500,-0.85000)--cycle;
\draw (-0.55000,-0.65000)--(-0.55000,-0.40000)--(-1.0500,-0.60000)--(-1.0500,-0.85000)--cycle;
\fill[yellow] (0.050000,-0.37500)--(-0.20000,-0.47500)--(0,-0.82500)--(0.25000,-0.72500)--cycle;
\draw (0.050000,-0.37500)--(-0.20000,-0.47500)--(0,-0.82500)--(0.25000,-0.72500)--cycle;
\fill[yellow] (-0.20000,-0.72500)--(-0.20000,-0.47500)--(0,-0.82500)--(0,-1.0750)--cycle;
\draw (-0.20000,-0.72500)--(-0.20000,-0.47500)--(0,-0.82500)--(0,-1.0750)--cycle;
\fill[yellow] (0.25000,-0.97500)--(0.25000,-0.72500)--(0,-0.82500)--(0,-1.0750)--cycle;
\draw (0.25000,-0.97500)--(0.25000,-0.72500)--(0,-0.82500)--(0,-1.0750)--cycle;
\draw (-0.80625,1.4906) node {canonical graphon};
\draw[gray] (3.0125,-0.50000)--(3.5125,-1.3750);
\draw[gray] (4.7625,-0.87500)--(3.5125,-1.3750);
\draw[->] (4.2625,0)--(4.2625,0.93750);
\draw[->] (4.2625,0)--(1.5750,-1.0750);
\draw[->] (4.2625,0)--(5.3375,-1.8813);
\draw (4.7625,-0.87500) node[right] {\tiny 1};
\draw (3.0125,-0.50000) node[above left] {\tiny 1};
\fill (4.2625,0) circle [radius=0.02];
\fill (4.2625,0.25000) circle [radius=0.02];
\fill (4.2625,0.50000) circle [radius=0.02];
\fill (4.2625,0.75000) circle [radius=0.02];
\fill (3.0125,-0.50000) circle [radius=0.02];
\fill (1.7625,-1.0000) circle [radius=0.02];
\fill (4.7625,-0.87500) circle [radius=0.02];
\fill (5.2625,-1.7500) circle [radius=0.02];
\fill[yellow] (4.0125,0.68125)--(3.5125,0.48125)--(3.6125,0.30625)--(4.1125,0.50625)--cycle;
\draw (4.0125,0.68125)--(3.5125,0.48125)--(3.6125,0.30625)--(4.1125,0.50625)--cycle;
\fill[yellow] (3.5125,-0.30000)--(3.5125,0.48125)--(3.6125,0.30625)--(3.6125,-0.47500)--cycle;
\draw (3.5125,-0.30000)--(3.5125,0.48125)--(3.6125,0.30625)--(3.6125,-0.47500)--cycle;
\fill[yellow] (4.1125,-0.27500)--(4.1125,0.50625)--(3.6125,0.30625)--(3.6125,-0.47500)--cycle;
\draw (4.1125,-0.27500)--(4.1125,0.50625)--(3.6125,0.30625)--(3.6125,-0.47500)--cycle;
\fill[yellow] (4.3625,0.60625)--(4.1125,0.50625)--(4.3125,0.15625)--(4.5625,0.25625)--cycle;
\draw (4.3625,0.60625)--(4.1125,0.50625)--(4.3125,0.15625)--(4.5625,0.25625)--cycle;
\fill[yellow] (4.1125,-0.27500)--(4.1125,0.50625)--(4.3125,0.15625)--(4.3125,-0.62500)--cycle;
\draw (4.1125,-0.27500)--(4.1125,0.50625)--(4.3125,0.15625)--(4.3125,-0.62500)--cycle;
\fill[yellow] (4.5625,-0.52500)--(4.5625,0.25625)--(4.3125,0.15625)--(4.3125,-0.62500)--cycle;
\draw (4.5625,-0.52500)--(4.5625,0.25625)--(4.3125,0.15625)--(4.3125,-0.62500)--cycle;
\fill[yellow] (3.6125,0.30625)--(3.1125,0.10625)--(3.2125,-0.068750)--(3.7125,0.13125)--cycle;
\draw (3.6125,0.30625)--(3.1125,0.10625)--(3.2125,-0.068750)--(3.7125,0.13125)--cycle;
\fill[yellow] (3.1125,-0.67500)--(3.1125,0.10625)--(3.2125,-0.068750)--(3.2125,-0.85000)--cycle;
\draw (3.1125,-0.67500)--(3.1125,0.10625)--(3.2125,-0.068750)--(3.2125,-0.85000)--cycle;
\fill[yellow] (3.7125,-0.65000)--(3.7125,0.13125)--(3.2125,-0.068750)--(3.2125,-0.85000)--cycle;
\draw (3.7125,-0.65000)--(3.7125,0.13125)--(3.2125,-0.068750)--(3.2125,-0.85000)--cycle;
\fill[yellow] (4.3125,0.15625)--(4.0625,0.056250)--(4.2625,-0.29375)--(4.5125,-0.19375)--cycle;
\draw (4.3125,0.15625)--(4.0625,0.056250)--(4.2625,-0.29375)--(4.5125,-0.19375)--cycle;
\fill[yellow] (4.0625,-0.72500)--(4.0625,0.056250)--(4.2625,-0.29375)--(4.2625,-1.0750)--cycle;
\draw (4.0625,-0.72500)--(4.0625,0.056250)--(4.2625,-0.29375)--(4.2625,-1.0750)--cycle;
\fill[yellow] (4.5125,-0.97500)--(4.5125,-0.19375)--(4.2625,-0.29375)--(4.2625,-1.0750)--cycle;
\draw (4.5125,-0.97500)--(4.5125,-0.19375)--(4.2625,-0.29375)--(4.2625,-1.0750)--cycle;
\draw (3.4563,1.4906) node {rescaled graphon};
\draw[->] (8.5250,0)--(8.5250,0.93750);
\draw[->] (8.5250,0)--(5.8375,-1.0750);
\draw[->] (8.5250,0)--(9.6000,-1.8813);
\draw (8.5250,0.25000) node[right] {\tiny 1};
\fill (8.5250,0) circle [radius=0.02];
\fill (8.5250,0.25000) circle [radius=0.02];
\fill (8.5250,0.50000) circle [radius=0.02];
\fill (8.5250,0.75000) circle [radius=0.02];
\fill (7.2750,-0.50000) circle [radius=0.02];
\fill (6.0250,-1.0000) circle [radius=0.02];
\fill (9.0250,-0.87500) circle [radius=0.02];
\fill (9.5250,-1.7500) circle [radius=0.02];
\fill[yellow] (8.0831,0.073223)--(7.1992,-0.28033)--(7.3760,-0.58969)--(8.2598,-0.23614)--cycle;
\draw (8.0831,0.073223)--(7.1992,-0.28033)--(7.3760,-0.58969)--(8.2598,-0.23614)--cycle;
\fill[yellow] (7.1992,-0.53033)--(7.1992,-0.28033)--(7.3760,-0.58969)--(7.3760,-0.83969)--cycle;
\draw (7.1992,-0.53033)--(7.1992,-0.28033)--(7.3760,-0.58969)--(7.3760,-0.83969)--cycle;
\fill[yellow] (8.2598,-0.48614)--(8.2598,-0.23614)--(7.3760,-0.58969)--(7.3760,-0.83969)--cycle;
\draw (8.2598,-0.48614)--(8.2598,-0.23614)--(7.3760,-0.58969)--(7.3760,-0.83969)--cycle;
\fill[yellow] (8.7018,-0.059359)--(8.2598,-0.23614)--(8.6134,-0.85485)--(9.0553,-0.67808)--cycle;
\draw (8.7018,-0.059359)--(8.2598,-0.23614)--(8.6134,-0.85485)--(9.0553,-0.67808)--cycle;
\fill[yellow] (8.2598,-0.48614)--(8.2598,-0.23614)--(8.6134,-0.85485)--(8.6134,-1.1049)--cycle;
\draw (8.2598,-0.48614)--(8.2598,-0.23614)--(8.6134,-0.85485)--(8.6134,-1.1049)--cycle;
\fill[yellow] (9.0553,-0.92808)--(9.0553,-0.67808)--(8.6134,-0.85485)--(8.6134,-1.1049)--cycle;
\draw (9.0553,-0.92808)--(9.0553,-0.67808)--(8.6134,-0.85485)--(8.6134,-1.1049)--cycle;
\fill[yellow] (7.3760,-0.58969)--(6.4921,-0.94324)--(6.6688,-1.2526)--(7.5527,-0.89905)--cycle;
\draw (7.3760,-0.58969)--(6.4921,-0.94324)--(6.6688,-1.2526)--(7.5527,-0.89905)--cycle;
\fill[yellow] (6.4921,-1.1932)--(6.4921,-0.94324)--(6.6688,-1.2526)--(6.6688,-1.5026)--cycle;
\draw (6.4921,-1.1932)--(6.4921,-0.94324)--(6.6688,-1.2526)--(6.6688,-1.5026)--cycle;
\fill[yellow] (7.5527,-1.1490)--(7.5527,-0.89905)--(6.6688,-1.2526)--(6.6688,-1.5026)--cycle;
\draw (7.5527,-1.1490)--(7.5527,-0.89905)--(6.6688,-1.2526)--(6.6688,-1.5026)--cycle;
\fill[yellow] (8.6134,-0.85485)--(8.1714,-1.0316)--(8.5250,-1.6503)--(8.9669,-1.4736)--cycle;
\draw (8.6134,-0.85485)--(8.1714,-1.0316)--(8.5250,-1.6503)--(8.9669,-1.4736)--cycle;
\fill[yellow] (8.1714,-1.2816)--(8.1714,-1.0316)--(8.5250,-1.6503)--(8.5250,-1.9003)--cycle;
\draw (8.1714,-1.2816)--(8.1714,-1.0316)--(8.5250,-1.6503)--(8.5250,-1.9003)--cycle;
\fill[yellow] (8.9669,-1.7236)--(8.9669,-1.4736)--(8.5250,-1.6503)--(8.5250,-1.9003)--cycle;
\draw (8.9669,-1.7236)--(8.9669,-1.4736)--(8.5250,-1.6503)--(8.5250,-1.9003)--cycle;
\draw (7.7188,1.4906) node {stretched graphon};
\end{tikzpicture}
\end{center}
\caption{The figure shows three graphons associated with the same simple
graph $G$ on five vertices. In the classical theory of graphons all simple
sparse graphs  converge to the zero graphon. We may prevent this by
renormalizing the graphons, either by rescaling the height of the graphon
(middle) or by stretching the domain on which it is defined (right). The
first approach was chosen by \citet*{br-09} and \citet*{lp1,lp2}, and the second approach
is chosen in this paper. In our forthcoming paper \citep*{multiscale2} we
choose a combined approach, where the renormalization depends on the
observed graph.} \label{fig2}
\end{figure}

Particular random graph models of special interest arise by considering
certain classes of graphons $\W$. \citet*{caron-fox} consider graphons on the
form $W(x_1,x_2)=1-\exp(-f(x_1)f(x_2))$ (with a slightly different definition
on the diagonal, since they also allow for self-edges) for certain decreasing
functions $f\colon \R_+\to\R_+$. In this model $x$ represents a sociability
parameter of each vertex. A multi-edge version of this model allows for an
alternative sampling procedure to the one we present above \cite[Section
3]{caron-fox}. \citet*{kallenberg-block} introduced a generalization of the
model of \citet*{caron-fox} to graphs with block structure. In this model
each node is associated to a type from a finite index set
$[K]:=\{1,\dots,K\}$ for some $K\in\N$, in addition to its sociability
parameter, such that the probability of two nodes connecting depends both on
their type and their sociability. More generally we can obtain sparse graphs
with block structure by considering  integrable functions $W_{k_1,k_2}\colon
\R_{+}^2\to [0,1]$ for $k_1,k_2\in\{1,\dots,K\}$, and defining $S:=[K]\times
\R_+$ and $W((k_1,x_1),(k_2,x_2)):=W_{k_1,k_2}(x_1,x_2)$. As compared to the
block model of \citet*{kallenberg-block}, this allows for a more complex
interaction within and between the blocks. An alternative generalization of
the stochastic block model to our setting is to consider infinitely many
disjoint intervals $I_k\subset\R_+$ for $k\in\N$, and define
$W:=\sum_{k_1,k_2\in\N}p_{k_1,k_2}\1_{I_{k_1}\times I_{k_2}}$ for constants
$p_{k_1,k_2}\in[0,1]$. For the block model of \citet*{kallenberg-block} and
our first generalization above (with $S:=[K]\times \R_+$), the degree
distribution of the vertices within each block will typically be strongly
inhomogeneous; by contrast, in our second generalization above (with
infinitely many blocks), all vertices within the same block have the same
connectivity probabilities, and hence the degree distribution will be more
homogeneous.

We can also model sparse graphs with mixed membership structure within our
framework. In this case we let $\wt S\subset[0,1]^{K}$ be the standard
$(K-1)$-simplex, and define $S:=\wt S\times\R_+$. For a vertex with feature
$(\wt x,x)\in \wt S\times \R_+$ the first coordinate $\wt x=(\wt
x_1,\dots,\wt x_K)$ is a vector such that $\wt x_j$ for $j\in[K]$ describes
the proportion of time the vertex is part of community $j\in[K]$, and the
second coordinate $x$ describes the role of the vertex within the community;
for example, $x$ could be a sociability parameter. For each $k_1,k_2\in[K]$
let $\W_{k_1,k_2}=(W_{k_1,k_2},\R_+)$ be a graphon describing the
interactions between the communities $k_1$ and $k_2$. We define our mixed
membership graphon $\W=(W,\scr S)$ by
\[
W\big((\wt x^1,x^1),(\wt
x^2,x^2)\big) :=\sum_{k_1,k_2\in[K]} \wt x^1_{k_1}\wt x^2_{k_2}
W_{k_1,k_2}(x^1,x^2).
\]
Alternatively, we could define $S:=\wt S\times\R_+^K$, which would provide a
model where, for example, the sociability of a node varies depending on which
community it is part of.

In the classical setting of dense graphs, many papers only consider graphons
defined on the unit square, instead of graphons on more general probability
spaces. This is justified by the fact that every graphon with a probability
space as base space is equivalent to a graphon with base space $[0,1]$. The
analogue in our setting would be graphons over $\R_+$ equipped with the
Lebesgue measure. As the examples in the preceding paragraphs illustrate, for
certain random graph models it is more natural to consider another underlying
measure space. For example, each coordinate in some higher-dimensional space
may correspond to a particular feature of the vertices, and changing the base
space can disrupt certain properties of the graphon, such as smoothness
conditions. For this reason we consider graphons defined on general
$\sigma$-finite measure spaces in this paper. However, we will prove that
every graphon is equivalent to a graphon on $\R_+$ equipped with the Borel
$\sigma$-algebra and Lebesgue measure, in the sense that their cut distance
is zero; see Proposition~\ref{prop7} in Section~\ref{sec:graphon-results}. As
stated before, our results then imply that they correspond to the same random
graph model.

The set of $[0,1]$-valued graphons on probability spaces is compact for the
cut metric. For the possibly unbounded graphons studied by \citet*{lp1},
which are real-valued and defined on probability spaces, compactness holds if
we consider closed subsets of the space of graphons which are \emph{uniformly
upper regular} (see Section~\ref{sec:Convergence-results} for the
definition). In our setting, where we look at graphons over spaces of
possibly infinite measure, the analogous regularity condition is
\emph{uniform regularity of tails} if we restrict ourselves to, say,
$[0,1]$-valued graphons. In particular our results imply that a sequence of
simple graphs with uniformly regular tails is subsequentially convergent, and
conversely, that every convergent sequence of simple graphs has uniformly
regular tails. See Theorem~\ref{prop1} in
Section~\ref{sec:Convergence-results} and the two corollaries following this
theorem.

In the setting of dense graphs, convergence for the cut metric is equivalent
to left convergence, meaning that subgraph densities converge. This
equivalence does not hold in our setting, or for the unbounded graphons
studied by \citet*{lp1,lp2}; its failure is characteristic of sparse graphs,
because deleting even a tiny fraction of the edges in a sparse graph can
radically change the densities of larger subgraphs (see the discussion by
\citel{lp1}, Section~2.9). However, randomly sampled graphs do satisfy a
notion of left convergence; see Proposition~\ref{prop3} in
Section~\ref{sec7}.

As previously mentioned, in our forthcoming paper \citep*{multiscale2} we
will generalize and unify the theories and models presented by
\citet*{br-09}, \citet*{lp1,lp2}, \citet*{caron-fox},
\citet*{kallenberg-block}, and \citet*{veitchroy}. Along with the
introduction of a generalized model for sampling graphs and an alternative
(and weaker) cut metric, we will prove a number of convergence properties of
these graphs. Since the graphs in this paper are obtained as a special case
of the graphs in our forthcoming paper, the mentioned convergence results
also hold in our setting.

In Section~\ref{sec2} we will state the main results of this paper, which
will be proved in the subsequent appendices. In Appendix~\ref{sec4} we prove
that the cut metric $\delta_\square$ is well defined. In Appendix
\ref{sec:equiv} we prove that any graphon is equivalent to a graphon with
underlying measure space $\R_+$. We also prove that under certain conditions
on the underlying measure space we may define the cut metric $\delta_\square$
in a number of equivalent ways.  In Appendix~\ref{sec:measurability}, we deal
with some technicalities regarding graph-valued processes.  In
Appendix~\ref{sec3} we prove that certain random graph models derived from a
graphon $\W$, including the graphon processes defined above, give graphs
converging to $\W$ for the cut metric. We also prove that two graphons are
equivalent (i.e., they have cut distance zero) iff the corresponding graphon
processes are equal in law. In Appendix \ref{sec:compact} we prove that
uniform regularity of tails is sufficient to guarantee subsequential metric
convergence for a sequence of graphs; conversely, we prove that every
convergent sequence of graphs with non-negative edge weights has uniformly
regular tails. In Appendix~\ref{sec8} we prove some basic properties of
sequences of graphs which are metric convergent, for example that metric
convergence implies unbounded average degree if the number of edges diverge
and the graph does not have too many isolated vertices; see
Proposition~\ref{prop:unbdd_avg_deg} below.  We also compare the notion of
metric graph convergence in this paper to the one studied by \citet*{lp1}. In
Appendix~\ref{sec6} we prove with reference to the Kallenberg theorem for
jointly exchangeable measures that graphon processes for integrable $W$ are
uniquely characterized as exchangeable graph processes satisfying uniform
tail regularity. We also describe more general families of graphs that may be
obtained from the Kallenberg representation theorem if this regularity
condition is not imposed. Finally, in Appendix~\ref{sec5} we prove our
results on left convergence of graphon processes.

\begin{remark}
After writing a first draft of this work, but a little over a month before
completing the paper, we became aware of parallel, independent work by
\citet*{veitchroy}, who introduce a closely related model for exchangeable
sparse graphs and interpret it with reference to the Kallenberg theorem for
exchangeable measures. The random graph model studied by \citet*{veitchroy}
is (up to minor differences) the same as the graphon processes introduced in
the current paper. Aside from both introducing this model, the results of the
two papers are essentially disjoint. While \citet*{veitchroy} focus on
particular properties of the graphs in a graphon process (in particular, the
expected number of edges and vertices, the degree distribution, and the
existence of a giant component under certain assumptions on $\W$), our focus
is graph convergence, the cut metric, and the question of when two different
graphons lead to the same graphon process.

See also the subsequent paper by \citet*{janson16} expanding on the results
of our paper, characterizing in particular when two graphons are equivalent,
and proving additional compactness results for graphons over $\sigma$-finite
spaces.
\end{remark}

\section{Definitions and Main Results}
\label{sec2}

We will work mainly with simple graphs, but we will allow the graphs to have
weighted vertices and edges for some of our definitions and results. We
denote the vertex set of a graph $G$ by $V(G)$ and the edge set of $G$ by
$E(G)$. The sets $V(G)$ and $E(G)$ may be infinite, but we require them to be
countable. If $G$ is weighted, with edge weights $\beta_{ij}(G)$ and vertex
weights $\alpha_i(G)$, we require the vertex weights to be non-negative, and
we often (but not always) require that $\|\beta(G)\|_1:=\sum_{i,j\in
V(G)}\alpha_i(G)\alpha_j(G)|\beta_{ij}(G)|<\infty$ (note that
$\|\beta(G)\|_1$ is defined in such a way that for an unweighted graph, it is
equal to $2|E(G)|$, as opposed to the density, which is ill-defined if
$|V(G)|=\infty$). We define the edge density of a finite simple graph $G$ to
be $\rho(G):=2|E(G)|/|V(G)|^2$. Letting $\N=\{1,2,\dots\}$ denote the
positive integers, a sequence $(G_n)_{n\in\N}$ of simple, finite graphs will
be called sparse if $\rho(G_n)\to 0$ as $n\to\infty$, and dense if
$\liminf_{n\to\infty} \rho(G_n)>0$. When we consider graph-valued stochastic
processes $(G_n)_{n\in\N}$ or $(G_t)_{t\geq 0}$ of simple graphs, we will
assume each vertex is labeled by a distinct number in $\N$, so we can view
$V(G)$ as a subset of $\N$ and $E(G)$ as a subset of $\N\times\N$. The labels
allow us to keep track of individual vertices in the graph over time. In
Section \ref{sec:W-random-results} we define a topology and $\sigma$-algebra
on the set of such graphs.

\subsection{Measure-theoretic Preliminaries}
\label{sec:measuretheory}

We start by recalling several notions from measure theory.

For two  measure spaces $\scr{S}=(S,\cS,\mu)$ and $\scr{S}'=(S',\cS',\mu')$,
a measurable map $\phi\colon S\to S'$ is called \emph{measure-preserving} if
for every $A\in\cS'$ we have $\mu(\phi^{-1}(A))=\mu'(A)$. Two measure spaces
$(S,\cS,\mu)$ and $(S',\cS',\mu')$ are called \emph{isomorphic} if there
exists a bimeasurable, bijective, and measure-preserving map $\phi\colon S\to
S'$.  A \emph{Borel measure space} is defined as a measure space that is
isomorphic to a Borel subset of a complete separable metric space equipped
with a Borel measure.

Throughout most of this paper, we consider \emph{$\sigma$-finite measure
spaces}, i.e., spaces $\scr S=(S,\cS,\mu)$ such that $S$ can be written as a
countable union of sets $A_i\in \cS$ with $\mu(A_i)<\infty$. Recall that a
set $A\in\cS$ is an \emph{atom} if $\mu(A)>0$ and if every measurable
$B\subseteq A$ satisfies either $\mu(B)=0$ or $\mu(B)=\mu(A)$. The measure
space $\scr S$ is atomless if it has no atoms. Every atomless $\sigma$-finite
Borel space of infinite measure is isomorphic to $(\R_+,\mcl B,\lambda)$,
where $\mcl B$ is the Borel $\sigma$-algebra and $\lambda$ is Lebesgue
measure; for the convenience of the reader, we prove this as
Lemma~\ref{prop9} below.

We also need the notion of a coupling, a concept well known for probability
spaces: if $(S_i,\cS_i,\mu_i)$ is a measure space for $i=1,2$ and
$\mu_1(S_1)=\mu_2(S_2)\in(0,\infty]$, we say that $\mu$ is a \emph{coupling}
of $\mu_1$ and $\mu_2$ if $\mu$ is a measure on $(S_1\times S_2,\cS_1\times
\cS_2)$ with marginals $\mu_1$ and $\mu_2$, i.e., if $\mu(U\times
S_2)=\mu_1(U)$ for all $U\in \cS_1$ and $\mu(S_1\times U)=\mu_2(U)$ for all
$U\in \cS_2$. Note that this definition of coupling is closely related to the
definition of coupling of probability measures, which applies when
$\mu_1(S_1)=\mu_2(S_2)=1$.  For probability spaces, it is easy to see that
every pair of measures has a coupling (for example, the product space of the
two probability spaces). We prove the existence of a coupling for
$\sigma$-finite measure spaces in Appendix~\ref{sec4}, where this fact is
stated as part of a more general lemma, Lemma~\ref{prop20}.

Finally, we say that a measure space $\wt{\scr S}=(\wt S,\wt {\cS},\wt\mu)$
\emph{extends} a measure space $\scr S=(S,\cS,\mu)$ if $S\in\wt{\cS}$,
$\cS=\{A\cap S\,:\,A\in\wt{\cS}\}$, and $\mu(A)=\wt\mu(A)$ for all $A\in\cS$.
We say that $\scr S$ is a \emph{restriction} of $\wt{\scr S}$, or, if $S$ is
specified, \emph{the restriction} of $\wt{\scr S}$ \emph{to} $S$.

\subsection{Graphons and Cut Metric}
\label{sec:graphon-results}

We will work with the following definition of a graphon.

\begin{definition}
A \emph{graphon} is a pair $\mcl W=(W,\scr S)$, where  $\scr S=(S,\cS,\mu)$
is a $\sigma$-finite measure space satisfying $\mu(S)>0$ and $W$ is a
symmetric real-valued function $W\in L^1(S\times S)$ that is measurable with
respect to the product $\sigma$-algebra $\cS\times\cS$ and integrable with
respect to $\mu \times \mu$. We say that $\mcl W$ is a \emph{graphon over $\scr S$}.
\end{definition}

\begin{remark}
Most literature on graphons defines a graphon to be the function $W$ instead
of the pair $(W,\scr S)$. We have chosen the above definition since the
underlying measure space will play an important role. Much literature on
graphons requires $W$ to take values in $[0,1]$, and some of our results will
also be restricted to this case. The major difference between the above
definition and the definition of a graphon in the existing literature,
however, is that we allow the graphon to be defined on a measure space of
possibly infinite measure, instead of a probability space.\footnote{The term
``graphon'' was coined by \citet*{denseconv1}, but the use of this concept in
combinatorics goes back to at least \citet*{FK99}, who considered a version
of the regularity lemma for functions over $[0,1]^2$. As a limit object for
convergent graph sequences it was introduced by \citet*{ls-graphlimits},
where it was called a $W$-function, and graphons over general probability
spaces were first studied by \citet*{BCL10} and \citet*{janson-survey}.}
\end{remark}

\begin{remark}
One may relax the integrability condition for $W$ in the above definition
such that the corresponding random graph model (as defined in Definition
\ref{defn3} below) still gives graphs with finitely many vertices and edges
for each bounded time. This more general definition is used by
\citet*{veitchroy}. We work with the above definition since the majority of
the analysis in this paper is related to convergence properties and graph
limits, and our definition of the cut metric is most natural for integrable
graphons. An exception is the notion of subgraph density convergence in the
corresponding random graph model, which we discuss in the more general
setting of not necessarily integrable graphons; see
Remark~\ref{rem:W-notin-L1} below. \label{rmk:notL1}
\end{remark}

We will mainly study simple graphs in the current paper, in particular,
graphs which do not have self-edges. However, the theory can be generalized
in a straightforward way to graphs with self-edges, in which case we would
also impose an integrability condition for $W$ along its diagonal.

If $\W=(W,(S,\mcl B,\lambda))$, where $S$ is a Borel subset of $\R$, $\mcl B$
is the Borel $\sigma$-algebra, and $\lambda$ is Lebesgue measure, we write
$\W=(W,S)$ to simplify notation. For example, we write $\W=(W,\R_+)$ instead
of $\W=(W,(\R_+,\mcl B,\lambda))$.

For any measure space $\scr S=(S,\cS,\mu)$ and integrable function $W\colon
S\times S\to\R$, define the \emph{cut norm} of $W$ over $\scr S$ by
\[
\|W\|_{\square,S,\mu} := \sup_{U,V\in\cS} \left|\int_{U\times V}W(x,y)\,d\mu(x)\, d\mu(y)\right|.
\]
If $S$ and/or $\mu$ is clear from the context we may write
$\|\cdot\|_\square$ or $\|\cdot\|_{\square,\mu}$ to simplify notation.

Given a graphon $\wt{\mcl W}=(\wt W,\wt{\scr{S}})$ with $\wt{\scr S}=(\wt S,
\wt{\cS},\wt\mu)$ and a set $S\in \wt{\cS}$, we say that  $\mcl
W=(W,\scr{S})$ is \emph{the restriction of $\wt{\mcl W}$ to}  $S$ if
$\scr{S}$ is the restriction of $\wt{\scr{S}}$ to $S$ and $\wt W|_{S\times
S}=W$. We say that $\wt{\mcl W}=(\wt W,\wt{\scr{S}})$ is \emph{the trivial
extension of $\mcl W$ to $\wt{\scr S}$} if $\mcl W=(W,\scr{S})$ is the
restriction of $\wt{\mcl W}$ to $S$ and $\op{supp}(\wt W)\subseteq S\times
S$. For measure spaces $\scr{S}=(S,\cS,\mu)$ and $\scr{S}'=(S',\cS',\mu')$, a
graphon $\mcl W=(W,\scr{S})$, and a measurable map $\phi\colon S'\to S$, we
define the graphon $\mcl W^\phi=(W^\phi,\scr{S}')$ by
$W^\phi(x_1,x_2):=W(\phi(x_1),\phi(x_2))$ for $x_1,x_2\in S'$. We say that
$\mcl W^\phi$ (resp.\ $W^\phi$) is a \emph{pullback} of $\mcl W$ (resp.\ $W$)
onto $\scr{S'}$. Finally, let $\|\cdot\|_1$ denote the $L^1$ norm.

\begin{definition}
For $i=1,2$, let $\mcl W_i=(W_i,\scr{S}_i)$ with $\scr{S}_i=(S_i,\mcl
S_i,\mu_i)$ be a graphon.
\begin{itemize}
\item[(i)] If $\mu_1(S_1)=\mu_2(S_2)\in(0,\infty]$, the \emph{cut metric}
    $\delta_\square$ and \emph{invariant $L^{1}$ metric} $\delta_{1}$ are
    defined by
\begin{equation}
\begin{split}
\delta_{\square}(\mcl W_1,\mcl W_2) &:= \inf_\mu \|W_1^{\pi_1}-W_2^{\pi_2}\|_{\square,S_1 \times S_2,\mu}
\quad\text{and}
\\
\delta_{1}(\mcl W_1,\mcl W_2) & := \inf_\mu \|W_1^{\pi_1}-W_2^{\pi_2}\|_{1,S_1 \times S_2,\mu},
\end{split}
\label{eq1}
\end{equation}
where $\pi_i\colon S_1\times S_2\to S_i$ denotes projection for $i=1,2$,
and we take the infimum over all couplings $\mu$ of $\mu_1$ and $\mu_2$.
\item[(ii)] If $\mu_1(S_1)\neq \mu_2(S_2)$, let $\wt{\scr S}_i=(\wt S_i,\wt
    {\cS}_i,\wt\mu_i)$ be a $\sigma$-finite measure space extending $\scr
    S_i$ for $i=1,2$ such that $\wt{\mu}_1(\wt S_1)=\wt{\mu}_2(\wt
    S_2)\in(0,\infty]$. Let $\wt{\mcl W}_i=(\wt W_i,\wt{\scr{S}}_i)$ be the
    trivial extension of $\mcl W_i$ to $\wt{\scr S}_i$, and define
\[
\delta_\square(\mcl W_1,\mcl W_2):=\delta_\square(\wt{\mcl W}_1,\wt{\mcl W}_2)
\quad\text{and}\quad
\delta_1(\mcl W_1,\mcl W_2):=\delta_1(\wt{\mcl W}_1,\wt{\mcl W}_2).
\]
\item[(iii)] We call two graphons $\W_1$ and $\W_2$
    \emph{equivalent} if $\delta_\square(\W_1,\W_2)=0$.
\end{itemize}
\label{defn1}
\end{definition}

The following proposition will be proved in Appendix~\ref{sec4}. Recall that
a pseudometric on a set $S$ is a function from $S\times S$ to $\R_+$ which
satisfies all the requirements of a metric, except that the distance between
two different points might be zero.

\begin{proposition}
The  metrics $\delta_\square$ and $\delta_{1}$ given in Definition
\ref{defn1} are well defined; in other words, under the assumptions of (i)
there exists at least one coupling $\mu$, and under the assumptions of (ii)
the definitions of $\delta_\square(\mcl W_1,\mcl W_2)$ and $\delta_{\1}(\mcl
W_1,\mcl W_2)$ do not depend on the choice of extensions
$\wt{\scr{S}}_1,\wt{\scr{S}}_2$. Furthermore, $\delta_\square$ and
$\delta_{1}$ are pseudometrics on the space of graphons. \label{prop8}
\end{proposition}

An important input to the proof of the proposition (Lemma~\ref{prop21} in
Appendix~\ref{sec4}) is that the $\delta_\square$ (resp.\ $\delta_1$)
distance between two graphons over spaces of equal measure, as defined in
Definition~\ref{defn1}(i), is invariant under trivial extensions. The lemma
is proved by first showing that it holds for step functions (where the proof
more or less boils down to an explicit calculation) and then using the fact
that every graphon can be approximated by a step function.

We will see in Proposition~\ref{prop10} in Appendix~\ref{sec:equiv} that
under additional assumptions on the underlying measure spaces $\scr S_1$ and
$\scr S_2$ the cut metric can be defined equivalently in a number of other
ways, giving, in particular, the equivalence of the definitions
\eqref{def:delta-Borel} and \eqref{eq1} in the case of two Borel spaces of
infinite mass. Similar results hold for the metric $\delta_1$; see
Remark~\ref{rem:alt-def-delta-p}.

While the two metrics $\delta_\square$ and $\delta_1$ are not equivalent, a
fact which is  already well known from the theory of graph convergence for
dense graphs, it turns out that the statement that two graphons have distance
zero in the cut metric is equivalent to the same statement in the invariant
$L^1$ metric.  This is the content of our next proposition.

\begin{proposition}\label{prop:dcut-d1}
Let $\W_1$ and $\W_2$ be graphons. Then $\delta_\square(\W_1,\W_2)=0$ if and
only if $\delta_1(\W_1,\W_2)=0$.
\end{proposition}

The proposition will be proved in Appendix~\ref{sec:equiv}. (We will actually
prove a generalization of this proposition involving an invariant version of
the $L^p$ metric; see Proposition~\ref{prop:dcut-dp}.) The proof proceeds by
first showing (Proposition~\ref{prop27}) that if
$\delta_\square(\W_1,\W_2)=0$ for graphons $\W_i=(W_i,\scr S_i)$ with $\scr
S_i=(S_i,\cS_i,\mu_i)$ for $i=1,2$, then there exists a particular measure
$\mu$ on $S_1\times S_2$  such that
$\|W_1^{\pi_1}-W^{\pi_2}_2\|_{\square,\mu}=0$. Under certain conditions we
may assume that $\mu$ is a coupling measure, in which case it follows that
the infimum in the definition of $\delta_\square$ is a minimum; see
Proposition~\ref{prop:dcut-coupling} below.

To state our next proposition we define a coupling between two graphons
$\W_i=(W_i,\scr S_i)$ with $\scr S_i=(S_i,\cS_i,\mu_i)$ for $i=1,2$ as a pair
of graphons $\wt\W_i$ over a space of the form $\scr S=(S_1\times
S_2,\cS_1\times \cS_2,\mu)$, where $\mu$ is a coupling of $\mu_1$ and $\mu_2$
and $\wt\W_i=W_i^{\pi_i}$, and where as before, $\pi_i$ denotes the
projection from $S_1\times S_2$ onto $S_i$ for $i=1,2$.

\begin{proposition}\label{prop:dcut-coupling}
Let $\W_i$ be graphons over $\sigma$-finite Borel spaces $\scr S_i=(S_i,\mcl
S_i,\mu_i)$, and let $\wt S_i=\{x\in S_i\colon \int |W_i(x,y)|d\mu_i(y)>0\}$,
for $i=1,2$. If $\delta_\square(\W_1,\W_2)=0$, then the restrictions of
$\W_1$ and $\W_2$ to $\wt S_1$ and $\wt S_2$ can be coupled in such a way
that they are equal a.e.
\end{proposition}

The proposition will be proved in Appendix~\ref{sec:equiv}. Note that
\citet*[Theorem~5.3]{janson16} independently proved a similar result,
building on a previous version of the present paper which did not yet contain
Proposition~\ref{prop:dcut-coupling}.  His result states that if the cut
distance between two graphons over $\sigma$-finite Borel spaces is zero, then
there are trivial extensions of these graphons such that the extensions can
be coupled so as to be equal almost everywhere. It is easy to see that our
result implies his, but we believe that with a little more work, it
should be possible to deduce ours from his as well.

\begin{remark}
Note that the classical theory of graphons on probability spaces appears as a
special case of the above definitions by taking $\scr S$ to be a probability
space.  Our definition of the cut metric $\delta_\square$ is equivalent to
the standard definition for graphons on probability spaces; see, for example,
papers by \citet*{denseconv1} and \citet*{janson-survey}. Note that
$\delta_\square$ is not a true metric, only a pseudometric, but we call it a
metric to be consistent with existing literature on graphons. However, it is
a metric on the set of equivalence classes as derived from the equivalence
relation in Definition \ref{defn1} (iii).
\end{remark}

We work with graphons defined on general $\sigma$-finite measure spaces,
rather than graphons on $\R_+$, since particular underlying spaces are more
natural to consider for certain random  graphs or networks. However, the
following proposition shows that every graphon is equivalent to a graphon
over $\R_+$.

\begin{proposition}
For each graphon $\mcl W=(W,\scr{S})$ there exists a graphon $\mcl
W'=(W',\R_+)$ such that $\delta_\square(\W,\W')=0$. \label{prop7}
\end{proposition}

The proof of the proposition follows a similar strategy as the proof of the
analogous result for probability spaces by \citet*[Theorem 3.2]{BCL10} and
\citet*[Theorem 7.1]{janson-survey}, and will be given in Appendix
\ref{sec:equiv}. The proof uses in particular the result that an atomless
$\sigma$-finite Borel space is isomorphic to an interval equipped with
Lebesgue measure (Lemma~\ref{prop9}).

\subsection{Graph Convergence}
\label{sec:Convergence-results}

To define graph convergence in the cut metric, one traditionally
\citep*{graph-hom,ls-graphlimits,denseconv1} embeds the set of graphs into the
set of graphons via the following map. Given any finite weighted graph $G$ we
define the \emph{canonical graphon} $\mcl W^G=(W^G,[0,1])$ as follows. Let
$v_1,\dots,v_n$ be an ordering of the vertices of $G$. For any $v_i\in V(G)$
let $\alpha_i>0$ denote the weight of $v_i$, for any $(v_i,v_j)\in E(G)$ let
$\beta_{ij}\in\R$ denote the weight of the edge $(v_i,v_j)$, and for
$(v_i,v_j)\not\in V(G)$ define $\beta_{ij}=0$. By rescaling the vertex
weights if necessary we assume without loss of generality that
$\sum_{i=1}^{|V(G)|}\alpha_i=1$. If $G$ is simple all vertices have weight
$|V(G)|^{-1}$, and we define $\beta_{ij}:=\1_{(v_i,v_j)\in E(G)}$. Let
$I_1,\dots,I_n$ be a partition of $[0,1]$ into adjacent intervals of lengths
$\alpha_1,\dots,\alpha_n$ (say the first one closed, and all others half
open), and finally define $W^G$ by
\[
W^G(x_1,x_2) = \beta_{ij} \quad\text{if}\quad x_1\in I_i {\text{ and }} x_2\in I_j.
\]
Note that $W^G$ depends on the ordering of the vertices, but that different
orderings give graphons with cut distance zero. We define a sequence of
weighted, finite graphs $G_n$ to be \emph{sparse}\footnote{Note that in the
case of weighted graphs there are multiple natural definitions of what it
means for a sequence of graphs to be sparse or dense. Instead of considering
the $L^1$ norm as in our definition, one may for example consider the
fraction of edges with non-zero weight, either weighted by the vertex weights
or not. In the current paper we do not define what it means for a sequence of
weighted graphs to be dense, since it is not immediate which definition is
most natural, and since the focus of this paper is sparse graphs.} if
$\|W^{G_n}\|_1\rta 0$ as $n\rta\infty$. Note that this generalizes the
definition we gave in the very beginning of Section~\ref{sec2} for simple
graphs.

A sequence $(G_n)_{n\in\N}$ of  graphs is then defined to be \emph{convergent
in metric} if $\mcl W^{G_n}$ is a Cauchy sequence in the metric
$\delta_\square$, and it is said to be \emph{convergent to a graphon} $\W$
if $\delta_\square(\mcl W^{G_n},\W)\to 0$. Equivalently, one can define
convergence of $(G_n)_{n\in\N}$ by identifying a weighted graph $G$ with the
graphon $(\beta(G),\scr{S}_G)$, where $\scr{S}_G$ consists of the vertex set
$V(G)$ equipped with the probability measure given by the weights $\alpha_i$
(or the uniform measure if $G$ has no vertex weights), and $\beta(G)$ is the
function that maps  $(i,j)\in V(G)\times V(G)$ to $\beta_{ij}(G)$.

In the classical theory of graph convergence a sequence of sparse graphs
converges to the trivial graphon with $W\equiv 0$. This follows immediately
from the fact that $\delta_\square(\W^{G_n},0)\leq \|W^{G_n}\|_1\to 0$ for
sparse graphs. To address this problem, \citet*{br-09} and \citet*{lp1}
considered the sequence of reweighted graphons $(\W^{G_n,r})_{n\in\N}$, where
$\mcl W^{G,r}:=(W^{G,r},[0,1])$ with $W^{G,r}:=\frac 1{\|W^G\|_1}W^G$ for any
graph $G$, and defined $(G_n)_{n\in\N}$ to be convergent iff
$(\W^{G_n,r})_{n\in\N}$ is convergent. The theory developed in the current
paper considers a different rescaling, namely a rescaling of the
\emph{arguments} of the function $W^G$, which, as explained after Definition
\ref{defn1a} below, is equivalent to \emph{rescaling the measure} of the
underlying measurable space.

We define the \emph{stretched canonical graphon} $\mcl W^{G,s}$ to be
identical to $\mcl W^{G}$ except that we ``stretch'' the function $W^G$ to a
function $W^{G,s}$ such that $\|W^{G,s}\|_1=1$. More precisely, $\mcl
W^{G,s}:=(W^{G,s}, \R_+)$, where \[ W^{G,s}(x_1,x_2):=
\begin{cases}
W^{G}\left(\|W^G\|_1^{1/2}x_1,\|W^G\|_1^{1/2} x_2\right) &\text{if } 0\leq x_1,x_2\leq\|W^G\|_1^{-1/2}, \text{ and}\\
0 &\text{otherwise}.
\end{cases}
\]
Note that in the case of a simple graph $G$, each node in $V(G)$ corresponds
to an interval of length $1/|V(G)|$ in the canonical graphon $\W^G$, while it
corresponds to an interval of length $1/\sqrt{2|E(G)|}$ in the stretched
canonical graphon.

It will sometimes be convenient to define stretched canonical graphons
for graphs with infinitely many vertices (but finitely%
\footnote{More generally, in the setting of weighted graphs, we can allow for
infinitely many  edges as long as $\|\beta(G)\|_1<\infty$.} many edges). Our
definition of $W^G$ makes no sense for simple graphs with infinitely many
vertices, because they cannot all be crammed into the unit interval. Instead,
given a finite or countably infinite  graph $G$ with vertex weights
$(\alpha_i)_{i\in V(G)}$ which do not necessarily sum to $1$ (and may even
sum to $\infty$), we define a graphon $\wt {\W}^G=(\wt W^G,\R_+)$ by setting
$\wt W^G(x,y)=\beta_{ij}(G)$ if $(x,y)\in I_i\times I_j$, and $\wt
W^G(x,y)=0$ if there exist no such pair $(i,j)\in V(G)\times V(G)$, with
$I_i$ being the interval $[a_{i-1}, a_i)$ where we assume the vertices of $G$
have been labeled $1,2,\dots$,  and $a_i=\sum_{1\leq k\leq i}\alpha_k$ for
$i=0,1,\dots$. The stretched canonical graphon will then be defined as the
graphon $\mcl W^{G,s}:=(W^{G,s},\R_+)$ with
\[ W^{G,s}(x_1,x_2):=
\wt W^{G}\left(\|\wt W^G\|_1^{1/2}x_1,\|\wt W^G\|_1^{1/2} x_2\right),
\]
a definition which can easily be seen to be equivalent to the previous one if
$G$ is a finite graph.

Alternatively, one can define a stretched graphon $G^s$ as a graphon over
$V(G)$ equipped with the measure $\wh\mu_G$, where
\[
\wh\mu_G(A)=\frac 1{\sqrt{\|\beta(G)\|_1}}
\sum_{i\in A}\alpha_i
\]
for any $A\subseteq V(G)$.  In the case where $\sum_i\alpha_i<\infty$, this
graphon is obtained from the graphon representing $G$ by rescaling the
probability measure
\[
\mu_G(A)=\frac 1{\sum_{i\in V(G)}\alpha_i}
\sum_{i\in A}\alpha_i
\]
to the measure $\wh\mu_G$, while the function $\beta(G)\colon V(G)\times
V(G)\to \R$ with $(i,j)\mapsto\beta_{ij}(G)$ is left untouched.

Note that any graphon with underlying measure space $\R_+$ can be
``stretched'' in the same way as $W^G$; in other words, given any graphon
$\W=(W,\R_+)$ we may define a graphon $(W^\phi,\R_+)$, where $\phi\colon
\R_+\to\R_+$ is defined to be the linear map such that $\|W^\phi\|_1=1$,
except when $\|W\|_1=0$, in which case we define the stretched graphon to be
$0$. But for graphons over general measure spaces, this rescaling is
ill-defined.  Instead, we consider a different, but related, notion of
rescaling, by rescaling the measure of the underlying space, a notion which
is the direct generalization of our definition of the stretched graphon
$G^s$.

\begin{definition}
\begin{itemize}
\item[(i)] For two graphons $\mcl W_i=(W_i,\scr{S}_i)$ with
    $\scr{S}_i=(S_i,\mcl S_i,\mu_i)$ for $i=1,2$, define the
    \emph{stretched cut metric} $\delta_\square^s$ by
\[
\delta_\square^s(\mcl W_1,\mcl W_2) :=
\delta_\square(\wh{\mcl W}_1,\wh{\mcl W}_2),
\]
where $\wh {\mcl W}_i:=(W_i,\wh{\scr{S}}_i)$ with $\wh{\scr
S}_i:=(S_i,\cS_i,\wh\mu_i)$ and $\wh\mu_i:=\|W_i\|_1^{-1/2}\mu_i$. (In
the particular case where $\|W_i\|_1=0$, we define $\wh {\mcl
W}_i:=(0,{\scr{S}}_i)$.) Identifying $G$ with the graphon ${\wt
{\W}}^G$ introduced above, this also defines the stretched distance between
two  graphs, or a graph and a graphon.
\item[(ii)] A sequence of graphs $(G_n)_{n\in\N}$ or graphons
    $(\W_n)_{n\in\N}$ is called \emph{convergent in the stretched cut
 metric} if they form a Cauchy sequence for this metric; they are called
\emph{convergent to a graphon} $\W$ \emph{for the stretched cut metric} if
$\delta^s_\square(G_n,\W)\to 0$ or $\delta^s_\square(\W_n,\W)\to 0$,
respectively.
\end{itemize}
\label{defn1a}
\end{definition}

Note that for the case of graphons over $\R_+$, the above notion of
convergence is equivalent to the one involving the stretched graphons
$\W_i^s=(W_i^s,\R_+)$ of $\W_i$ defined by
\[
W_i^{s}(x_1,x_2):=W_i\left(\|W_i\|_1^{1/2}x_1,\|W_i\|_1^{1/2} x_2\right).
\]
To see this, just note that by the obvious coupling between $\lambda$ and
$\wh{\mu_i}$, where in this case $\wh{\mu_i}$ is a constant multiple of
Lebesgue measure, we have $\delta_\square(\W_i^s,\wh{\W}_i)=0$, and hence
$\delta_\square^s(\W_1,\W_2)=\delta_\square(\W_1^s,\W_2^s)$. As a
consequence, we have in particular that
$\delta_\square^s(G,G')=\delta_\square(G^s,(G')^s)=\delta_\square(\W^{G,s},\W^{G',s})$
for any two graphs $G$ and $G'$. Note also that the stretched cut metric does
not distinguish two graphs obtained from each other by deleting isolated
vertices, in the sense that
\begin{equation}
\label{iso-ver-inv}
\delta_\square^s(G,G')=0
\end{equation}
whenever $G$ is obtained from $G'$ by removing a set of  vertices that have
degree $0$ in $G'$.

The following basic example illustrates the difference between the notions of
convergence in the classical theory of graphons, the approach for sparse
graphs taken by \citet*{br-09} and \citet*{lp1}, and the approach of the
current paper. Proposition~\ref{prop13} below makes this comparison more
general.

\begin{examp}
Let $\alpha\in(0,1)$. For any $n\in\N$ let $G_n$ be an Erd\H{o}s-R\'{e}nyi
graph on $n$ vertices with parameter $n^{\alpha-1}$; i.e., each two vertices
of the graph are connected independently with probability $n^{\alpha-1}$. Let
$\wt G_n$ be a simple graph on $n$ vertices, such that $\lfloor
n^{(1+\alpha)/2}\rfloor$ vertices form a complete subgraph, and $n-\lfloor
n^{(1+\alpha)/2}\rfloor$ vertices are isolated. Both graph sequences are
sparse, and hence their canonical graphons converge to the trivial graphon
for which $W\equiv 0$, i.e., $\delta_\square(\mcl
W^{G_n},0),\delta_\square(\mcl W^{\wt G_n},0)\rta 0$, where we let 0 denote
the mentioned trivial graphon. The sequence $(G_n)_{n\in\N}$ converges to
$\W_1:=(\1_{[0,1]^2},[0,1])$ with the notion of convergence introduced by
\citet*{br-09} and \citet*{lp1}, but does not converge for
$\delta_\square^s$. The sequence $(\wt G_n)_{n\in\N}$ converges to $\W_1$ for
the stretched cut metric, i.e., $\delta^s_\square(\wt
G_n,\W_1)=\delta_\square(\W^{\wt G_n,s},\W_1)\rta 0$, but it does not
converge with the notion of convergence studied by \citet*{br-09} and
\citet*{lp1}. \label{ex1}
\end{examp}

The sequence $(\wt G_n)_{n\in\N}$ defined above illustrates one of our
motivations to introduce the stretched cut metric. One might argue that this
sequence of graphs should converge to the same limit as a sequence of
complete graphs; however, earlier theories for graph convergence are too
sensitive to isolated vertices or vertices with very low degree to
accommodate this.

The space of all $[0,1]$-valued graphons over $[0,1]$ is compact under the
cut metric \citep*{ls-analyst}. This implies that every sequence of simple
graphs is subsequentially convergent to some graphon under $\delta_\square$,
when we identify a graph $G$ with its canonical graphon $\W^G$. Our
generalized definition of a graphon, along with the introduction of the
stretched canonical graphon $\W^{G,s}$ and the stretched cut metric
$\delta_\square^s$, raises the question of whether a similar result holds in
this setting. We will see in Theorem~\ref{prop1} and Corollary
\ref{cor:subseq-conv} below that the answer is yes, provided we restrict
ourselves to uniformly bounded graphons and impose a suitable regularity
condition; see Definition~\ref{defn2}. The sequence $(G_n)_{n\in\N}$ in
Example~\ref{ex1} illustrates that we may not have subsequential convergence
when this regularity condition is not satisfied.

\begin{definition}
Let $\wt{\scr W}$ be a set of uniformly bounded graphons. We say that
$\wt{\scr W}$ has \emph{uniformly regular tails} if for every $\ep>0$ we can
find an $M>0$ such that for every $\mcl W=(W,\scr{S})\in\wt{\scr W}$ with
$\scr S=(S,\cS,\mu)$, there exists $U\in\cS$ such that $\|W-W\1_{U\times
U}\|_1<\ep$ and $\mu(U)\leq M$. A set $\mcl G$ of graphs has \emph{uniformly
regular tails} if $\|\beta(G)\|_1<\infty$ for all $G\in \mcl G$ and the
corresponding set of stretched canonical graphons $\{\mcl
W^{G,s}\,:\,G\in\mcl G\}$ has uniformly regular tails. \label{defn2}
\end{definition}

\begin{remark}
It is immediate from the definition that a set of simple graphs $\mcl G$ has
uniformly regular tails if and only if for each $\eps>0$ we can find $M>0$
such that the following holds. For all $G\in\mcl G$, assuming the vertices of
$G$ are labeled by degree (from largest to smallest) with ties resolved in an
arbitrary way,
\[
\sum_{i\leq \lceil M\sqrt{ |E(G)|} \rceil} \op{deg}(i;G) \leq
\eps|E(G)|.
\]
\end{remark}

In Lemma~\ref{prop26} in Appendix~\ref{sec8} we will prove that for a set of
graphs with uniformly regular tails we may assume the sets $U$ in the above
definition correspond to sets of vertices. Note that if a collection
$\wt{\scr W}$ of graphons has uniformly regular tails, then every collection
of graphons which can be derived from $\wt{\scr W}$ by adding a finite number
of the graphons to $\wt{\scr W}$ will still have uniformly regular tails. In
other words, if $\wt{\scr W},M,\ep$ are such that the conditions of
Definition~\ref{defn2} are satisfied for all but finitely many graphons in
$\wt{\scr W}$, then the collection $\wt{\scr W}$ has uniformly regular tails.

The following theorem shows that a necessary and sufficient condition for
subsequential convergence is the existence of a subsequence with uniformly
regular tails.

\begin{theorem}\label{prop1}
Every sequence $(\W_n)_{n\in\N}$ of uniformly bounded graphons with uniformly
regular tails converges subsequentially to some graphon $\W$ for the cut
metric $\delta_\square$. Moreover, if $\W_n$ is non-negative then every
$\delta_\square$-Cauchy sequence of uniformly bounded, non-negative graphons
has uniformly regular tails.
\end{theorem}

The proof of the theorem will be given in Appendix~\ref{sec:compact}. The
most challenging part of the proof is to show that uniform regularity of
tails implies subsequential convergence. We prove in Lemma~\ref{prop25} that
the property of having uniformly regular tails is invariant under certain
operations, which allows us to prove subsequential convergence similarly as
in the setting of dense graphs, i.e., by approximating the graphons by step
functions and using a martingale convergence theorem.

Two immediate corollaries of Theorem~\ref{prop1} are the following results.

\begin{corollary}
\label{cor:complete} The set of all $[0,1]$-valued graphons is complete for
the cut metric $\delta_\square$, and hence also for $\delta_\square^s$.
\end{corollary}

\begin{corollary}
\label{cor:subseq-conv} Let $(G_n)_{n\in\N}$ be a sequence of finite graphs
with non-negative, uniformly bounded edge weights such that
$|E(G_n)|<\infty$ for each $n\in\N$. Then the following hold:
\begin{enumerate}
\item[(i)] If $(G_n)_{n\in\N}$ has uniformly regular tails, then
    $(G_n)_{n\in\N}$ has a subsequence that converges to some graphon $\W$
    in the stretched cut metric.
\item[(ii)] If $(G_n)_{n\in\N}$ is a $\delta^s_\square$-Cauchy sequence,
    then it has uniformly regular tails.
\item[(iii)] If $(G_n)_{n\in\N}$ is a $\delta^s_\square$-Cauchy sequence,
    then it converges to some graphon $\W$  in the stretched cut metric.
\end{enumerate}
\end{corollary}

The former of the above corollaries makes two assumptions: (i) the
graphons are uniformly bounded, and (ii) the graphons are non-negative. We
remark that both of these conditions are necessary.

\begin{remark}
The set of all $\R_+$-valued graphons is not complete for the cut metric
$\delta_\square$; see for example the argument of \citet*[Proposition
2.12(b)]{lp1} for a counterexample. The set of all $[-1,1]$-valued graphons
is also not complete, as the following example suggested to us by Svante
Janson illustrates. For each $n\in\N$ let $\cV_n=(V_n,\R_+)$ be a
$\{-1,1\}$-valued graphon supported in $[n-1,n]^2$ satisfying
$\|V_n\|_\square<2^{-n}$ and $\|V_n\|_1=1$, by defining $\cV_n$ to be an
appropriately rescaled version of a graphon for a sufficiently large
Erd\H{o}s-R\'{e}nyi random graph with edge density $1/2$. Define
$\W_n=(W_n,\R_+)$ by $W_n:=\sum_{k=1}^{n} V_k$, and assume there is a graphon
$\W=(W,\R_+)$ such that $\lim_{n\rta\infty}\delta_\square(\W,\W_n)=0$. Then
we can find a sequence of measure-preserving transformations
$(\phi_n)_{n\in\N}$ with $\phi_n\colon \R_+\to\R_+$, such that
$\lim_{n\rta\infty}\|W^{\phi_n}-W_n\|_\square=0$. This implies that
$\lim_{n\rta\infty}\|W^{\phi_n}\1_{[k-1,k]^2}-V_k\|_\square=0$ for each
$k\in\N$. Since $V_k$ is a graphon associated with an Erd\H{o}s-R\'{e}nyi
random graph it is a step graphon. For any intervals $I,J\subseteq\R_+$ such
that $V_k|_{I\times J}=1$ or  $V_k|_{I\times J}=-1$ we have
$\lim_{n\rta\infty}\|(W^{\phi_n}\1_{[k-1,k]^2}-V_k)\1_{I\times
J}\|_\square=0$, so since $W$ takes values in $[-1,1]$ we have
$\lim_{n\rta\infty}\|W^{\phi_n}\|_{L^1(I\times J)}=\|V_k\|_{L^1(I\times J)}$.
Since $\|V_k\|_{L^1([k-1,k]^2)}=1$ this implies that
$\lim_{n\rta\infty}\|W^{\phi_n}\|_{L^1([k-1,k]^2)}= 1$. We have obtained a
contradiction to the assumption that $\W$ is a graphon, since for each
$n\in\N$ we have $\|W\|_1\geq \sum_{k=1}^{\infty}
\|W^{\phi_n}\|_{L^1([k-1,k]^2)}$.
\end{remark}

\begin{remark}
For comparison, \citet*[Theorem~5.1]{ls-analyst} proved that $[0,1]$-valued
graphons on the probability space $[0,1]$ (and hence any probability space)
form a compact metric space under $\delta_\square$.  Compactness fails in our
setting, because convergence requires uniformly regular tails, but
completeness still holds.
\end{remark}

Our next result compares the theory of graph convergence developed by
\citet*{lp1,lp2} with the theory developed in this paper. First we will
define the \emph{rescaled cut metric} $\delta_\square^r$. A sequence of
graphs is convergent in the sense considered by \citet*{lp1,lp2} iff it
converges for this metric.  For two graphons $\mcl W_1=(W_1,\scr{S}_1)$ and
$\mcl W_2=(W_2,\scr{S}_2)$, where $\scr{S}_1$ and $\scr{S}_2$ are measure
spaces of the same total measure, define $\wt W_1:=\|W_1\|_1^{-1} W_1$, $\wt
W_2:=\|W_2\|_1^{-1} W_2$, and
\[
\delta_\square^r(\mcl W_1,\mcl W_2) := \inf_\mu \|\wt W_1^{\pi_1}-\wt W_2^{\pi_2}\|_{\square,S_1 \times S_2,\mu},
\]
where we take the infimum over all measures $\mu$ on $S_1\times S_2$ with
marginals $\mu_1$ and $\mu_2$, respectively. For any graphs $G$ and $G'$ we
let $\mcl W^G$ and $\mcl W^{G'}$, respectively, denote the canonical graphons
associated with $G$ and $G'$, and for any graphon $\mcl W$ we define
\[
\delta_\square^r(G,\mcl W) :=  \delta_\square^r(\mcl W^G,\mcl W),\qquad
\delta_\square^r(G,G') :=  \delta_\square^r(\mcl W^G,\mcl W^{G'}).
\]

For the notion of convergence studied by \citet*{lp1}, \emph{uniform upper
regularity} plays a similar role to that of regularity of tails in the
current paper. More precisely, subsequential uniform upper regularity for a
sequence of graphs or graphons defined over a probability space is equivalent
to subsequential convergence to a graphon for the metric $\delta_\square^r$
\citep*[Appendix~C]{lp1}.  The primary conceptual
difference is that the analogue of Corollary~\ref{cor:complete} does not hold
in the theory studied by \citet*{lp1}.

We will now define what it means for a sequence of graphs or graphons to be
uniformly upper regular. A \emph{partition} of a measurable space $(S,\cS)$
is a finite collection $\mcl P$ of disjoint elements of $\cS$ with union $S$.
For any graphon $\mcl W=(W,\scr{S})$ with $\scr{S}=(S,\cS,\mu)$ and a
partition $\mcl P$ of $(S,\cS)$ into parts of nonzero measure, define $\mcl
W_{\mcl P}$ by averaging $W$ over the partitions. More precisely, if $\mcl
P=\{I_i\,:\,i=1,\dots,m\}$ for some $m\in\N$, define $\mcl W_{\mcl
P}:=((W)_{\mcl P},\scr{S})$, where
\[
(W_{\mcl P})(x_1,x_2):= \frac{1}{\mu(I_i)\mu(I_j)} \int_{I_i\times I_j}
W(x'_1,x'_2)\,dx'_1\,dx'_2\qquad \text{if } (x_1,x_2)\in I_i\times I_j.
\]
A sequence $(\mcl W_n)_{n\in\N}$ of graphons $\mcl W_n=(W_n,\scr S_n)$ over
probability spaces $\scr{S}_n=(S_n,\cS_n,\mu_n)$ is \emph{uniformly upper
regular} if there exists a function $K\colon (0,\infty)\to(0,\infty)$ and a
sequence $\{\eta_n\}_{n\in\N}$ of positive real numbers converging to zero,
such that for every $\ep>0$, $n\in\N$, and partition $\mcl P$ of $S_n$ such
that the $\mu_n$-measure of each part is at least $\eta_n$, we have
\[
\|(W_n)_{\mcl P}\1_{|(W_n)_{\mcl P}|\geq K(\ep)}\|_1\leq \ep.
\]
For any graph $G$ define the \emph{rescaled canonical graphon} $\mcl
W^{G,r}=(W^{G,r},[0,1])$ of $G$ to be equal to the canonical graphon $\mcl
W^G$ of $G$, except that we rescale the graphon such that $\|W^{G,r}\|_1=1$.
More precisely, we define $\W^{G,r}:=(W^{G,r},[0,1])$ with
$W^{G,r}:=\|W^G\|^{-1}_1W^G$. We say that a sequence of graphs
$(G_n)_{n\in\N}$ is uniformly upper regular if $(\W^{G_n,r})_{n\in\N}$ is
uniformly upper regular, where we only consider partitions $\mcl P$
corresponding to partitions of $V(G_n)$, and we require every vertex of $G_n$
to have weight less than a fraction $\eta_n$ of the total weight of $V(G_n)$.

The following proposition, which will be proved in Appendix~\ref{sec8},
illustrates the very different nature of the sparse graphs studied by
\citet*{lp1,lp2} and the graphs studied in this paper.

\begin{proposition}
Let $(G_n)_{n\in\N}$ be a sequence of simple graphs satisfying
$|V(G_n)|<\infty$ for each $n\in\N$.
\begin{itemize}
\item[(i)] If $(G_n)_{n\in\N}$ is sparse it cannot both be uniformly upper
    regular and have uniformly regular tails; hence it cannot converge for
    both metrics $\delta_\square^s$ and $\delta_\square^r$ if it is sparse.
\item[(ii)] Assume $(G_n)_{n\in\N}$ is dense and has convergent edge
    density. Then $(G_n)_{n\in\N}$ is a Cauchy sequence for
    $\delta_\square^s$ iff it is a Cauchy sequence for $\delta_\square^r$.
    If we do not assume convergence of the edge density, being a Cauchy
    sequence for $\delta_\square^r$ (resp.\ $\delta_\square^s$) does not
    imply being a Cauchy sequence for $\delta_\square^s$ (resp.\
    $\delta_\square^r$).
\end{itemize}
\label{prop13}
\end{proposition}

Many natural properties of graphons are continuous under the cut metric, for
example certain properties related to the degrees of the vertices. For
graphons defined on probability spaces it was shown by
\citet*[Section~2.6]{w-estimation} that the appropriately normalized degree
distribution is continuous under the cut metric. A similar result holds in
our setting, but the normalization is slightly different: instead of the
proportion of vertices whose degrees are at least $\lambda$ times the average
degree, we will consider a normalization in terms of the square root of the
number of edges. Given a graph $G$ and vertex $v \in V(G)$, let $d_G(v)$
denote the degree of $v$, and given a graphon $\W=(W,(S,\cS,\mu))$, define
the analogous function $D_W \colon S \to \R$ by
\[
D_W(x) = \int_S W(x,y) \, d\mu(y).
\]
The following proposition is an immediate consequence of Lemma~\ref{lem:D_W}
in Appendix~\ref{sec4},
which compares the functions $D_{W_1}$ and $D_{W_2}$ for graphons that are
close in the cut metric.

\begin{proposition}
\label{pro:deg-conv} Let $\W_n=(W_n,(S_n,\cS_n,\mu_n))$ be a sequence of
graphons that converge to a graphon $\W=(W,(S,\cS,\mu))$ in the cut metric
$\delta_\square$, and let $\lambda>0$ be a point where the function
$\lambda\mapsto \mu(\{D_W>\lambda\})$ is continuous. Then
$\mu_n(\{D_{W_n}>\lambda\})\to \mu(\{D_W>\lambda\})$. In particular,
\[
\frac 1{\sqrt{2|E(G_n)|}}\left|\left\{v\in V(G_n)\colon d_{G_n}(v)> {\lambda}\sqrt{2|E(G_n)|}\right\} \right|
\to \mu(\{D_{W^s}>\lambda\})
\]
whenever  $G_n$ is a sequence of finite simple graphs that converge to a
graphon $\W^s$ in the stretched cut metric and $\mu(\{D_{W^s}>\lambda\})$ is
continuous at $\lambda$.
\end{proposition}

Our final result in this section, which will be proved in
Appendix~\ref{sec8}, is that graphs which converge for the stretched cut
metric have unbounded average degree under certain assumptions, a result
which also holds for graphs that converge under the rescaled cut metric
\citep*[Proposition~C.15]{lp1}.

\begin{proposition}
Let $(G_n)_{n\in\N}$ be a sequence of finite simple graphs such that the
number of isolated vertices in $G_n$ is $o(|E(G_n)|)$ and such that
$\lim_{n\rta\infty}|E(G_n)|=\infty$. If  there is a graphon $\W$ such that
$\lim_{n\rta\infty}\delta^s_\square(G_n,\W)=0$, then $(G_n)_{n\in\N}$ has
unbounded average degree. \label{prop:unbdd_avg_deg}
\end{proposition}

The proof of the proposition proceeds by showing that graphs with bounded
average degree and a divergent number of edges cannot have uniformly regular
tails.

\subsection{Random Graph Models}
\label{sec:W-random-results}

In this section we will present two random graph models associated with a
given $[0,1]$-valued graphon $\mcl W=(W,\scr{S})$ with $\scr S=(S,\cS,\mu)$.

Before defining these models, we introduce some notation.  In particular, we
will introduce the notion of a graph process, defined as a stochastic process
taking values in the set of labeled graphs with finitely many edges and
countably many vertices, equipped with a suitable $\sigma$-algebra.
Explicitly, consider a family of graphs ${\mcl G}=( G_t)_{t\geq 0}$, where
the vertices have labels in $\N$. Let $\mathbb G$ denote the set of simple
graphs  with finitely many edges and countably many vertices, such that the
vertices have distinct labels in $\N$. Observe that a graph in this space can
be identified with an element of $\{0,1 \}^{\N\cup\binom{\N}{2}}$.  We equip
$\{0,1 \}^{\N\cup\binom{\N}{2}}$ with the product topology and $\mathbb G$
with the subspace topology $\mathbb T$.  Recall that a stochastic process is
c\`{a}dl\`{a}g if it is right-continuous with a left limit at every point.
Observe that the topological space $(\mathbb G, \mathbb T)$ is Hausdorff,
which implies that a convergent sequence of graphs has a unique limit. The
$\sigma$-algebra on $\mathbb G$ is the Borel $\sigma$-algebra induced by
$\mathbb T$.

\begin{definition}
\label{def:graph-process} A \emph{graph process} is a c\`{a}dl\`{a}g
stochastic process ${\mcl G}=( G_t)_{t\geq 0}$ taking values in the space of
graphs $\mathbb G$ equipped with the topology $\mathbb T$ defined above. The
process is called \emph{projective} if for all $s<t$,  $G_s$ is an induced
subgraph of $G_t$.
\end{definition}

We now define the graphon process already described informally in the
introduction. Sample a Poisson random measure $\mcl V$ on $\BB R_+\times S$
with intensity given by $\lambda\times\mu$ (see the book of \citel{cinlar-book},
Chapter~VI, Theorem~2.15), and identify $\mcl V$ with the collection of
points $(t,x)$ at which $\mcl V$ has a point
mass.%
\footnote{We see that this collection of points exists by observing that for
any measurable set $A\subset \R_+\times S$ of finite measure, we may sample
$\{ (t,x)\in\mcl V\cap A\}$ by first sampling the total number of points
$N_A\in\N\cup\{0 \}$ in the set (which is a Poisson random variable with
parameter $\mu(A)$), and then sampling $N_A$ points independently at random
from $A$ using the measure $\mu|_A$ renormalized to be a probability measure.
Note that our Poisson random measure is not necessarily a random counting
measure as defined for example by \citet*{cinlar-book}, since in general, not
all singletons $(t,x)$ are measurable, unless we assume that the singletons
$\{x\}$ in $S$ are measurable.} Let $\wt G$ be a graph with vertex set $\mcl
V$, such that for each pair of vertices $v_1=(t_1,x_1)$ and $v_2=(t_2,x_2)$
with $v_1\neq v_2$, there is an edge between $v_1$ and $v_2$ with probability
$W(x_1,x_2)$, independently for any two $v_1,v_2$. Note that $\wt G$ is a
graph with countably infinitely many vertices, and that the set of edges is
also countably infinite except if $W$ is equal to 0 almost everywhere. For
each $t\geq 0$ let $\wt G_t$ be the induced subgraph of $\wt G$ consisting
only of the vertices $(t',x)$ for which $t'\leq t$. Finally define $G_t$ to
be the induced subgraph of $\wt G_t$ consisting only of the vertices having
degree at least one. While $\wt G_t$ is a graph on infinitely many vertices
if $\mu(S)=\infty$, it has finitely many edges almost surely, and thus $G_t$
is a graph with finitely many vertices. We view the graphs $G_t$ and $\wt
G_t$ as elements of $\mathbb G$ by enumerating the points of $\mcl V$ in an
arbitrary but fixed way.

When $\mu(S)<\infty$ the set of graphs $\{\wt G_t\,:\,t\geq 0\}$ considered
above is identical in law to a sequence of $\mcl W$-random graphs as defined
by \citet*{ls-graphlimits} for graphons over $[0,1]$ and, for example, by
\citet*{bjr07} for graphons over general probability spaces. More precisely,
defining a stopping time $t_n$ as the first time when $|V(\wt G_t)|=n$ and
relabeling the vertices in $V(\wt G_{t_n})$ by labels in $[n]$, we have that
the sequence $\{\wt G_{t_n}\,:\,n\in \N\}$ has the same distribution as the
sequence of random graphs generated from $\mcl W$, except for the fact that
$\mu$ should be replaced by the probability measure $\wt\mu=\frac
1{\mu(S)}\mu$, a fact which follows immediately from the observation that a
Poisson process with intensity $t\mu$ conditioned on having $n$ points is
just a distribution of $n$ points chosen i.i.d.\  from the distribution $\wt
\mu $. In the case when $\mu(S)=\infty$ it is primarily the graphs $G_t$
(rather than $\wt G_t$) which are of interest for applications, since the
graphs $\wt G_t$ have infinitely many (isolated) vertices. But from a
mathematical point of view, both turn out to be useful.

\begin{definition}
Two graph processes $(G_t^1)_{t\geq 0}$ and $(G_t^2)_{t\geq 0}$ are said to
be \emph{equal up to relabeling of the vertices} if there is a bijection
$\phi\colon \bigcup_{t\geq 0} V( G_t^1)\to \bigcup_{t\geq 0}V(G_t^2)$ such
that $\phi(G_t^1)=G_t^2$ for all $t\geq 0$, where $\phi(G_t^1)$ is the graph
whose vertex and edge sets are $\{\phi(i)\}_{i\in V(G_t)}$ and
$\{\phi(i)\phi(j)\}_{ij\in E(G_t)}$, respectively. Two graph processes
$(G_t^1)_{t\geq 0}$ and $(G_t^2)_{t\geq 0}$ are said to be \emph{equal in law
up to relabeling of the vertices} if they can be coupled in such a way that
a.s., the two families are equal up to relabeling of the vertices.
\label{def:relab}
\end{definition}

Note that in order for the notion of ``equal in law up to relabeling of the
vertices'' to be well defined, one needs to show that the event that two
graph processes $(G_t)_{t\geq 0}$ and $(\wt G_t)_{t\geq 0}$ are equal up to
relabeling is measurable. The proof of this fact is somewhat technical and
will be given in Appendix~\ref{sec:measurability}.

\begin{definition}
Let $\mcl W=(W,\scr{S})$ be a $[0,1]$-valued graphon. Define $\wt{\mcl
G}(\mcl W)=(\wt G_t(\mcl W))_{t\geq 0}$ (resp.\ $\mcl G(\mcl W)=( G_t(\mcl
W))_{t\geq 0}$) to be a random family of graphs with the same law as the
graphs $(\wt G_t)_{t\geq 0}$ (resp.\ $(G_t)_{t\geq 0}$) defined above.
\begin{itemize}
\item[(i)] A random family of simple graphs is called a \emph{graphon
    process  without isolated vertices} generated by $\W$ if it has the
    same law as $\mcl G(\mcl W)$ up to relabeling of the vertices, and it
    is called a \emph{graphon process with isolated vertices} generated by
    $\W$ if it has the same law as $\wt{\mcl G}(\mcl W)$ up to relabeling
    of the vertices.
\item[(ii)] A random family $\wt{\mcl G}=(\wt G_t)_{t\geq 0}$ of simple
    graphs is called a \emph{graphon process} if there exists a graphon
    $\W$ such that after removal of all isolated vertices, $\wt{\mcl G}$
    has the same law as $\mcl G(\W)$ up to relabeling of the vertices.
\end{itemize}
If $\mcl G=(G_t)_{t\geq 0}$ is a graphon process, then we refer to $G_t$ as
the \emph{graphon process at time $t$}. \label{defn3}
\end{definition}

Given a graphon $\W=(W,\scr{S})$ one can define multiple other natural random
graph models; see below. However, the graph models of Definition~\ref{defn3}
have one property which sets them apart from these models: exchangeability.
To formulate this, we first recall that a random measure $\xi$ in the first
quadrant $\R_+^2$ is jointly exchangeable iff for every $h>0$, permutation
$\sigma$ of $\N$, and $i,j\in\N$,
\[
\xi(I_i\times I_j) \eqD \xi(I_{\sigma(i)}\times I_{\sigma(j)}),\qquad \text{where } I_k:=[h(k-1),hk].
\]
Here $\eqD$ means equality in distribution, and, as usual, a random measure
on $\R_+^2$ is a measure drawn from some probability distribution over the
set of all Borel measures on $\R_+^2$, equipped with the minimal
$\sigma$-algebra for which the functions $\mu\mapsto \mu(B)$ are measurable
for all Borel sets $B$.

To relate this notion of exchangeability to a property of a graphon process,
we will assign a random measure $\xi({\mcl G})$ to an arbitrary projective
graph process ${\mcl G}=(G_t)_{t\geq 0}$.  Defining the birth time $t_v$ of a
vertex $v\in V({\mcl  G})$ as the infimum over all times $t$ such that $v\in
V( G_t)$, we  define a random measure $\xi=\xi({\mcl G})$ on $\R_+^2$ by
\begin{equation}
\xi(\mcl G):=\sum_{(u,v)\in E({\mcl G})} \delta_{(t_u,t_v)},
\label{eq5}
\end{equation}
where each edge $(u,v)=(v,u)$ is counted twice so that the measure is
symmetric. If $\mcl G$ is a graphon process with isolated vertices, i.e.,
$\mcl G=\wt{\mcl G}(\mcl W)$ for some graphon $\mcl W$, it is easy to see
that at any given time, at most one vertex is born, and that at time $t=0$,
$G_t$ is empty. In other words,
\begin{equation}
V(G_0)=\emptyset
\quad\text{ and }\quad
|V(G_{t})\backslash V( G_{t^-})|\leq 1 \text{ for all } t>0.
\label{eq19}
\end{equation}

It is not that hard to check that the measure $\xi$ is jointly exchangeable
if ${\mcl G}$ is a graphon process with isolated vertices%
\footnote{This is one of the instances in which the family $\wt{\mcl G}(\W)$
is more useful that the family ${\mcl G}(\W)$: the latter only contains
information about when a vertex first appeared in an edge in $\wt{\mcl
G}(\W)$, and not information about when it was born.} generated from some
graphon $\W$. But it turns out that the converse is true as well, provided
the sequence has uniformly regular tails. The following theorem will be
proved in Appendix~\ref{sec6}, and as with \citet*{caron-fox} we will rely on the
Kallenberg theorem for jointly exchangeable measures \cite[Theorem~9.24]{kallenberg-exch}
for this description. \citet*{veitchroy} have
independently formulated and proved a similar theorem, except that their
version does not include integrability of the graphon or uniform tail
regularity of the sequence of random graphs.

Before stating our theorem, we note that given a locally finite symmetric
measure $\xi$ that is a countable sum of off-diagonal, distinct atoms of
weight one in the interior of $\R_+^2$, we can always find a projective
family of simple graphs $ G_t$ obeying the condition \eqref{eq19} and the
other assumptions we make above, and that up to vertices which stay isolated
for all times, this family of graphs is uniquely determined by $\xi$ up to
relabeling of the vertices. Any projective family of countable simple graphs
$\mcl G$ with finitely many edges at any given time can be transformed into
one obeying the condition \eqref{eq19} (by letting the vertices appear in the
graph $G_t$ exactly at the time they were born and merging vertices born at
the same time, and then labeling vertices by their birth time), provided the
measure $\xi({\mcl G})$ has only point masses of weight one, and has no
points on the diagonal and the coordinate axes.

\begin{theorem}
Let $\wt{\mcl G}=(\wt G_t)_{t\geq 0}$ be a projective family of random
simple graphs which satisfy \eqref{eq19}, and define $\xi=\xi(\wt{\mcl G})$
by \eqref{eq5}. Then the following two conditions are equivalent:
\begin{itemize}
\item[(i)] The measure $\xi$ is a jointly exchangeable random measure and
    $(\wt G_t)_{t\geq 0}$ has uniformly regular tails.
\item[(ii)] There is a $\R_+$-valued random variable $\alpha$ such that
    $\mcl W_\alpha=(W_\alpha,\R_+)$ is a $[0,1]$-valued graphon almost
    surely, and such that  conditioned on $\alpha$, $(\wt G_t)_{t\geq 0}$
    (modulo vertices that are isolated for all $t\geq 0$) has the law of
    $\wt {\mcl G}(\W_\alpha)$ up to relabeling of the vertices.
\end{itemize}
\label{prop5}
\end{theorem}

Recall that we called two graphons equivalent if their distance in the cut
metric $\delta_\square$ is zero. The following theorem shows that this notion
of equivalence is the same as equivalence of the graphon process generated
from two graphons, in the sense that the resulting random graphs have the
same distribution. Note that in (ii) we only identify the law of $\wt G_t$ up
to vertices that are isolated for all times; it is clear that if we extend
the underlying measure space $\scr S$ and extend $\W$ trivially to this
measure space, the resulting graphon is equivalent to $\W$ and the law of the
graphs $G_t$ remains unchanged, while the law of $\wt G_t$ might change due
to additional isolated vertices.

\begin{theorem}
For $i=1,2$ let $\mcl W_i=(W_i,\scr{S}_i)$ be $[0,1]$-valued graphons, and
let $(\wt G^i_t)_{t\geq 0}$ and $( G^i_t)_{t\geq 0}$ be the graphon
processes generated from $\mcl W_i$ with and without, respectively, isolated
vertices. Then the following statements are equivalent:
\begin{itemize}
\item[(i)] $\delta_\square(\mcl W_1,\mcl W_2)=0$.
\item[(ii)] After removing all vertices which are isolated for all times,
    $(\wt G^1_t)_{t\geq 0}$ and $(\wt G^2_t)_{t\geq 0}$ are equal in law up
    to relabeling of the vertices.
\item[(iii)] $(G^1_t)_{t\geq 0}$ and $(G^2_t)_{t\geq 0}$ are equal in law
    up to relabeling of the vertices.
\end{itemize}
\label{prop11}
\end{theorem}

The theorem will be proved in Appendix~\ref{sec3}. We show that (i) implies
(ii) and (iii) by using Proposition~\ref{prop27}, which says that the infimum
in the definition of $\delta_\square$ is attained under certain assumptions
on the underlying graphons. We show that (ii) or (iii) imply (i) by using
Theorem~\ref{prop4}(i).

As indicated before, in addition to the graphon processes defined above,
there are several other natural random graph models generated from a graphon
$\W$. Consider a sequence of probability measures $(\mu_n)_{n\in\N}$ on
$(S,\cS)$, and construct a sequence of random graphs $G_n$ as follows. Start
with a single vertex $(1,x_1)$ with $x_1$ sampled from $\mu_1$.  In step $n$,
sample $x_n$ from $\mu_n$, independently from all vertices and edges sampled
so far,  and for each $i=1,\dots k$, add an edge between $(i,x_i)$ and
$(n,x_n)$ with probability $W(x_i,x_n)$, again independently for each $i$
(and independently of all vertices and edges chosen before). Alternatively,
sample an infinite sequence of independent features $x_1,x_2,\dots$
distributed according to $\mu_1,\mu_2,\dots$, and let $G$ be the graph on
infinitely many vertices with vertex set identified with $\{(n,x_n)\,:\,n\in
\N\}$, such that for any two $n_1,n_2\in\N$ there is an edge between
$(n_1,x_{n_1})$ and $(n_2,x_{n_2})$ independently with probability
$W(x_{n_1},x_{n_2})$. For each $n\in\N$ let $G_n$ be the induced subgraph of
$G$ consisting only of the vertices $(k,x_k)$ for which $k\leq n$.

It was proven by \citet*{denseconv1} that dense $\W$-random graphs generated
from graphons on probability spaces converge to $\W$.  The following theorem
generalizes this to graphon processes, as well as for the alternative model
defined in terms of a suitable sequence of measures $\mu_n$.

\begin{theorem}
Let $\mcl W=(W,\scr{S})$ with $\scr S=(S,\cS,\mu)$ be a $[0,1]$-valued
graphon. Then the following hold:
\begin{itemize}
\item[(i)]  Almost surely $\lim_{t\rta\infty}\delta^s_\square(\mcl W,\wt
    G_t(\mcl W))=0$ and $\lim_{t\rta\infty}\delta^s_\square(\mcl W,G_t(\mcl
    W))=0$.
\item[(ii)] Let $(G_n)_{n\in N}$ be the sequence of simple graphs generated
    from $\W$ with arrival probabilities $\mu_n:=\mu(S_n)^{-1}\mu|_{S_n}$
    as described above, where we assume $\bigcup_{n\in\N} S_n=S$,
    $S_n\subseteq S_{n+1}$, and $0<\mu(S_n)<\infty$ for all $n\in\N$, and
    $W$ is not equal to 0 almost everywhere. Then
    a.s.-$\lim_{n\rta\infty}\delta_\square^s(\mcl W,G_n)=0$ if and only if
    $\sum_{n=1}^\infty \mu(S_n)^{-1}=\infty$.
\end{itemize}
\label{prop4}
\end{theorem}

We will prove the theorem in Appendix~\ref{sec3}. Part (i) of the theorem is
proved by observing that for any set $A\subseteq S$ of finite measure, the
induced subgraph of $\wt G_t$ consisting of the vertices with feature in $A$
has the law of a graph generated from a graphon over a probability space.
This implies that we can use convergence results for dense graphs to conclude
the proof. In our proof of part (ii) we first show that the condition on
$S_n$ is necessary for convergence, by showing that otherwise $E(G_n)$ is
empty for all $n\in\N$ with positive probability. We show that the condition
on $S_n$ is sufficient by constructing a coupling of $(G_n)_{n\in\N}$ and a
graphon process $(\wt G_t)_{t\geq 0}$.

\subsection{Left Convergence}
\label{sec7}

Left convergence is a notion of convergence where we consider subgraph counts
of small test graphs. Existing literature defines left convergence both for
dense graphs \citep*{ls-graphlimits} and for bounded degree graphs
\citep*{left-right-conv}, with a different renormalization factor to adjust for
the difference in edge density. We will operate with a definition of subgraph
density with an intermediary renormalization factor, to take into account
that our graphon process satisfies $\omega(|V(G_t)|)=|E(G_t)|=O(|V(G_t)|^2)$.
For dense graphs our definition of left convergence coincides with the
standard definition in the theory of dense graphs.

For a simple graph $F$  and a simple graph $G$ define $\op{hom}(F,G)$ to be
the number of adjacency preserving maps $\phi\colon V(F)\to V(G)$, i.e., maps
$\phi$ such that if $(v_1,v_2)\in E(F)$, then $(\phi(v_1),\phi(v_2))\in
E(G)$, and define $\op{inj}(F,G)$ be the number of such maps that are
injective.

Define the \emph{rescaled homomorphism density} $h(F,G)$ and the \emph{
rescaled injective homomorphism density} $h_{\mathrm{inj}}(F,G)$ of $F$ in
$G$ by
\[
h(F,G) := \frac{\op{hom}(F,G)}{(2|E(G)|)^{|V(F)|/2}}
\quad\text{and}\quad
h_{\mathrm{inj}}(F,G) := \frac{\op{inj}(F,G)}{(2|E(G)|)^{|V(F)|/2}}.
\]
For any $[0,1]$-valued graphon $\mcl W=(W,\scr{S})$ we define the rescaled
homomorphism density of $F$ in $\mcl W$ by
\[
h(F,\mcl W) := \|W\|_{1}^{-|V(F)|/2} \int_{S^{|V(F)|}} \prod_{(i,j)\in E(F)} W(x_i,x_j)\,dx_1\dots dx_{|V(F)|}.
\]
Note that in general, $h(F,\W)$ need not be finite.  Take, for example,
$\W=(W,\R_+)$ to be a graphon of the form
\[
W(x,y) = \begin{cases} 1 & \textup{if $y\leq f(x)$, and}\\
0 & \textup{otherwise,}
\end{cases}
\]
where
\[
f(x) = \begin{cases} x^{-1/2} & \textup{if $0 \le x \le 1$, and}\\
x^{-2} & \textup{if $x \ge 1$.}
\end{cases}
\]
Let $D_W(x)=\int_{\R_+} W(x,y)\,dy$.  Then $D_W$ is in $L^1(\R_+)$, but not
in $L^k(\R_+)$ for any $k\geq 2$. Thus if $F$ is a star with $k\geq 2$
leaves, then $h(F,\mcl W) := \|W\|_{1}^{-(k+1)/2} \int_{\R_+}D_W^k(x)
dx=\infty$. Proposition~\ref{prop3}(ii) below, whose proof is based on
Lemma~\ref{lem:h-bd} in Appendix~\ref{sec5}, gives one criterion which
guarantees that $h(F,\mcl W)<\infty$ for all simple connected graphs $F$.

\begin{definition}
A sequence $(G_n)_{n\in\N}$ is left convergent if its edge density is
converging, and if for every simple connected graph $F$ with at least two
vertices, the limit $\lim_{n\rta\infty} h(F,G_n)$ exists and is finite.
Left convergence is defined similarly for a continuous-time family of graphs
$(G_t)_{t\geq 0}$.
\end{definition}

For dense graphs left convergence is equivalent to metric convergence
\citep*{denseconv1}. This equivalence does not hold for our graphs, but
convergence of subgraph densities (possibly with an infinite limit) does
hold for graphon processes.

\begin{proposition}
\begin{itemize}
\item[(i)] If $\mcl W=(W,\scr{S})$ is a $[0,1]$-valued graphon and
    $(G_t)_{t>0}$ is a graphon process, then for every simple connected
    graph $F$ with at least two vertices,
    \[
    \lim_{t\rta\infty}h_{\mathrm{inj}}(F,G_t)
    =h(F,\mcl W)\in[0,\infty]
    \] almost surely.
\item[(ii)] In the setting of (i), if $D_W(x):=\int_S W(x,x')\,d\mu(x')$ is
    in $L^p$ for all $p\in [1,\infty)$, then $h(F,\W)<\infty$ for every
    simple connected graph $F$ with at least two vertices and
\[
   \lim_{t\rta\infty}h(F,G_t)
   = \lim_{t\rta\infty}h_{\mathrm{inj}}(F,G_t)
   = h(F,\mcl W)
\]
almost surely, so in particular $(G_t)_{t>0}$ is left convergent.
\item[(iii)]  Assume $(G_n)_{n\in\N}$ is a sequence of simple graphs with
    bounded degree such that \[\lim_{n\rta\infty}|V(G_n)|=\infty\] and
    $E(G_n)\neq\emptyset$ for all sufficiently large $n$. Then
    $(G_n)_{n\in\N}$ is trivially left convergent, and $\lim_{n\rta\infty}
    h(F,G_n)=0$ for every connected $F$ for which $|V(F)|\geq 3$.
\item[(iv)] Left convergence does not imply convergence for
    $\delta_\square^s$, and convergence for $\delta_\square^s$ does not
    imply left convergence.
\end{itemize}
\label{prop3}
\end{proposition}

The proposition will be proved in Appendix~\ref{sec5}. Part (i) is immediate
from Proposition~\ref{prop2}, which is proved using martingale convergence
and that $\op{inj}(F,G_{-t})$ appropriately normalized evolves as a backwards
martingale. Part (ii) is proved by using that $h(F,G_t)$ and
$h_{\op{inj}}(F,G_t)$ are not too different under certain assumptions on the
underlying graphon. Part (iii) is proved by bounding $\op{hom}(F,G_n)$ from
above, and part (iv) is proved by constructing explicit counterexamples.

\begin{remark}
\label{rem:W-notin-L1} While we stated the above proposition for graphons,
i.e., for the case when $W\in L^1$, the main input used in the proof,
Proposition~\ref{prop2} below, does not require an integrable $W$, but just
the measurability of the function $W\colon S\times S\to [0,1]$.
\end{remark}

\acks{We thank Svante Janson for his careful reading of an earlier version of
the paper and for his numerous suggestions, and we thank Edoardo Airoldi for
helpful discussions. Holden was supported by an internship at Microsoft
Research New England and by a doctoral research fellowship from the Norwegian
Research Council.}

\appendix

\section{\texorpdfstring{Cut Metric and Invariant $L^p$ Metric}{Cut Metric and Invariant L\string^p Metric}}
\label{sec4}

The main goal of this appendix is to prove Proposition~\ref{prop8}, which
says that $\delta_\square$ and $\delta_1$ are well defined and pseudometrics.
In the course of our proof, we will actually generalize this proposition, and
show that it can be extended to the invariant $L^p$ metric $\delta_p$,
provided the two graphons are non-negative and in $L^p$.

We start by defining the distance $\delta_p(\mcl W_1,\mcl W_2)$ for two such
graphons $\mcl W_1=(W_1,\scr{S}_1)$ and $\mcl W_2=(W_1,\scr{S}_2)$ over two
spaces $\scr{S}_1=(S_1,\mcl S_1,\mu_i)$ and $\scr{S}_2=(S_2,\mcl S_2,\mu_2)$
of equal total measure, in which case we set
\[
\delta_{p}(\mcl W_1,\mcl W_2) := \inf_\mu \|W_1^{\pi_1}-W_2^{\pi_2}\|_{p,S_1 \times S_2,\mu},
\]
where, as before, $\pi_1$ and $\pi_2$ are the projections from $S_1\times
S_2$ to $S_1$ and $S_2$, respectively, and the infimum is over all couplings
$\mu$ of $\mu_1$ and $\mu_2$. If $\mu_1(S_1)\neq \mu_2(S_2)$ we define
$\delta_p(\W_1,\W_2)$ by trivially extending $\W_1$ and $\W_2$ to two
graphons $\wt\W_1$ and $\wt\W_2$, respectively, over measure spaces of equal
measure, and defining $\delta_{p}(\mcl W_1,\mcl W_2):=
\delta_{p}(\wt\W_1,\wt\W_2)$, just as in Definition~\ref{defn1} (ii).

\begin{proposition}
\label{prop:delta-p} For $i=1,2$, let $\mcl W_i=(W_i,\scr{S}_i)$ be
non-negative graphons over $\scr{S}_i=(S_i,\mcl S_i,\mu_i)$ with $W_i\in
L^p(S_i\times S_i)$  for some $p\in (1,\infty)$. Then $\delta_p(\mcl W_1,\mcl
W_2)$ is well defined. In particular, $\delta_p(\mcl W_1,\mcl W_2)$ does not
depend on the choice of extensions $\wt{\W}_1$ and $\wt{\W}_2$. Furthermore,
$\delta_p$ is a pseudometric on the space of non-negative graphons in $L^p$.
\end{proposition}

We will prove Proposition~\ref{prop:delta-p} at the same time as
Proposition~\ref{prop8}. We will also establish an estimate (Lemma
\ref{prop16}) saying that two graphons are close in the cut metric if we
obtain one from the other by slightly modifying the measure of the underlying
measure space. Finally we state and prove a lemma, Lemma~\ref{lem:D_W}, that
immediately implies Proposition~\ref{pro:deg-conv}.

The following lemma will be used in the proof of Propositions~\ref{prop7} and
\ref{prop10}. The analogous result for probability spaces can for example be
found in a paper by \citet*[Theorem A.7]{janson-survey}, and the extension to
$\sigma$-finite measure spaces is straightforward.

\begin{lemma}
Let $\scr S=(S,\cS,\mu)$ be an atomless $\sigma$-finite Borel space. Then
$\scr S$ is isomorphic to $([0,\mu(S)),\mcl B,\lambda)$, where $\mcl B$ is
the Borel $\sigma$-algebra and $\lambda$ is Lebesgue measure. \label{prop9}
\end{lemma}

\begin{proof}
For $\mu(S)<\infty$, this holds because every atomless Borel probability
space is isomorphic to $[0,1]$ equipped with the Borel $\sigma$-algebra and
Lebesgue measure \citep*[Theorem~A.7]{janson-survey}. For $\mu(S)=\infty$ we
use that by the hypotheses of $\sigma$-finiteness there exist disjoint sets
$S_k\in\cS$ for $k\in\N$ such that $S=\bigcup_{k=1}^\infty S_k$ and
$\mu(S_k)<\infty$ for all $k\in\N$. For each $k\in\N$ we can find
isomorphisms $\phi\colon [0,\mu(S_k)]\to [0,\mu(S_k))$ and $\wt{\phi}\colon
S_k\to [0,\mu(S_k)]$. It follows by considering the composed map
$\phi\circ\wt\phi$ that $S_k$ is isomorphic to $[0,\mu(S_k))$. The lemma
follows by constructing an isomorphism from $S$ to $\R_+$ where each set
$S_k$ is mapped onto an half-open interval of length $\mu(S_k)$.
\end{proof}

The first statement of Proposition~\ref{prop8}, i.e., the existence of a
coupling, follows directly from the following more general result.

\begin{lemma}
For $k=1,2$ let $\scr S_k=(S_k,\cS_k,\mu_k)$ be a $\sigma$-finite measure
space such that $\mu_1(S_1)=\mu_2(S_2)\in (0,\infty]$. Let $D_k\in\cS_k$, and
let $\wt\mu$ be a measure on the product space $D_1\times D_2$, where $D_k$
is equipped with the induced $\sigma$-algebra from $S_k$.  Assume the
marginals $\wt\mu_1$ and $\wt\mu_2$ of $\wt\mu$ are bounded above by
$\mu_1|_{D_1}$ and $\mu_2|_{D_2}$, respectively, and that either
$D_1=D_2=\emptyset$ or $\mu_k(S_k\setminus D_k)=\infty$ for $k=1,2$. Then
there exists a coupling $\mu$ of $\scr S_1$ and $\scr S_2$, such that
$\mu|_{D_1\times D_2}=\wt\mu$. \label{prop20}
\end{lemma}

\begin{proof}
First we consider the case when $D_1=D_2=\emptyset$. If
$\mu_1(S_1)=\mu_2(S_2)<\infty$ we define $\mu$ to be proportional to the
product measure of $\mu_1$ and $\mu_2$. Explicitly, for $A\in \mcl S_1$ and
$B\in \mcl S_2$, we set $\mu(A\times B)=\mu_1(A)\mu_2(B)/\mu_1(S_1)$. This
clearly gives $\mu(S_1\times B)=\mu_2(B)$ and $\mu(A\times S_2)=\mu_1(A)$, as
required.

If $\mu_1(S_1)=\mu_2(S_2)=\infty$, we consider partitions of $S_1$ and $S_2$
into disjoint sets of finite measure, with $S_k=\bigcup_{i \ge 1} A^k_i$ for
$k=1,2$. Let $I_1,I_2,\dots$ and $J_1,J_2,\dots$ be decompositions of
$[0,\infty)$ into adjacent intervals of lengths $\mu_1(A^1_1), \mu_1(A^1_2),
\dots$ and $\mu_2(A^2_1), \mu_2(A^2_2), \dots$, respectively. We then define
a measure $\mu$ on $(S_1\times S_2, \mcl S_1\times \mcl S_2)$ by
\[
\mu(A\times B)=
\sum_{i,j\geq 1}
\frac{\lambda(I_i\cap I_j)}{\lambda(I_i)\lambda(J_j)}
\mu_1(A\cap A^1_i)\mu_2(B\cap A^2_j),\quad \text{for } A\in\mcl S_1, B\in \mcl S_2.
\]
As a weighted sum of product measures, $\mu$ is a measure, and inserting
$A=S_1$ or $B=S_2$, one easily verifies that $\mu$ has marginals $\mu_1$ and
$\mu_2$. This completes the proof of the lemma in the case that
$D_1=D_2=\emptyset$.

Now we consider the general case. Decomposing $D_1$ and $D_2$ into disjoint
sets of finite mass with respect to $\mu_1$ and $\mu_2$, $D_k=\bigcup_{i \ge
1} D_i^k$ with $\mu_k(D_i^k)<\infty$, we define measures $\wh\mu_k^{(\ell)}$
on $S_k$ for $k,\ell=1,2$ by
\[
\begin{split}
\wh\mu^{(1)}_1(A)&=\frac 12\mu_1(A\cap (S_1\backslash D_1))\text{ for all } A\in \cS_1,\qquad\\
\wh\mu^{(1)}_2(B) &= \frac 12 \mu_2(B\cap (S_2\backslash D_2))
+\sum_{i\geq 1}\Bigl[\mu_2(B\cap D_i^2)-\wt\mu_2(B\cap D_i^2)\Bigr] \text{ for all } B\in \cS_2,\qquad \\
\wh\mu_1^{(2)}(A) &=\frac 12 \mu_1(A\cap (S_1\backslash D_1))
+\sum_{i\geq 1}\Bigl[\mu_1(A\cap D_i^1)-\wt\mu_1(A\cap D_i^1) \Bigr]\text{ for all } A\in \cS_1, \text{ and}\\
\wh\mu_2^{(2)}(S_1)&=\frac 12\mu_2(B\cap(S_2\backslash D_2))\text{ for all } B\in \cS_2.
\end{split}
\]
Note that $\wh\mu_1^{(\ell)}(S_1)=\wh\mu_2^{(\ell)}(S_2)=\infty$ for
$\ell=1,2$ by our assumption $\mu_k(S_k\backslash D_k)=\infty$ for $k=1,2$.
By the result for the case $D_1=D_2=\emptyset$, for $\ell=1,2$, we can find
couplings $\wh\mu^{(\ell)}$ of $\wh\mu_1^{(\ell)}$ and $\wh\mu_2^{(\ell)}$ on
$S_1\times S_2$. Extending the measure $\wt\mu$ to a measure on $S_1\times
S_2$ by assigning measure $0$ to all sets which have an empty intersection
with $D_1\times D_2$, the measure $\mu:=\wh\mu^{(1)}+\wh\mu^{(2)}+\wt\mu$ has
the appropriate marginals. To see that $\mu|_{D_1\times D_2}=\wt\mu$, we note
that $\wh\mu^{(1)}(D_1\times S_2)=\wh\mu^{(1)}_1(D_1)=0$ and
$\wh\mu^{(2)}(S_1\times D_2)=\wh\mu^{(2)}_2(D_2)=0$, implying in particular
that $\wh\mu^{(1)}(D_1\times D_2)=\wh\mu^{(2)}(D_1\times D_2)=0$.
\end{proof}

\begin{corollary}
\label{cor-to-prop20} For $k=1,2$ let $\scr S_k=(S_k,\cS_k,\mu_k)$ be a
$\sigma$-finite measure space such that $\mu_1(S_1)=\mu_2(S_2)\in
(0,\infty]$, and let $\mu$ be a coupling of $\mu_1$ and $\mu_2$. Let
$D_k\in\cS_k$ be such that $\mu(D_1\times (S_2\setminus
D_2))=\mu((S_1\setminus D_1)\times D_2)\in(0,\infty]$. Then there exists a
coupling $\wt\mu$ of $\mu_1$ and $\mu_2$ such that $\wt\mu$ is supported on
$(D_1\times D_2)\cup ((S_1\setminus D_1)\times (S_2\setminus D_2))$ and
$\wt\mu\geq\mu$ on $(D_1\times D_2)\cup ((S_1\setminus D_1)\times
(S_2\setminus D_2))$.
\end{corollary}

\begin{proof}
Let $\mu'$ be the restriction of $\mu$ to $(D_1\times D_2)\cup ((S_1\setminus
D_1)\times (S_2\setminus D_2))$, let $\mu'_1$ and $\mu_2'$ be its marginals,
and let $\delta_i=\mu_i-\mu'_i$. Then $\delta_1(D_1)=\mu(D_1\times
(S_2\setminus D_2))$ and $\delta_2(D_2)=\mu((S_1\setminus D_1)\times
D_2)=\delta_1(D_1)$ by the hypotheses of the corollary. In a similar way,
$\delta_1(S_1\setminus D_1) =\mu((S_1\setminus D_1)\times
D_2)=\delta_2(S_2\setminus D_2)$. With the help of the previous lemma, and
considering the domains $D_1\times D_2$ and $(S_1\setminus D_1)\times
(S_2\setminus D_2)$ separately, we then construct a coupling $\delta$ of
$\delta_1$ and $\delta_2$ that has support on $(D_1\times D_2)\cup
((S_1\setminus D_1)\times (S_2\setminus D_2))$. Setting
$\tilde\mu=\mu'+\delta$ we obtain the statement of the corollary.
\end{proof}

The following basic lemma will be used multiple times throughout this
appendix. The analogous result for probability spaces can be found for
example in a paper by \citet*[Lemma~6.4]{janson-survey}.

\begin{lemma}
Let $p\geq 1$, let $\scr S_i=(S_i,\cS_i,\mu_i)$ for $i=1,2$ be such that
$\mu_1(S_1)=\mu_2(S_2)$, and let $\mcl W_1=(W_1,\scr{S}_1)$, $\mcl
W'_1=(W'_1,\scr{S}_1)$, and $\mcl W_2=(W_2,\scr{S}_2)$ be graphons in $L^p$.
Defining $\delta_\square$ and $\delta_p$ as in Definition~\ref{defn1}(i), we
have
\[
\delta_\square(\mcl W_1,\mcl W_2)
\leq \delta_\square(\mcl W'_1,\mcl W_2) + \|W_1-W'_1\|_\square
\leq \delta_\square(\mcl W'_1,\mcl W_2) + \|W_1-W'_1\|_1
\]
and
\[
\delta_p(\mcl W_1,\mcl W_2)
\leq \delta_p(\mcl W'_1,\mcl W_2) + \|W_1-W'_1\|_p.
\]
\label{prop15}
\end{lemma}

\begin{proof}
The second bound on $\delta_\square(\mcl W_1,\mcl W_2)$ is immediate, so the
rest of the proof will consist of proving the first bound on
$\delta_\square(\mcl W_1,\mcl W_2)$ as well as the bound on $\delta_p(\mcl
W_1,\mcl W_2)$. Let $\mu$ be a measure on $(S_1\times S_2,\cS_1\times \cS_2)$
with marginals $\mu_1$ and $\mu_2$, respectively, and let $\pi_i\colon
S_1\times S_2\to S_i$ denote projections for $i=1,2$. Since
$\|\cdot\|_\square$ clearly satisfies the triangle inequality,
\[
\begin{split}
\delta_\square(\mcl W_1,\mcl W_2)
&\leq \|W^{\pi_1}_1-W^{\pi_2}_2\|_{\square,S_1\times S_2,\mu}\\
&\leq \|(W'_1)^{\pi_1}-W^{\pi_2}_2\|_{\square,S_1\times S_2,\mu}
+ \|(W'_1)^{\pi_1}-W^{\pi_1}_1\|_{\square,S_1\times S_2,\mu}\\
&=\|(W'_1)^{\pi_1}-W^{\pi_2}_2\|_{\square,S_1\times S_2,\mu}
+ \|W'_1-W_1\|_{\square,S_1,\mu_1}.
\end{split}
\]
The desired result follows by taking an infimum over all couplings. The bound
on $\delta_p(\mcl W_1,\mcl W_2)$ follows in the same way from the triangle
inequality for $\|\cdot\|_p$.
\end{proof}

\begin{remark}
We state the above lemma only for the case when $\mu_1(S_1)=\mu_2(S_2)$,
since we have not yet proved that $\delta_\square$ and $\delta_p$ are well
defined otherwise. However, once we have proved this, it is a direct
consequence of Definition~\ref{defn1}(ii) that the above lemma also holds
when $\mu_1(S_1)\neq\mu_2(S_2)$.
\end{remark}

\begin{definition}
Let $(S,\cS)$ be a measurable space, and consider a function $W\colon S\times
S\to\R$. Then $W$ is a \emph{step function} if there are some $n\in\N$,
disjoint sets $A_i\in\cS$ satisfying $\mu(A_i)<\infty$ for
$i\in\{1,\dots,n\}$, and constants $a_{i,j}\in\R$ for $i,j\in\{1,\dots,n\}$
such that
\[
W=\sum_{i,j\in\{1,\dots,n\}} a_{i,j}\1_{A_i\times A_j}.
\]
\label{defn4}
\end{definition}

Note that in order for $W$ to be a step function it is not sufficient that it
is simple, i.e., that it attains a finite number of values; the sets on which
the function is constant are required to be product sets. The set of step
functions is dense in $L^1$; hence Lemma~\ref{prop15} implies that every
graphon can be approximated arbitrarily closely by a step function for the
$\delta_\square$ metric.

\begin{lemma}
Let $p\geq 1$, let $\cW_1=(W_1,\scr S_1)$ and $\cW_2=(W_2,\scr S_2)$ be
graphons, and let $S_1=\bigcup_{i\in I} A_i$ and $S_2=\bigcup_{k\in J} B_k$
for finite index sets $I$ and $J$ such that $A_i\cap A_{i'}=\emptyset$ for
$i\neq i'$ and $B_j\cap B_{j'}=\emptyset$ for $j\neq j'$. Suppose $W_1$ and
$W_2$ are step functions of the form
\[
W_1 = \sum_{i,i'\in I} a_{i,i'}\1_{A_i\times A_{i'}}\quad\text{and}\quad
W_2 = \sum_{j,j'\in J} b_{j,j'}\1_{B_j\times B_{j'}},
\]
where $a_{i,i'}$ and $b_{j,j'}$ are constants in $\R$. Let $\mu$ and $\mu'$
be two coupling measures on $S_1\times S_2$, such that $\mu(A_i\times B_k)
=\mu'(A_i\times B_k)$ for all $(i,k)\in I\times J$. Then
$\|W_1^{\pi_1}-W_2^{\pi_2}\|_{\square,\mu}=\|W_1^{\pi_1}-W_2^{\pi_2}\|_{\square,\mu'}$
and
$\|W_1^{\pi_1}-W_2^{\pi_2}\|_{p,\mu}=\|W_1^{\pi_1}-W_2^{\pi_2}\|_{p,\mu'}$.
\label{prop23}
\end{lemma}

\begin{proof}
For all $U,V\subseteq S_1\times S_2$,
\begin{equation}
\int_{U\times V} \big(W_1^{\pi_1} - W_2^{\pi_2} \big)\,d\mu\,d\mu
= \sum_{i,i',j,j'} \mu(U\cap (A_i\times B_j)) \mu(V\cap (A_{i'}\times B_{j'}))(a_{i,i'}-b_{j,j'}).
\label{eq32}
\end{equation}
From the form of this expression and the definition of $\|\cdot\|_\square$
it is clear that we may assume there are sets $U',V'\subseteq S_1\times S_2$
such that for all $i,j$,
\[ A_i\times B_{j}\subseteq U'
\quad\text{or}\quad (A_i\times B_{j})\cap
U'=\emptyset
\]
and
\[
A_i\times B_{j}\subseteq V' \quad\text{or}\quad (A_i\times B_{j})\cap
V'=\emptyset,
\]
and such that
\[
\|W_1^{\pi_1} - W_2^{\pi_2}\|_{\square,\mu}
= \int_{U'\times V'} \big(W_1^{\pi_1} - W_2^{\pi_2}\big)\,d\mu\, d\mu.
\]
Hence it follows from \eqref{eq32} that
$\|W_1^{\pi_1}-W_2^{\pi_2}\|_{\square,\mu}=\|W_1^{\pi_1}-W_2^{\pi_2}\|_{\square,\mu'}$
if $\mu(A_i\times B_j) =\mu'(A_i\times B_j)$ for all $i,j$. The proof for the
$L^p$ metric follows from the fact that
\[
\|W_1^{\pi_1} - W_2^{\pi_2}\|_{p,\mu}^p
=  \sum_{i,i',j,j'} |a_{i,i'}-b_{j,j'}|^p
\mu{(A_i\times B_{j})}\mu(A_{i'}\times B_{j'}).
\]
\end{proof}

\begin{corollary}\label{cor:cut-norm-dep-on-mu}
Let $p\geq 1$ and for $k=1,2$, let $\cW_k=(W_k,\scr S_k)$ with $\scr
S_k=(S_k,\cS_k,\mu_k)$ be graphons in $L^p$.  For $k=1,2$,  let $\wt{\scr
S}_k=(\wt S_k,\wt\cS_k,\wt\mu_k)$ and $\wh{\scr S}_k=(\wh
S_k,\wh\cS_k,\wh\mu_k)$ be extensions of $\scr S_k$ with $\wt\mu_1(\wt
S_1)=\wt\mu_2(\wt S_2)\in(0,\infty]$ and $\wh\mu_1(\wh S_1)=\wh\mu_2(\wh
S_2)\in(0,\infty]$, and let $\wt W_k$ and $\wh W_k$ be the trivial extensions
of $W_k$ to $\wt{\scr S}_k$ and $\wh{\scr S}_k$. Let $\wt\mu$ and $\wh\mu$ be
couplings of $\wt\mu_1$ and $\wt\mu_2$, and $\wh\mu_1$ and $\wh\mu_2$,
respectively. If $\wt\mu$ and $\wh\mu$ agree on $S_1\times S_2$, then $\|\wt
W_1^{\pi_1}-\wt W_2^{\pi_2}\|_{\square,\wt\mu}= \|\wh W_1^{\pi_1}-\wh
W_2^{\pi_2}\|_{\square,\wh\mu}$ and $\|\wt W_1^{\pi_1}-\wt
W_2^{\pi_2}\|_{p,\wt\mu}= \|\wh W_1^{\pi_1}-\wh W_2^{\pi_2}\|_{p,\wh\mu}$.
\end{corollary}

\begin{proof}
By Lemma~\ref{prop15} and the fact that step functions are dense in $L^1$ and
in $L^p$, it is sufficient to prove the corollary for step functions.  The
corollary then follows from Lemma~\ref{prop23} by observing that for two sets
$A\in {\cS}_1$ and $B\in {\cS}_2$ with finite measure $\mu_1(A)$ and
$\mu_2(B)$, the $\wt\mu$ measure of sets of the form $A\times (\wt
S_2\setminus S_2)$ and $(\wt S_1\setminus S_1)\times B$ can be expressed as
$\mu_1(A) -\wt\mu(A\times S_2)$ and $\mu_2(B) -\wt\mu(S_1\times B)$,
respectively, implying that $\|\wt W_1^{\pi_1}-\wt
W_2^{\pi_2}\|_{\square,\wt\mu}$ and $\|\wt W_1^{\pi_1}-\wt
W_2^{\pi_2}\|_{p,\wt\mu}$ depend only on the restriction of $\wt\mu$ to
$S_1\times S_2$, and similarly for $\|\wh W_1^{\pi_1}-\wh
W_2^{\pi_2}\|_{\square,\wh\mu}$ and $\|\wh W_1^{\pi_1}-\wh
W_2^{\pi_2}\|_{p,\wh\mu}$.
\end{proof}

Lemma~\ref{prop23} is also used in the proof of the triangle inequality in
the following lemma. The proof follows the same strategy as the proof by
\citet*[Lemma 6.5]{janson-survey} for the case of probability spaces.

\begin{lemma}
Let $p\geq 1$. For $i=1,2,3$ let $\cW_i=(W_i,\scr S_i)$ with $\scr
S_i=(S_i,\cS_i,\mu_i)$ be a graphon in $L^p$, such that
$\mu_1(S_1)=\mu_2(S_2)=\mu_3(S_3)\in(0,\infty]$. Defining $\delta_\square$
and $\delta_p$ as in Definition~\ref{defn1}(i), we have
\[
\delta_\square(\cW_1,\cW_3)
\leq \delta_\square(\cW_1,\cW_2)+\delta_\square(\cW_2,\cW_3)
\quad\text{and}\quad
\delta_p(\cW_1,\cW_3)
\leq \delta_p(\cW_1,\cW_2)+\delta_p(\cW_2,\cW_3).
\]
\label{prop24}
\end{lemma}
\begin{proof}
By Lemma~\ref{prop15} and since step functions are dense in $L^1$, we may
assume that $W_i$ is a step function for $i=1,2,3$. Let $\scr
S_1=\bigcup_{j=1}^\infty A_j$ (resp.\ $\scr S_2=\bigcup_{j=1}^\infty B_j$,
$\scr S_3=\bigcup_{j=1}^\infty C_j$) be such that $W_1|_{A_j\times A_k}$
(resp.\ $W_2|_{B_j\times B_k}$, $W_3|_{C_j\times C_k}$) is constant for all
$j,k\in\N$, and assume without loss of generality that
$\mu_1(A_j),\mu_2(B_j),\mu_3(C_j)\in(0,\infty)$ for all $j\in\N$. Throughout
the proof we abuse notation slightly and let $\pi_i$ denote projection onto
$S_i$ from any space which is a product of $S_i$ and another space; for
example, $\pi_1$ denotes projection onto $S_1$ from $S_1\times S_2\times
S_3$, $S_1\times S_2$, and $S_1\times S_3$.

Let $\ep>0$, and let $\mu'$ (resp.\ $\mu''$) be a coupling measure on
$S_1\times S_2$ (resp.\ $S_2\times S_3$) such that
\[
\|W_1^{\pi_1} - W_2^{\pi_2}\|_{\square,\mu'}<\delta_\square(\W_1,\W_2)+\ep
\quad\text{and}\quad
\|W_2^{\pi_2} - W_3^{\pi_3}\|_{\square,\mu''}<\delta_\square(\W_2,\W_3)+\ep.
\]
We define a measure $\mu$ on $S_1\times S_2\times S_3$ for any $E\subseteq
S_1\times S_2\times S_3$ which is measurable for the product $\sigma$-algebra
by
\[
\mu(E) = \sum_{i,j,k} \frac{\mu'(A_i\times B_j) \mu''(B_j\times C_k)}{\mu_2(B_j)}
\frac{\mu_1\times\mu_2\times\mu_3(E\cap(A_i\times B_j\times C_k))}{\mu_1(A_i)\mu_2(B_j)\mu_3(C_k)}.
\]
By a straightforward calculation (see, for example, the paper by
\citel{janson-survey}, Lemma~6.5) the three mappings $\pi_l\colon (S_1\times
S_2\times S_3,\mu)\to(S_l,\mu_l)$ for $l=1,2,3$ are measure-preserving.
Furthermore, if $\wt\mu'$ is the pushforward measure of $\mu$ for the
projection $\pi_{12}\colon S_1\times S_2\times S_3\to S_1\times S_2$, then
\[
\wt \mu'(A_i\times B_j) = \mu'(A_i\times B_j)\quad\text{for all } i,j.
\]
By Lemma~\ref{prop23} and since $\pi_{12}\colon (S_1\times S_2\times
S_3,\mu)\to (S_1\times S_2,\wt\mu')$ is measure-preserving,
\[
\|W_1^{\pi_1}-W_2^{\pi_2}\|_{\square,S_1\times S_2,\mu'}
= \|W_1^{\pi_1}-W_2^{\pi_2}\|_{\square,S_1\times S_2,\wt\mu'}
= \|W_1^{\pi_1}-W_2^{\pi_2}\|_{\square,S_1\times S_2\times S_3,\mu}.
\]
Hence,
\[
\|W_1^{\pi_1}-W_2^{\pi_2}\|_{\square,S_1\times S_2\times S_3,\mu} < \delta_\square(\W_1,\W_2) + \ep.
\]
Similarly,
\[
\|W_2^{\pi_2}-W_3^{\pi_3}\|_{\square,S_1\times S_2\times S_3,\mu}
< \delta_\square(\W_2,\W_3) + \ep.
\]
Letting $\wh\mu$ be the pushforward measure on $S_1\times S_3$ of $\mu$ for
the projection $\pi_{13}\colon S_1\times S_2\times S_3\to S_1\times S_3$, we
have
\[
\|W_1^{\pi_1}-W_3^{\pi_3}\|_{\square,S_1\times S_3,\wh\mu}
= \|W_1^{\pi_1}-W_3^{\pi_3}\|_{\square,S_1\times S_2\times S_3,\mu}.
\]
Since the cut norm $\|\cdot\|_\square$ clearly satisfies the triangle
inequality,
\[
\begin{split}
\delta_\square(\W_1,\W_3)
&\leq \|W_1^{\pi_1}-W_3^{\pi_3}\|_{\square,S_1\times S_3,\wh\mu}
=    \|W_1^{\pi_1}-W_3^{\pi_3}\|_{\square,S_1\times S_2\times S_3,\mu}\\
&\leq \|W_1^{\pi_1}-W_2^{\pi_2}\|_{\square,S_1\times S_2\times S_3,\mu} +
\|W_2^{\pi_2}-W_3^{\pi_3}\|_{\square,S_1\times S_2\times S_3,\mu}\\
&< \delta_\square(\cW_1,\cW_2) + \delta_\square(\cW_2,\cW_3)+2\ep.
\end{split}
\]
Since $\ep$ was arbitrary this completes our proof for $\delta_\square$.  The
proof for $\delta_p$ is identical.
\end{proof}

\begin{lemma}
Let  $\W_i=(W_i,\scr S_i)$ with $\scr S_i=(S_i,\cS_i,\mu_i)$ be a graphon for
$i=1,2$, such that $\mu_1(S_1)=\mu_2(S_2)\in(0,\infty]$. For $i=1,2$ let
$\wt{\scr S}_i=(\wt S_i,\wt\cS_i,\wt\mu_i)$ be an extension of $\scr S_i$,
such that $\wt\mu_1(\wt S_1)=\wt\mu_2(\wt S_2)\in(0,\infty]$, and let
$\wt{\W}_i$ be the trivial extension of $\W_i$ to $\wt{\scr S}_i$. Then
$\delta_\square(\W_1,\W_2)=\delta_\square(\wt\W_1,\wt\W_2)$ and
$\delta_1(\W_1,\W_2)=\delta_1(\wt\W_1,\wt\W_2)$, where $\delta_\square$ and
$\delta_1$ are as in Definition~\ref{defn1}(i). If $p>1$ and $\W_1$ and
$\W_2$ are non-negative graphons in $L^p$, then the result holds for
$\delta_p$ as well. \label{prop21}
\end{lemma}

\begin{remark}
We remark that the assumption of non-negativity is necessary for the lemma to
hold when $p>1$. If for example $\W_1=(\1,[0,1])$ and $\W_2=(-\1,[0,1])$ are
graphons over $[0,1]$, and if $\wt{\W}_1$ and $\wt{\W}_2$ are the trivial
extensions to $[0,2]$, then $\delta_p(\W_1,\W_2)=2$ and
$\delta_p(\wt\W_1,\wt\W_2)=2^{1/p}$.
\end{remark}

\begin{proof}
We start with the proof for the cut metric. We will first prove the result
for the case when $\wt S_i=\N$ and $S_i=S:=\{1,\dots,n\}$ for $i=1,2$ and
some $n\in\N$, $\cS_i$ and $\wt{\cS}_i$ are the associated discrete
$\sigma$-algebras, and $\wt\mu_i(x)=c$ for all $x\in \wt S_i$ and some $c\in
(0,1)$.

First we will argue that
\[
\delta_\square(\W_1,\W_2)\geq\delta_\square(\wt\W_1,\wt\W_2).
\]
By Definition~\ref{defn1}(i) it is sufficient to prove that for each coupling
measure $\mu$ on $S\times S$ we can define a coupling measure $\wt\mu$ on
$\N\times\N$ such that $\|W_1^{\pi_1}-W_2^{\pi_2}\|_{\square,\mu}=\|\wt
W_1^{\wt\pi_1}-\wt W_2^{\wt\pi_2}\|_{\square,\wt\mu}$. But this is immediate,
since we can define $\wt\mu$ such that $\wt\mu|_{S\times S}=\mu$, and
$\wt\mu(A_1\times A_2)=c|A_1\cap A_2|$ for $A_i \subseteq \N\backslash S$.

Next we will prove that
\begin{equation}
\delta_\square(\W_1,\W_2)\leq\delta_\square(\wt\W_1,\wt\W_2).
\label{eq28}
\end{equation}
Again by Definition~\ref{defn1}(i), it will be sufficient to prove that given
any coupling measure $\wt\mu$ on $\N\times\N$ we can find a coupling measure
$\mu$ on $S\times S$ such that
$\|W_1^{\pi_1}-W_2^{\pi_2}\|_{\square,\mu}\leq\|\wt W_1^{\pi_1}-\wt
W_2^{\pi_2}\|_{\square,\wt\mu}$.

By the following argument we may approximate $\|\wt W_1^{\pi_1}-\wt
W_2^{\pi_2}\|_{\square,\wt\mu}$ arbitrarily well by replacing $\wt\mu$ with a
coupling measure which is supported on $(\wh S\times\wh
S)\cup((\N\backslash\wh S)\times(\N\backslash\wh S))$, where $\wh
S:=\{1,\dots,K\}$ for some sufficiently large $K\in\N$. Indeed, by
Corollary~\ref{cor-to-prop20}, given a coupling measure $\wt\mu$ on
$\N\times\N$ and $K\in\N$, we can define a measure $\wh\mu$ supported on
$(\wh S\times\wh S)\cup((\N\backslash\wh S)\times(\N\backslash\wh S))$ such
that $\wh\mu \geq \wt\mu$ on $(\wh S\times\wh S)\cup((\N\backslash\wh
S)\times(\N\backslash\wh S))$. It is easy to see from the construction of
this measure in the proof of Corollary~\ref{cor-to-prop20} that when $K$
converges to infinity, the measure $\wh\mu$ converges to $\wt\mu$ when
restricted to $(S\times\N)\cup( \N\times S)$ (for example for the topology
where we look at the maximum difference of the measure assigned to any set in
$(S\times\N)\cup( \N\times S)$). Therefore the corresponding cut norms also
converge. This shows that we may assume  that $\wt\mu$ is supported on $(\wh
S\times\wh S)\cup((\N\backslash\wh S)\times(\N\backslash\wh S))$ for some
$K\in\N$ when proving \eqref{eq28}.

Let $\wt\mu'$ be the restriction of $\wt\mu$ to $\wh S\times\wh S$. Then
$\|\wt{W}_1-\wt{W}_2\|_{\square,\N\times\N,\wt\mu}=
\|\wh{W}_1-\wh{W}_2\|_{\square,\wh S\times\wh S,\wt\mu'}$ where
$\wh{\cW}_i=(\wh W_i,\wh{\scr S}_i)$ is the trivial extensions of $\cW_i$ to
the measure space $\wh{\scr S}_i$ associated with $\wh S_i$. We will prove
that we may assume without loss of generality that $\wt\mu'$ corresponds to a
permutation of $\wh S$. By choosing $M\in\N$ sufficiently large we can
approximate $\|\wh{W}_1-\wh{W}_2\|_{\square,\wt\mu'}$ arbitrarily well by
replacing $\wt\mu'$ with a measure such that each element $(i,j)\in \wh
S\times \wh S$ has a measure which is an integer multiple of $c/M$; hence we
may assume $\wt\mu'$ is on this form. Each such $\wt\mu'$ can be described in
terms of a permutation $\sigma'$ of $[KM]$ via
$\wt\mu'((i,j))=\sum_{\ell=1}^{KM} c/M\delta_{i,\lceil
\ell/M\rceil}\delta_{j,\lceil\sigma'(\ell)/M\rceil}$. Let
$\wh{\W}'_i=(\wh{W}'_i,[KM])$ be the graphon such that each $j\in[KM]$ has
measure $c/M$, and such that $\wh{W}'_i=(\wh W_i)^{\phi}$ for the
measure-preserving map $\phi\colon [KM]\to[K]$ defined by $\phi(j):=\lceil
j/M\rceil$. Using Proposition~\ref{prop23} and the above observation on
describing $\wt\mu'((i,j))$ in terms of a permutation $\sigma'$ of $[KM]$ we
see that
$\|\wh{W}_1-\wh{W}_2\|_{\square,\wt\mu'}=\|\wh{W}'_1-(\wh{W}'_2)^{\sigma'}\|_{\square}$.
Upon replacing $\wh W_i$ by $\wh W'_i$ throughout the proof, we may assume
that the measure $\wt\mu'$ is a permutation.

To complete the proof it is therefore sufficient to consider some permutation
$\wh\sigma$ of $\wh S$ and prove that we can find a permutation $\sigma$ of
$\wh S$ mapping $S$ to $S$ such that
\begin{equation}
\|\wh W_1-\wh W_2^{\wh\sigma}\|_{\square} \geq
\| \wh W_1- \wh W_2^{\sigma}\|_{\square}.
\label{eq-32}
\end{equation}
We modify the permutation $\wh\sigma$ step by step to obtain a permutation
mapping $S$ to $S$. Abusing notation slightly we let $\wh\sigma$ and $\sigma$
denote the old and new, respectively, permutations in a single step. In each
step choose $i_1,i_2\leq n$ and  $ j_1,j_2>n$ such that $\wh\sigma(i_1)=j_1$
and $\wh\sigma(j_2)=i_2$; if such $i_1,i_2,j_1,j_2$ do not exist we know that
$\wh\sigma$ maps $S$ to $S$. Then define $\sigma(i_1):=i_2$ and
$\sigma(j_2):=j_1$, and for $k\not\in \{i_1,j_2\}$ define
$\sigma(k):=\wh\sigma(k)$. We have $\|\wh W_1-\wh W_2^{\sigma}\|_\square \leq
\|\wh W_1-\wh W_2^{\wh\sigma}\|_\square$ by the following argument. Let
$U,V\subseteq \N$ be such that $\|\wh W_1-\wh
W_2^{\sigma}\|_\square=|\int_{U\times V} \big(\wh W_1-\wh
W_2^{\sigma}\big)\,dx\, dy|$. Since $\sigma(j_2)>n$ (implying that both $\wh
W_1$ and $\wh W_2^{\sigma}$ are trivial on $(j_2\times \N)$ and $(\N\times
j_2)$) the following identity holds if we define $U':=U\backslash\{j_2\}$ or
$U':=U\cup\{j_2\}$, and if we define $V':=V\backslash\{j_2\}$ or
$V':=V\cup\{j_2\}$:
\[
\int_{U\times V} \big(\wh W_1-\wh W_2^{\sigma}\big)\,dx\, dy = \int_{U'\times V'} \big(\wh W_1-\wh W_2^{\sigma}\big)\,dx\, dy.
\]
In other words, $\int_{U\times V} (\wh W_1-\wh W_2^{\sigma})\,dx\, dy$ is
invariant under adding or removing $j_2$ from $U$ and/or $V$. Therefore we
may assume without loss of generality that
\begin{equation}
j_2\in U \text{ iff } i_1\in U, \qquad
j_2\in V \text{ iff } i_1\in V,
\label{eq-35}
\end{equation}
since if \eqref{eq-35} is not satisfied we may redefine $U$ and $V$ such that
\eqref{eq-35} holds, and we still have $\|\wh W_1-\wh
W_2^{\sigma}\|_\square=|\int_{U\times V} (\wh W_1-\wh W_2^{\sigma})\,dx\, dy|$.
The assumption \eqref{eq-35} implies that
$\int_{U\times V} (\wh W_1-\wh
W_2^{\sigma})\,dx\, dy=\int_{U\times V} (\wh W_1-\wh W_2^{\wh\sigma})\,dx\, dy$,
which implies \eqref{eq-32} since we can obtain a permutation $\sigma$
mapping $S$ to $S$ in finitely many steps as described above.

Now we will prove the lemma for general graphons. We will reduce the problem
step by step to a problem with additional conditions on the measure spaces
involved, until we have reduced the problem to the special case considered
above.

First we show that we may assume $\scr S_i$ and $\wt{\scr S}_i$ are
non-atomic. Define $S'_i:=S_i\times[0,1]$ and $\wt S'_i:=\wt S_i\times[0,1]$,
let $\scr S'_i$ and $\wt{\scr S}'_i$ be the corresponding atomless product
measure spaces when $[0,1]$ is equipped with Lebesgue measure, and let
$\W'_i=(W'_i,\scr S'_i)$ and $\wt{\W}'_i=(\wt W'_i,\wt{\scr S}'_i)$ be
graphons such that $W'_i=(W_i)^{\pi_i^1}$ and $\wt W'_i=(\wt
W_i)^{\wt\pi_i^1}$, where $\pi_i^1\colon {\scr S}'_i\to {\scr S}_i$ and
$\wt\pi_i^1\colon \wt{\scr S}'_i\to \wt{\scr S}_i$ are the projection maps on
the first coordinates. By considering the natural coupling of $\wt S'_i$ and
$\wt S_i$ it is clear that $\delta_\square(\wt W'_i,\wt W_i)=0$. It therefore
follows from the triangle inequality that $\delta_\square(\wt W_1,\wt
W_2)=\delta_\square(\wt W'_1,\wt W'_2)$. Similarly,
$\delta_\square(W_1,W_2)=\delta_\square(W'_1,W'_2)$. In order to prove that
$\delta_\square(\wt W_1,\wt W_2)=\delta_\square(W_1,W_2)$ it is therefore
sufficient to prove that $\delta_\square(\wt W'_1,\wt
W'_2)=\delta_\square(W'_1,W'_2)$. Since $\scr S'_i$ and $\wt{\scr S}'_i$ are
atomless and $\wt W'_i$ is a trivial extension of $W'_i$ it is therefore
sufficient to prove the lemma for atomless measure spaces.

Next we will reduce the general case to the case when $\wt\mu_i(\wt
S_i)=\infty$. If $\wt\mu_i(\wt S_i)<\infty$ we extend $\wt{\scr S}_i$ to a
space $\wh{\scr S}_i$ of infinite measure, and let $\wh{\cW}_i$ be the
trivial extension of $\wt{\cW}_i$ to $\wh{\scr S}_i$. Assuming we have proved
the lemma for the case when the extended measure spaces have infinite
measure, it follows that
\[
\delta_\square(\cW_1,\cW_2)
= \delta_\square(\wh{\cW}_1,\wh{\cW}_2)
= \delta_\square(\wt{\cW}_1,\wt{\cW}_2);
\]
hence the lemma also holds for the case when $\wt\mu_i(\wt S_i)<\infty$.

Next we prove that we may assume $\mu_i(S_i)<\infty$. We proceed similarly as
in the previous paragraph, and assume $\mu_i(S_i)=\infty$. By Lemma
\ref{prop15} we may assume $W_i$ are supported on sets of finite measure, and
we let $\wh{\scr S}_i=(\wh S_i,\wh{\cS}_i,\wh\mu_i)$ be a restriction of
$\scr S_i$ such that $\text{supp}(W_i)\subseteq \wh S_i\times \wh S_i$ and
$\wh\mu_i(\wh S_i)<\infty$. Since $\scr S_i$ is non-atomic we may assume
$\wh{\mu}_1(\wh S_1)=\wh{\mu}_2(\wh S_2)$. Define the graphon
$\wh{\cW}_i=(\wh W_i,\wh{\scr S}_i)$ to be such that $\cW_i$ is the trivial
extension of $\wh{\cW}_i$ to $\scr S_i$. Assuming we have proved the lemma
for the case when $\mu_i(S_i)<\infty$, it follows that
\[
\delta_\square(\cW_1,\cW_2)
= \delta_\square(\wh{\cW}_1,\wh{\cW}_2)
= \delta_\square(\wt{\cW}_1,\wt{\cW}_2);
\]
hence the lemma also holds for the case when $\mu_i(S_i)<\infty$.

Next we will prove that we may assume $W_i$ is a step function for $i=1,2$,
such that each step has the same measure $c>0$. Step functions are dense in
$L^1$, and hence it is immediate from Lemma~\ref{prop15} that we may assume
$W_i$ is a step function. We may assume that the measure of each step is a
rational multiple of $\mu_i(S_i)$; if this is not the case we may adjust the
steps slightly (because $\scr S_i$ is non-atomic, we can choose subsets of
the steps of any desired measures, by Exercise~2 from \S41 in the book of
\citel{halmos}) to obtain this. Assuming each step has a measure which is a
rational multiple of $\mu_i(S_i)$ we may subdivide each step such that each
step obtains the same measure $c>0$, again using the exercise in the book by
\citet*{halmos}.

Assume $W_i$ are step functions consisting of $k\in\N$ steps each having
measure $c>0$, and that $\mu_i(S_i)<\infty$ and $\wt\mu_i(\wt S_i)=\infty$. Let
$\cW'_i=(W'_i,[n])$ (resp.\ $\wt{\W}'_i=(\wt{W}'_i,\N)$) be a graphon over
$[n]:=\{1,\dots,n\}$ (resp.\ $\N$) equipped with the discrete
$\sigma$-algebra, such that each $j\in[n]$ (resp.\ $j\in\N$) has measure $c$,
and such that $W_i=(W'_i)^{\phi_i}$ (resp.\ $\wt W_i=(\wt W'_i)^{\wt\phi_i}$)
for a measure-preserving map $\phi_i\colon S_i\to [k]$ (resp.\
$\wt\phi_i\colon \wt S_i\to \N$). Then
$\delta_\square(\cW_1,\cW_2)=\delta_\square(\cW'_1,\cW'_2)$ and
$\delta_\square(\wt{\cW}_1,\wt{\cW}_2)=\delta_\square(\wt{\cW}'_1,\wt{\cW}'_2)$.
By the special case we considered in the first paragraphs of the proof,
$\delta_\square(\cW'_1,\cW'_2)=\delta_\square(\wt{\cW}'_1,\wt{\cW}'_2)$.
Combining the above identities, our desired result
$\delta_\square(\cW_1,\cW_2)=\delta_\square(\wt{\cW}_1,\wt{\cW}_2)$ follows.

To prove the result for the metric $\delta_p$, we follow the steps above. The
only place where the proof differs is in the proof of \eqref{eq-32}. Let
$\sigma$, $\wh\sigma$, and $i_1,i_2,j_1,j_2$ be as in the proof of
\eqref{eq-32}. We would like to show that
\[
\| \wh W_1- \wh W_2^{\sigma}\|_p^p\leq \|\wh W_1-\wh W_2^{\wh\sigma}\|_p^p.
\]
Writing both sides as a sum over $(i,j)\in \wh S^2$ we consider the
following three cases separately: (i) $(i,j)\in (\wh S\setminus
\{i_1,j_2\})^2$, (ii) $(i,j)\in \{i_1,j_2\}\times (\wh S\setminus
\{i_1,j_2\})$ or $(i,j)\in (\wh S\setminus \{i_1,j_2\})\times \{i_1,j_2\}$,
and (iii) $(i,j)\in \{i_1,j_2\}\times \{i_1,j_2\}$. In case (i) the terms are
identical on the left side and on the right side. For dealing with case (ii)
it is sufficient to prove that for an arbitrary $i\in (\wh S\setminus
\{i_1,j_2\})$,
\begin{equation*}
\begin{split}
|\wh W_1(i_1,i)- \wh W_2^{\sigma}(i_1,i)|^p&
+ |\wh W_1(j_2,i)-\wh W_2^{\sigma}(j_2,i)|^p\\
&\leq
|\wh W_1(i_1,i)- \wh W_2^{\wh\sigma}(i_1,i)|^p
+ |\wh W_1(j_2,i)-\wh W_2^{\wh\sigma}(j_2,i)|^p.
\end{split}
\end{equation*}
This is equivalent to
\begin{equation*}
|\wh W_1(i_1,i)- \wh W_2(i_2,\wh\sigma(i))|^p
\leq
|\wh W_1(i_1,i)|^p
+ |\wh W_2(i_2,\wh\sigma(i))|^p.
\end{equation*}
This inequality is obviously true if either $p=1$ or both $\wh W_1$ and
$\wh W_2$ are non-negative. For case (iii) we need to show that
\[
\sum_{i,j\in \{i_1,j_2\}}| \wh W_1(i,j)- \wh W_2^{\sigma}(i,j)|^p\leq\sum_{i,j\in \{i_1,j_2\}} |\wh W_1(i,j)-\wh W_2^{\wh\sigma}(i,j)|^p.
\]
Writing out both sides we see that this is equivalent to
\[
| \wh W_1(i_1,i_1)- \wh W_2(i_2,i_2)|^p\leq | \wh W_1(i_1,i_1)|^p+|\wh W_2(i_2,i_2)|^p,
\]
which is again true if either $p=1$ or both $\wh W_1$ and
$\wh W_2$ are non-negative.
\end{proof}

\begin{proof}[Proof of Proposition~\ref{prop8} and Proposition~\ref{prop:delta-p}]
The existence of a coupling follows from Lemma~\ref{prop20} with
$D_1=D_2=\emptyset$.

To prove the next statement of the proposition, i.e., that the value of
$\delta_\square(\W_1,\W_2)$ is independent of the extensions $\wt{\scr S}_i$,
we consider alternative extensions $\wh{\scr S}_i=(\wh
S_i,\wh\cS_i,\wh\mu_i)$ of $\scr S_i$ for $i=1,2$, and let $\wh{\W}_i$ denote
the trivial extension of $\W_i$ to $\wh{\scr S}_i$. By Lemma~\ref{prop21} it
is sufficient to consider the case when $\wt\mu_i(\wt S_i\backslash
S_i)=\wh\mu_i(\wh S_i\backslash S_i)=\infty$, since if the extensions do not
satisfy this property we can extend them to a space of infinite measure, and
Lemma~\ref{prop21} shows that the cut norm is unchanged. It is sufficient to
prove that, given any coupling measure $\wt\mu$ on $\wt S_1\times\wt S_2$, we
can find a coupling measure $\wh\mu$ on $\wh S_1\times\wh S_2$, such that
\begin{equation}
 \|\wt W_1^{\pi_1}-\wt W_2^{\pi_2}\|_{\square,\wt S_1\times \wt S_2,\wt\mu}
 =
 \|\wh W_1^{\pi_1}-\wh W_2^{\pi_2}\|_{\square,\wh S_1\times \wh S_2,\wh\mu}.
 \label{eq2}
\end{equation}
By Corollary~\ref{cor:cut-norm-dep-on-mu}, the left side  of \eqref{eq2} only
depends on $\wt\mu$ restricted to $S_1\times S_2$; in a similar way, the
right side only depends on $\wh\mu$ restricted to $S_1\times S_2$. We
therefore can define an appropriate measure $\wh\mu$ on $\wh S_1\times \wh
S_2$ by defining $\wh\mu|_{S_1\times S_2}=\wt\mu|_{S_1\times S_2}$, and
extending it to a coupling measure on $\wh S_1\times\wh S_2$ by Lemma
\ref{prop20}; this yields \eqref{eq2}.

The function $\delta_\square$ is clearly symmetric and non-negative. To prove
that it is a pseudometric it is therefore sufficient to prove that it
satisfies the triangle inequality. This is immediate by Lemma~\ref{prop24}
and the definition of $\delta_\square$ as given in Definition
\ref{defn1}(ii).

The proof for the metric $\delta_1$ follows exactly the same steps.

Using the statements of Lemma~\ref{prop21},
Corollary~\ref{cor:cut-norm-dep-on-mu}, and Lemma~\ref{prop24} for $p>1$, the
above proof of Proposition~\ref{prop8} immediately generalizes to the
invariant $L^p$ metric $\delta_p$ as long as the graphons in question are
non-negative graphons in $L^p$ (in addition to being in $L^1$, as required by
the definition of a graphon).  This proves Proposition~\ref{prop:delta-p}.
\end{proof}

For two graphons $\W=(W,\scr S)$ and $\wt{\W}=(\wt W,\wt{\scr S})$ for which
$\scr S=\wt{\scr S}$, it is immediate that $\delta_\square(\W,\wt{\W})\leq
\|W-\wt W\|_1$. The following lemma gives an analogous bound when $W$ and
$\wt W$ are defined on the same measurable space and $W=\wt W$, but the
measures are not identical. If $\mu$ and $\wt\mu$ are two measures on the
same measurable space $(S,\mcl S)$ and $a\geq 0$ we write $\mu\leq a\wt{\mu}$
to mean that $\mu(A)\leq a\wt{\mu}(A)$ for every $A\in\mcl S$.

\begin{lemma}
Let $\W=(W,\scr S)$ with $\scr S=(S,\mcl S,\mu)$ and $\wt{\W}=(W,\wt{\scr
S})$ with $\wt{\scr S}=(S,\mcl S,\wt\mu)$ be graphons, and assume there is an
$\ep\in(0,1)$ such that $\mu\leq \wt\mu\leq (1+\ep)\mu$. Then
$\delta_\square(\W,\wt{\W})\leq 3\ep\|W\|_{1,\mu}$. \label{prop16}
\end{lemma}

\begin{proof}
To distinguish between the graphons $\W$ and $\wt{\W}$ we will write
$\wt{\W}=(\wt{W},\wt{\scr S})$ and $\wt{\scr S}=(\wt S,\wt{\cS},\wt\mu)$, but
recall throughout the proof that $\wt W=W$ and $(\wt S,\wt{\cS})=(S,\cS)$.
Define $(S',\cS'):=(S,\cS)$, $\mu':=\wt\mu-\mu$, and $\scr
S':=(S',\cS',\mu')$, and let $\scr S''$ be the disjoint union of $\scr S$ and
$\scr S'$. Let $\W''=(W'',\scr{S}'')$ be the trivial extension of $\W$ to
$\scr S''$, and note that $\W''$ and $\wt{\W}$ are graphons over spaces of
equal total measure. Since $\delta_\square(\W'',\W)=0$ it is sufficient to
prove that $\delta_\square(\wt{\W},\W'')\leq 3\ep\|W\|_{1,\mu}$. Let $\wh\mu$
be the coupling measure on $\wt S\times S''$ such that if $\wt A\in\wt{\cS}$,
$A\in\cS$, and $A'\in\cS'$, then
\begin{equation}
\wh\mu(\wt A\times (A\cup A'))
=\mu(\wt A\cap A) + \mu'(\wt A\cap A').
\label{eq27}
\end{equation}
To complete the proof of the lemma it is sufficient to show that for all
measurable sets $U'',V''\subseteq \wt S\times S''$,
\begin{equation}
\bigg|
\int_{U''\times V''} (\wt W^{\pi_1} - (W'')^{\pi_2}) \,d\wh\mu\, d\wh\mu \bigg|
\leq 3\ep \|W\|_{1,\mu},
\label{eq23}
\end{equation}
where $\pi_1\colon \wt S\times S''\to \wt S$ (resp. $\pi_2\colon \wt S\times
S''\to S''$) is the projection onto the first (resp.\ second) coordinate of
$\wt S\times S''$. Let $U,V\subseteq \wt S\times S$ and $U',V'\subseteq \wt
S\times S'$ be such that $U''=U\cup U'$ and $V''=V\cup V'$. Recall that since
$S'=S$ we may also view $U',V'$ as sets in $\wt S\times S$, and we denote
these sets by $U'_S,V'_S$, respectively. By first using $W''|_{(S''\times
S'')\backslash (S\times S)}=0$ while $\wh\mu$-almost surely $\wt W^{\pi_1}
=(W'')^{\pi_2}$ on $(\tilde S\times  S)^2$, then using $\mu'\leq \ep\mu$, and
then using $\wt W=W$ and \eqref{eq27}, we obtain the estimate \eqref{eq23}:
\[
\begin{split}
\bigg|\int_{U''\times V''} & \big(\wt W^{\pi_1} - (W'')^{\pi_2} \big) \,d\wh\mu\, d\wh\mu\bigg|\\
&= \bigg|\int_{U\times V'} \wt W^{\pi_1} \,d\wh\mu\, d\wh\mu
+ \int_{U'\times V} \wt W^{\pi_1} \,d\wh\mu\, d\wh\mu
+ \int_{U'\times V'} \wt W^{\pi_1} \,d\wh\mu\, d\wh\mu\bigg|\\
& \leq \ep\int_{U\times V'_S} |\wt W|^{\pi_1} \,d\wh\mu\, d\wh\mu
+ \ep\int_{U'_S\times V} |\wt W|^{\pi_1} \,d\wh\mu\, d\wh\mu
+ \ep^2\int_{U'_S\times V'_S} |\wt W|^{\pi_1} \,d\wh\mu\, d\wh\mu\\
&\leq 3\ep \|W\|_{1}.
\end{split}
\]
\end{proof}

We close this appendix with a lemma that immediately implies
Proposition~\ref{pro:deg-conv}.

\begin{lemma}
Let $\eps\geq 0$ and let $\W=(W,(S,\cS,\mu))$ and $\W'=(W',(S',{\cS'},\mu'))$
be two graphons with $\delta_\square(\W,\W')\leq \eps^2/2$. Then
\[
\mu(\{D_W> \lambda+2\eps\})-\eps\leq\mu'(\{D_{ W' }> \lambda+\eps\})\leq \mu(\{D_{ W }> \lambda\})+\eps
\]
for all $\lambda\geq 0$.
\label{lem:D_W}
\end{lemma}

\begin{proof}
Since the trivial extension of a graphon $\W$  only changes the measure of
the set $\{D_W=0\}$, we may assume without loss of generality that
$\mu(S)=\mu'(S')$. Let $\pi_1$ and $\pi_2$ be the projections from  $S\times
S'$ onto the two coordinates, let $\eps'>\eps$, and let $\wh\mu$ be a
coupling of $\mu$ and $\mu'$ such that
\[
\|W^{\pi_1}-{(W')}^{\pi_2}\|_{\square,\wh\mu}\leq \frac{{(\eps')^2}}2.
\]
By definition of the cut metric, this implies that
\[
\left|\int_U (D_W(x)-D_{W'}(x'))\,d\wh\mu(x,x')\right|\leq\frac{{(\eps')^2}}2
\]
for all $U\subseteq S\times S'$.  Applying this bound for $U=\{(x,x')\in
S\times S' \colon D_W(x)-D_{W'}(x')\geq 0\}$ and $U=\{(x,x')\in S\times S'
\colon D_W(x)-D_{W'}(x')\leq 0\}$, this implies that
\[
\int_{S\times S'}\big|D_W(x)-D_{W'}(x')\big|\,d\wh\mu(x,x')\leq {{(\eps')^2}},
\]
which in turn implies that
\[
\wh\mu\left(\{(x,x')\in S\times S'\colon \big|D_W(x)-D_{W'}(x')\big|\geq \eps'\}\right)\leq \eps'.
\]
As a consequence
\[
\begin{split}
\mu(\{D_W> \lambda+2\eps'\})-\eps'
&\leq {\wh\mu}(\{(x,x')\colon D_W(x)>\lambda+2\eps'\text{ and } \big|D_W(x)-D_{W'}(x')\big|< \eps'\})
\\
&\leq {\wh\mu}(\{(x,x')\colon D_{W'}(x')>\lambda+\eps'\text{ and } \big|D_W(x)-D_{W'}(x')\big|< \eps'\})
\\
&\leq {\mu'}(\{D_{W'}>\lambda +\eps'\}).
\end{split}
\]
Taking $\eps'\downarrow \eps$ and using monotone convergence we obtain the
first inequality in the statement of the lemma.  The second is proved in the
same way.
\end{proof}

\begin{proof}[Proof of Proposition~\ref{pro:deg-conv}]
Let $\eps_n=\delta_\square(\W_n,\W)$, and choose $n$ large enough so that
$\eps_n< \lambda$. By  Lemma~\ref{lem:D_W},
\[
\mu_n(\{D_{W_n}> \lambda\})\leq\mu(\{D_{ W }> \lambda-\eps_n\})+\eps_n.
\]
Since $\mu(\{D_{ W }> \lambda\})$ is assumed to be continuous at $\lambda$,
this gives
\[
\limsup_{n\to\infty}\mu_n(\{D_{W_n}> \lambda\})\leq\mu(\{D_{ W }> \lambda\}).
\]
The matching lower bound on the $\liminf$ is proved in the same way.
\end{proof}

\section{\texorpdfstring{Representation of Graphons Over $\R_+$}{Representation of Graphons Over R\_+}}
\label{sec:equiv}

In this appendix we will prove that every graphon is equivalent to a graphon
over $\R_+$ (Proposition~\ref{prop7}), and prove that under certain
assumptions on the underlying measure space of a graphon the cut metric can
be defined in a number of equivalent ways (Proposition \ref{prop10}).

The first statement of the following lemma is a generalization of the
analogous result for probability spaces, which was considered by
\citet*[Corollary~3.3]{BCL10} and \citet*[Lemma~7.3]{janson-survey}. It will
be used to prove Theorem~\ref{prop1} and Proposition~\ref{prop7}. We proceed
similarly to the proof by \citet*[Lemma~7.3]{janson-survey}, but in this case
we also need to argue that underlying measure space of the constructed
graphon is $\sigma$-finite, and we include an additional result on atomless
measure spaces.

\begin{lemma}
Every graphon $\mcl W=(W,\scr{S})$ with $\scr{S}=(S,\cS,\mu)$ is a pullback
by a measure-preserving map of a graphon on some $\sigma$-finite Borel
measure space. If $\scr S$ is atomless, the $\sigma$-finite Borel space can
be taken to be atomless as well. \label{prop6}
\end{lemma}

\begin{proof}
Let $S_0:=\emptyset$, and let $(S_k)_{k\in\N}$ be such that for each
$k\in\N$, we have $S_k\in\cS$, $S_k\subseteq S_{k+1}$, $\mu(S_k)<\infty$, and
$\bigcup_{k\in\N} S_k=S$. We claim that we can find a sequence of sets
$(A_i)_{i\in\N}$ satisfying the following properties: (i) if $\mcl
A:=\{A_i\,:\,i\in\N\}$ and $\cS_0:=\sigma(\mcl A)$, then $W$ is
$\cS_0\times\cS_0$-measurable, (ii) for all $k\in\N$ there exists $i\in\N$
such that $A_i=S_{k}\setminus S_{k-1}$, (iii) for each $i\in \N$ there exists
a $k\in\N$ such that $A_i\subseteq S_{k}\backslash S_{k-1}$, and (iv)
$\bigcup_{i\in \N} A_i=S$.  A set $\cA$ satisfying (i) can be constructed by
noting that each level set $\{(x_1,x_2)\in S\times S\,:\,W(x_1,x_2)<q\}$ for
$q\in\Q$ is measurable with respect to $\sigma(\sigma(\cA_q)\times
\sigma(\cA_q))$ for some countable set $\cA_q$ (this follows, for example, by
Lemma~3.4 in the paper of \citel{BCL10}). By adding the sets $S_{k}\setminus
S_{k-1}$ to $\cA$ we obtain a collection of sets satisfying (i) and (ii).
Given a set $\wt{\mcl A}$ satisfying (i) and (ii), we can easily obtain an
$\mcl A$ satisfying (i)--(iii) by replacing each $A\in \wt{\mcl A}$ with the
countable collection of sets $\{A\cap(S_{k}\backslash S_{k-1})\,:\,k\in\N\}$.
Finally, (ii) implies (iv).

Let $\mcl C=\{0,1\}^\infty$ be the Cantor cube equipped with the  product
topology, and define $\phi\colon S\to\mcl C$ by $\phi(x):=(\1_{x\in
A_i})_{i\in\N}$. Let $\nu$ be the pushforward measure of $\mu$ onto $\mcl C$
equipped with the  Borel $\sigma$-algebra. We claim that $\nu$ is a
$\sigma$-finite measure on $\mcl C$. For each $k\in\N$ define $\wh C_k:=\{
(a_i)_{i\in\BB N} \in \mcl C\,:\, a_i=0 \text{ if } A_i\not\subseteq
S_{k}\backslash S_{k-1} \}$ and $\wt C_k:=(\bigcup_{i\leq k}\wh C_{i})\cup
\wh C_0$, where $\wh C_0:= \mcl C\setminus (\bigcup_{i\in\N}\wh C_{i})
\subseteq \mcl C\setminus \phi(S)$, and observe that all the subsets of $\mcl
C$ just defined are measurable. The claim will follow if we can prove that
(a) $\nu(\wt C_k)<\infty$ for each $k\in\N$, and (b) $\bigcup_{k\in\N} \wt
C_k=\mcl C$. Property (a) is immediate since from the definition of $\nu$,
the fact $\nu(\wh C_0)=0$, and the properties (ii) and (iii) of $\cA$, which
imply that $\nu(\wh C_k) = \mu(\phi^{-1}(\wh C_k)) = \mu(S_{k}\backslash
S_{k-1})<\infty$. To prove (b) let $(a_i)_{i\in\N}\in \mcl C$. We want to
prove that $(a_i)_{i\in\N}\in \wt C_k$ for some $k\in\N$. If
$(a_i)_{i\in\N}\not\in \phi(S)$ we have $x\in\wh{\mcl C}_0$, so
$(a_i)_{i\in\N}\in \wt C_k$ for all $k\in\N$. If $(a_i)_{i\in\N}=\phi(x)$ for
some $x\in S$, then $x\in S_{k}\backslash S_{k-1}$ for exactly one $k\in\N$,
so $(a_i)_{i\in\N}=\phi(x)\in\wh C_k\subseteq\wt C_k$.

The argument in the following paragraph is similar to the proof by
\citet*[Lemma~7.3]{janson-survey}, but we repeat it for completeness. Since
the $\sigma$-field on $S$ generated by $\phi$ equals $\cS_0$, the
$\sigma$-field on $S\times S$ generated by $(\phi,\phi)\colon S^2\to \mcl
C^2$ equals $\cS_0\times\cS_0$. Since $W$ is measurable with respect to
$\cS_0\times\cS_0$ we may use the Doob-Dynkin Lemma (see, for example, the
book of \citel{kallenberg-prob}, Lemma~1.13) to conclude that there exists a
measurable function $V\colon \mcl C^2\to[0,1]$ such that $W=V^\phi$. We may
assume $V$ is symmetric upon replacing it by $\frac 12(V(x,x')+V(x',x))$.
This completes the proof of the main assertion, since $V$ is a graphon on a
$\sigma$-finite Borel measure space.

Finally we will prove the last claim of the lemma, i.e., that if $\scr S$ is
atomless we may take $\nu$ to be atomless as well. To prove this claim it is
sufficient to establish that the set $\mcl A$ in the above argument can be
modified in such a way that $\nu(x)=0$ for each $x\in \mcl C$. Given a
collection of sets $\mcl A$ satisfying (i)--(iv) above we define a new
collection of sets $\wt{\mcl A}$ as follows. First define $\mcl A_1:=\mcl A$,
and then define $\mcl A_k$ for $k>1$ inductively as follows. For any $k>1$
and $A\in \mcl A_{k-1}$, let $A^1\in\cS$ and $A^2\in\cS$ be disjoint sets
with union $A$ such that $\mu(A^1)=\mu(A^2)=\frac 12\mu(A)$; note that such
sets $A^1$ and $A^2$ can be found since $\scr S$ is atomless. Then define
$\mcl A_k:=\{A^1\,:\,A\in\mcl A_{k-1} \}\cup \{A^2\,:\,A\in\mcl A_{k-1} \}$,
and finally define $\wt{\mcl A}=\bigcup_{k\in\N} \mcl A_k$. Then $\wt{\mcl
A}$ is countable, satisfies (i)--(iv), and by defining the measure $\nu$
using $\wt{\mcl A}$ instead of $\mcl A$ we have $\nu(x)=0$ for every
$x\in\mcl C$. Proceeding as above with $\wt{\mcl A}$ instead of $\mcl A$ we
get $W=V^\phi$, where $V$ is a graphon over an atomless $\sigma$-finite Borel
space.
\end{proof}

Proposition~\ref{prop7} follows immediately from the following lemma, whose
proof in turn  follows a similar strategy as a proof by
\citet*[Theorem~7.1]{janson-survey}.

\begin{lemma}
\label{lem:prop7} Let $\mcl W=(W,\scr{S})$ be a graphon over an arbitrary
$\sigma$-finite space $\scr S$.
\begin{itemize}
\item[(i)] There are two graphons $\mcl W'=(W',\R_+)$
and  $\W''=(W'',\scr S'')$ and measure-preserving maps
    $\phi\colon S\to S''$ and $\phi'\colon [0,\mu(S)) \to S''$ such that
    $W=(W'')^{\phi}$ and $W'$ is the trivial extension of $(W'')^{\phi'}$
    from $[0,\mu(S))$ to $\R_+$.

\item[(ii)] If $\scr S$ is a Borel measure space, then we can find a
    measure-preserving map \[\phi'\colon [0,\mu(S))\to S\] such that
    $\W^{\phi'}$ is a graphon over $[0,\mu(S))$ equipped with the Borel
    $\sigma$-algebra and Lebesgue measure.

\item[(iii)]If $\scr S$ is an atomless Borel measure space, we may take
    $\phi'$ in (ii) to be an isomorphism between $\scr S$ and $[0,\mu(S))$.
\end{itemize}
\end{lemma}

\begin{proof}
We start with the proof of (ii) and (iii). If $\scr S$ is an atomless Borel
space the statement is immediate from Lemma~\ref{prop9}.  If $\scr S$ has
atoms, we define a graphon $\wt{\mcl W}=(\wt W,\wt{\scr S})$ where $\wt{\scr
S}=(\wt S,\wt{\cS},\wt\mu)=(S\times[0,1],\cS\times\mcl B,\mu\times\lambda))$
is the product measure space and $\wt W:=(W)^\pi$, with $\pi\colon
S\times[0,1]\to S$ being the projection. Since $\wt{\scr S}$ is an atomless
Borel space we may again use Lemma~\ref{prop9}, giving the existence of an
isomorphism $\psi\colon [0,\wt\mu(\wt S))\to\wt S$ such that ${\wt W}^\psi$
is a graphon over $[0,\wt\mu(\wt S))$ equipped with Lebesgue measure.
Observing that $\wt\mu(\wt S)=\mu(S)$, we obtain statement (ii) with
$\phi'=\pi\circ\psi$.

To prove (i)  we use that by Lemma~\ref{prop6}, $\W$ can be expressed as
$(\W'')^\phi$ for a graphon $\W''$ on some Borel  space $\scr
S''=(S'',\cS'',\mu'')$ and some measure-preserving map $\phi\colon S\to S''$.
We then apply the just proven statement (ii) to the graphon $\W''$, and
define $\W'$ to be the trivial extension of $(\W'')^{\phi'}$ from
$[0,\mu(S))$ to $\R_+$.
\end{proof}

\begin{proof}[Proof of Proposition~\ref{prop7}]
The statement of the proposition follows from Lemma~\ref{lem:prop7} by
observing that
$\delta_\square(\W,\W')\leq\delta_\square(\W,\W'')+\delta_\square(\W'',\W')=
\delta_\square((\W'')^\phi,\W'')+\delta_\square(\W'',(\W'')^{\phi'})= 0$.
\end{proof}

The following proposition provides equivalent definitions of the cut metric
$\delta_\square$ under certain assumption on the underlying measure spaces.
See papers by \citet*[Lemma~3.5]{denseconv1} and \citet*[Theorem~6.9]{janson-survey} for
analogous results for probability spaces.

\begin{proposition}
For $j=1,2$ let $\mcl W_j=(W_j,\scr{S}_j)$ with $\scr{S}_j=(S_j,\cS_j,\mu_j)$
be a graphon satisfying $\mu_j(S_j)=\infty$. Then the following identities
hold, and thus (a)--(e) provide alternative definitions of $\delta_\square$
under certain assumptions on the underlying measure spaces:
\begin{itemize}
\item[(a)] If $S_j$ are Borel spaces, then $\delta_\square(\mcl W_1,\mcl
    W_2)=\inf_{\psi_1,\psi_2} \|W_1^{\psi_1}-W_2^{\psi_2}\|_\square$, where
    we take the infimum over measure-preserving $\psi_j\colon \R_+\to S_j$
    for $j=1,2$, where $\R_+$ is equipped with the Borel $\sigma$-algebra
    and Lebesgue measure.
\item[(b)] If $S_j$ are atomless Borel spaces, then $\delta_\square(\mcl
    W_1,\mcl W_2)=\inf_{\psi} \|W_1-W_2^{\psi}\|_\square$, where we take
    the infimum over measure-preserving $\psi\colon S_1\to S_2$.
\item[(c)] If $S_j$ are atomless Borel spaces, then $\delta_\square(\mcl
    W_1,\mcl W_2)=\inf_{\psi} \|W_1-W_2^{\psi}\|_\square$, where we take
    the infimum over isomorphisms $\psi\colon S_1\to S_2$.
\item[(d)] If $S_j=\BB R_+$, then $\delta_\square(\mcl W_1,\mcl
    W_2)=\inf_{\wt\sigma} \|W_1-W_2^{\wt\sigma}\|_\square$, where we take
    the infimum over all interval permutations $\wt\sigma$ (i.e.,
    $\wt\sigma$ maps $I_i$ to $I_{\sigma(i)}$ for some permutation $\sigma$
    of the non-negative integers, and $I_i:=[ih,(i+1)h]$ for some $h>0$).
\item[(e)] For $j=1,2$ let $(S_j^k)_{k\in\N}$ be increasing sets satisfying
    $\mu_j(S_j^k)<\infty$ and $\bigcup_{k\in\N} S_j^k=S_j$. Then
    $\delta_\square(\mcl W_1,\mcl W_2)=\lim_{k\rta\infty}
    \delta_\square(\mcl W_1|_{S_1^k},\mcl W_2|_{S_2^k})$, where
    $\cW_j|_{S_j^k}:=(W_j\1_{S_j^k\times S_j^k},\scr{S}_j^k)$ and
    $\scr{S}_j^k$ is the restriction of $\scr{S}^k$ to $S_j^k$.
\end{itemize}
\label{prop10}
\end{proposition}

\begin{proof}[Proof of Proposition~\ref{prop10}]
Let $\delta_\square^{(a)}$, $\delta_\square^{(b)}$, $\delta_\square^{(c)}$,
and $\delta_\square^{(d)}$ denote the right sides of the equalities in
(a), (b), (c), and (d), respectively. For $j=1,2$ fix some arbitrary
sequence $(S_j^k)_{k\in\N}$ satisfying $\mu_j(S_j^k)<\infty$ for all
$k\in\N$, $S_j^k\subseteq S_j^{k+1}$, and $\bigcup_{k\in\N} S_j^k=S_j$.
Define $\delta_\square^{(e)}$ and $\delta_\square^{(e')}$ by
\[
\delta_\square^{(e)}(\mcl W_1,\mcl W_2) := \limsup_{k\rta\infty} \delta_\square(\mcl W^k_1,\mcl W^k_2)
\quad\text{ and }\quad
\delta_\square^{(e')}(\mcl W_1,\mcl W_2) := \liminf_{k\rta\infty} \delta_\square(\mcl W^k_1,\mcl W^k_2),
\]
where $\W_k^j=\cW_j|_{S_j^k}$. By Lemma~\ref{prop9} it is sufficient to
consider the case $\scr S=(\R_+,\mcl B,\lambda)$ in (b) and (c), since we can
consider graphons $(W_j^{\phi_j},\R_+)$ on $\R_+$, which satisfy
$\delta_\square((W_j^{\phi_j},\R_+),\W_j)=0$, by using measure-preserving
transformations $\phi_j\colon \R_+\to S_j$. Under this assumption we have
$\delta_\square \leq \delta_\square^{(b)}\leq \delta_\square^{(c)} \leq
\delta_\square^{(d)}$, since we take the infimum over smaller and smaller
sets of maps. By definition, $\delta_\square^{(e')}\leq\delta_\square^{(e)}$.
To complete the proof of the proposition it is therefore sufficient to prove
the following results: (i)
$\delta_\square^{(e)}\leq\delta_\square\leq\delta_\square^{(e')}$ for general
$\sigma$-finite measure spaces $\scr S_1,\scr S_2$ of infinite measure, (ii)
$\delta_\square^{(d)} \leq \delta_\square$ for $\scr S_1=\scr S_2=(\R_+,\mcl
B,\lambda)$, and (iii) $\delta^{(a)}_\square=\delta^{(c)}_\square$.

We will start by proving (i). Since
$\lim_{k\rta\infty}\|W_j-W_j\1_{S_j^k\times S_j^k}\|_1=0$, Lemma
\ref{prop15} implies that it is sufficient to prove
$\delta_\square^{(e)}\leq\delta_\square\leq\delta_\square^{(e')}$ for the
case when $\op{supp}(W_j)\subseteq S_j^k\times S_{j}^k$ for some
$k\in\N$. Under this assumption
$\delta_\square(\W_1,\W_2)=\delta_\square(\W_1|_{S_1^{k'}},\W_2|_{S_2^{k'}})$
for all $k'\geq k$, and (i) follows.

Now we will prove (ii). Since $\lim_{M\rta\infty}\|W_j-W_j \1_{|W_j|\leq
M}\1_{[0,M]^2}\|_1 = 0$ by the dominated convergence theorem, as above we may
assume by Lemma~\ref{prop15} that there is an $M>0$ such that $W$ is bounded
and $\op{supp}(W_j)\subseteq[0,M]^2$ for $j=1,2$. For $j=1,2$ define
$\wh{\W}_j=(\wh W_j,[0,M])$, where $\wh W_j:=W_j|_{[0,M]^2}$ is a bounded
graphon on $[0,M]^2$. By Lemma~\ref{prop21} in the current paper and Lemma
3.5 of \citet*{denseconv1} (or, equivalently, Theorem~6.9 of
\citel{janson-survey}),
\[
\delta_\square(\mcl W_1,\mcl W_2)
= \delta_\square(\wh {\mcl W}_1,\wh{\mcl W}_2)
= \inf_{\wh{\sigma}} \|\wh W_1-\wh W_2^{\wh\sigma}\|_\square
\geq \inf_{\wt\sigma} \|W_1-W_2^{\wt\sigma}\|_\square,
\]
where $\wh{\sigma}$ is an interval permutation of $[0,M]$ and $\wt{\sigma}$
is an interval permutation of $\R_+$.

Finally we will prove (iii). By Lemma~\ref{lem:prop7} there are
measure-preserving maps $\phi_j\colon \R_+\to S_j$ such that
$\delta_\square(\W_j,(\W_j)^{\phi_j})=0$. The triangle inequality then
implies that
$\delta_\square(\W_1,\W_2)=\delta_\square((\W_1)^{\phi_1},(\W_2)^{\phi_2})$.
Since $(\W_1)^{\phi_1}$ and $(\W_2)^{\phi_2}$ are graphons over atomless
Borel spaces, it follows that $\delta_\square^{(a)}=\delta_\square^{(c)}$.
\end{proof}

\begin{remark}
\label{rem:alt-def-delta-p} The proof of the above proposition clearly
generalizes to the metric $\delta_1$, the only additional ingredient being
the analogue of a result by \citet*[Theorem~6.9]{janson-survey} for the
metric $\delta_1$ \cite[Remark~6.13]{janson-survey}.  Using the results and
proof techniques of \citet*[Appendix~A]{w-estimation} instead of
\citet*[Theorem~6.9]{janson-survey}, it can also be generalized to the metric
$\delta_p$ for $p>1$, again provided both graphons are non-negative and in
$L^p$.
\end{remark}

We close this appendix by proving Proposition~\ref{prop:dcut-d1}. In fact, we
will prove a generalization of this proposition for the invariant $L^p$
metric $\delta_p$. The second statement of this  proposition involves the
distance $\delta_p(\W_1,\W_2)$ of graphons $\W_1=(W_1,\R_+)$ and
$\W_2=(W_2,\R_+)$ that are not necessarily non-negative, which means we do
not have Proposition~\ref{prop:delta-p} at our disposal to guarantee that
$\delta_p$ is well defined.  We avoid this problem by defining $\delta_p$ as
in Proposition~\ref{prop10}, i.e., by setting
\[
\delta_p(\W_1,\W_2):=\inf_{\phi\colon\R_+\to\R_+}\|W_1-W_2^\phi\|_p
\]
with the infimum going over isomorphisms. Note that by
Remark~\ref{rem:alt-def-delta-p}, for non-negative graphons in $L^p$, this
definition is equivalent to the one given at the beginning of
Appendix~\ref{sec4}.

\begin{proposition}
\label{prop:dcut-dp}
Let $p\geq 1$, and let $\W_1$ and $\W_2$ {be} graphons
in $L^p$.  Then
\begin{itemize}
\item[(i)] $\delta_1(\W_1,\W_2)=0$ if and only if
    $\delta_\square(\W_1,\W_2)=0$, and
\item[(ii)] if $\W_1$ and $\W_2$ are non-negative or graphons over $\R_+$,
    then $\delta_p(\W_1,\W_2)=0$ if and only if
    $\delta_\square(\W_1,\W_2)=0$.
\end{itemize}
\end{proposition}

Proposition~\ref{prop:dcut-dp} (and hence Proposition~\ref{prop:dcut-d1}) and
Proposition~\ref{prop:dcut-coupling} follow from the next proposition.

\begin{proposition}
For $i=1,2$, let $\W_i=(W_i,\scr S_i)$ be a graphon over a Borel space $\scr
S_i$ such that $\delta_\square(\W_1,\W_2)=0$ and $\mu_1(S_1)=\mu_2(S_2)$.
Then there exists a measure $\mu$ on $S_1\times S_2$ such that
\begin{itemize}
\item[(i)] $\|W_1^{\pi_1}-W^{\pi_2}_2\|_{\square,\mu}=0$,
\item[(ii)] the first (resp.\ second) marginal of $\mu$ is dominated by
    $\mu_1$ (resp.\ $\mu_2$), i.e., for any $A\in\mcl S_1$ (resp.\
    $A\in\mcl S_2$) we have $\mu(A\times S_2)\leq \mu_1(A)$ (resp.\
    $\mu(S_1\times A)\leq \mu_2(A)$),
\item[(iii)] if $A\in\mcl S_1$ is such that $\mu(A\times S_2)<\mu_1(A)$
    then $\mu_1(A\cap E_1)>0$, where
\begin{equation*}
E_i:=\left\{x\in S_i\,:\, \int_{S_i}|W_i(x,x')|\,dx'=0 \right\}\quad \text{for } i=1,2,
\end{equation*}
and the same property holds with the roles of $\W_1$ and $\W_2$
interchanged, and
\item[(iv)] in particular, if $\mu_1(E_1)=\mu_2(E_2)=0$ with $E_i$ as in
    (iii), then $\mu$ is a coupling measure.
\end{itemize}
\label{prop27}
\end{proposition}

\begin{remark}
The analogous statement to Proposition~\ref{prop27} for graphons over
probability spaces (see, for example, the paper of
\citel{janson-survey}, Theorem~6.16) states that when the underlying space
is a Borel probability space, the infimum in the definition of the cut
distance using couplings is attained. Proposition~\ref{prop27} says that the
same result is true in our setting of $\sigma$-finite measure spaces if we
make two additional assumptions: (a) the cut distance between the graphons is
zero, and (b) $\mu_i(E_i)=0$ for $i=1,2$, where $E_i$ is defined in part
(iii) of the proposition. We remark that both of these assumptions are
necessary; see the examples following this remark.

\citet*[Theorem~5.3]{janson16} proves a related result, stating that if the
cut distance of two graphons over $\sigma$-finite Borel spaces is zero, then
there are trivial extensions of these graphons such that the extensions can
be coupled so as to be equal almost everywhere. Proposition~\ref{prop27}
implies a similar result, namely Proposition~\ref{prop:dcut-coupling}, which
states that under this assumption, the restrictions of the two graphons to
the sets $S_i\setminus E_i$ can be coupled so that they are equal a.e.  To
see this, we note that by Proposition~\ref{prop:dcut-dp} two graphons
$W_1,W_2$ with cut distance zero have distance zero in the metric $\delta_1$,
which in turn implies that $|W_1|$ and $|W_2|$ have distance zero in
$\delta_1$ and hence in $\delta_\square$. By Lemma~\ref{lem:D_W}, this in
turn implies that $\mu_1(S_1\setminus E_1)=\mu_2(S_2\setminus E_2)$, which
allows us to use Proposition~\ref{prop27} to deduce the claim. \label{rmk1}
\end{remark}

\begin{examp}
Assumption (a) in Remark~\ref{rmk1} is necessary by the following
counterexample, which illustrates that there are graphons $\W_1,\W_2$ over
$\R_+$ such that
$\|W_1^{\pi_1}-W_2^{\pi_2}\|_{\square,\mu}>\delta_\square(\W_1,\W_2)$ for all
coupling measures $\mu$. Let $\W_1=(W_1,\R_+)$ be an arbitrary graphon such
that $W_1$ is strictly positive everywhere, and let $\W_2=(W_2,\R_+)$ be
defined by $W_2(x,y)=W_1(x-1,y-1)$ for $x,y\geq 1$, $W_2(x,y)=-1$ for
$x,y\in[0,1)$, and $W_2(x,y)=0$ otherwise. First observe that
$\delta_\square(\W_1,\W_2)\leq 1$, since if $\phi_n\colon\R_+\to\R_+$ is
defined by $\phi_n(x)=x+1$ for $x\in[0,n]$, $\phi_n(x)=x-n$ for
$x\in(n,n+1]$, and $\phi_n(x)=x$ for $x>n+1$, then
\[
\lim_{n\rta\infty} \|W_1-W_2^{\phi_n}\|_{\square} = 1.
\]
Then observe that $\|W_1^{\pi_1}-W_2^{\pi_2}\|_{\square,\mu}>1$ for all
coupling measures $\mu$, since if $S=\R_+^2$ and $T=\R_+\times[0,1]$ then
\[
\begin{split}
\|W_1^{\pi_1}-W_2^{\pi_2}\|_{\square,\mu}
&\geq \left| \int_{S\times T} W_1^{\pi_1}-W_2^{\pi_2}\,d\mu\,d\mu \right|\\
&= \left| \int_{\R_+\times T} W_1(x,\pi_1(y))\,d\lambda\,d\mu- \int_{\R_+\times[0,1]}W_2\,d\lambda\,d\lambda \right|>1.
\end{split}
\]
Assumption (b) is necessary by the following counterexample, which
illustrates that there are non-negative graphons $\W_1,\W_2$ over $\R_+$ such
that $\delta_\square(\W_1,\W_2)=0$ and $\|W_1-W_2\|_{\square,\mu}>0$ for all
coupling measures $\mu$. Letting $\W_1$ and $\W_2$ be defined as above,
except that $W_2(x,y)=0$ for all $x,y\in\R_+$ for which either $x<1$ or
$y<1$, we proceed as in case (a) to conclude that the graphons satisfy the
desired property.
\end{examp}

\begin{proof}[Proof of Proposition~\ref{prop27}]
First we note that for $\mu_1(S_1)=\mu_2(S_2)=1$, the proposition follows
immediately from a result of \citet*[Theorem~6.16]{janson-survey}, which in
fact gives $\mu$ as a coupling of $\mu_1$ and $\mu_2$. The case
$\mu_1(S_1)=\mu_2(S_2)=c<\infty$ with $c\neq 1$ can be reduced to the case
$c=1$ by considering the graphons $\W_i'=(W_i,\scr S_i')$ where $\scr S_i'$
is obtained from $\scr S_i$ by multiplying the measures $\mu_i$ by $1/c$,
turning them into probability measures.  All that is left to consider is
therefore the case $\mu_1(S_1)=\mu_2(S_2)=\infty$.

Next we argue that we may assume $\scr S_i$ is atomless for $i=1,2$. Assuming
the proposition has been proved for the case of atomless Borel measure
spaces, and given graphons $\W_i$ over arbitrary Borel measure spaces, we
define graphons $\wt{\W}_i$ over the measure space $\wt{\scr S}_i$ defined as
the product of $\scr S_i$ and $[0,1]$, such that $\wt{\W}_i=\W_i^{\phi_i}$
for the projection map $\phi_i\colon S_{i}\times[0,1]\to S_{i}$. Assume that
$\wt\mu$ is a measure on $\wt S_1\times \wt S_2$ such that the statements of
the proposition hold for $\wt\mu$ and $\wt{\W}_1,\wt{\W}_2$. Defining a
measure $\mu$ for $S_1\times S_2$ by letting $\mu$ be the pushforward of
$\wt\mu$ for the map $\wt S_1\times \wt S_2\to S_1\times S_2$ sending
$\big((x_1,r_1), (x_2,r_2)\big)\mapsto (x_1,x_2)$, one easily checks that
$\mu$ is a measure satisfying the conclusions of the proposition for $\W_1$
and $\W_2$. It follows that we may assume the spaces $\scr S_i$ are atomless
Borel measure spaces, and by Lemma~\ref{prop9} we may assume that they are
$\R_+$ equipped with the Lebesgue measure; we will make these assumptions in
the remainder of the proof.  In particular, we will no longer use the
notation $\mu_1$ and $\mu_2$ from the proposition statement, since they are
now both Lebesgue measure $\lambda$; it will be convenient to use the
notation $\mu_n$ for other purposes.

Consider a sequence of coupling measures $(\mu_n)_{n\in\N}$ such that
$\|W_1^{\pi_1}-W_2^{\pi_1}\|_{\square,\mu_n}\rta 0$. For any given $M>0$, let
$\mu_n^M=\mu_n|_{[0,M]^2}$.  The measures $\mu_n^M$ are not necessarily
coupling measures, but their marginals are dominated by the Lebesgue measure
on $[0,M]$, and they satisfy $\lim_{n\rta\infty} \|W_1^{\pi_1}-W_2^{\pi_2}
\|_{\square,{\mu_n^M}}=0$. Furthermore, as a sequence of measures of
uniformly bounded total mass over a compact metrizable space, they have a
subsequence that converges in the weak topology \cite[Theorem~5.1]{billingsley-book},
i.e., in the topology in which the integrals over all
continuous functions on $[0,M]^2$ converge. Let $\mu^M$ be some subsequential
limit, and note that as a limit of a sequence of measures having this
property, the marginals of $\mu^M$ are dominated by the Lebesgue measure on
$[0,M]$ as well. Note also that
$\mu_n^M\times\mu_n^M$ converges weakly to $\mu^M\times \mu^M$ along any
subsequence on which $\mu_n^M$ converges weakly to $\mu^M$
\citep*[Theorem~2.8]{billingsley-book}. We will argue
that
\begin{equation}
\left| \int_{A\times B}\Bigl( W_1^{\pi_1}-W_2^{\pi_2}\Bigr)\,d\mu^M\,d\mu^M \right|=0
\label{eq40}
\end{equation}
for all measurable subsets $A,B\subseteq [0,M]^2$.

Indeed, given two such subsets  and $\eps>0$, let $\wt W_i$ be continuous
functions over $[0,M]^2$ such that $\|W_i-\wt
W_i\|_{1,\lambda|_{[0,M]^2}}\leq \eps$ for $i=1,2$, and let
$f,g\colon[0,M]^2\to [0,1]$ be continuous functions such that $(\|\wt
W_1\|_\infty+\|\wt W_2\|_\infty)\|{\bf 1}_{A}-f\|_{1,\mu^M}\leq \eps$ and
$(\|\wt W_1\|_\infty+\|\wt W_2\|_\infty)\|{\bf 1}_{B}-g\|_{1,\mu^M}\leq \eps$
(existence of appropriate functions $\wt W_i,f,g$ follows from, for example,
\citel{stroock-integration}, Corollary~3.2.15). Using the fact that $|\int
f(x)g(y)(W_1^{\pi_1}(x,y)-W_2^{\pi_2}(x,y))\,d\mu_n^M\,d\mu_n^M| \leq
\|W_1^{\pi_1}-W_2^{\pi_2} \|_{\square,{\mu_n^M}}$ and the fact that the
marginals of $\mu_n^M$ and $\mu^M$ are dominated by the Lebesgue measure on
$[0,M]$, this allows us to conclude that for some sufficiently large $n$
chosen from the subsequence along which $\mu_n^M$ converges,
\begin{align*}
\bigg| \int_{A\times B}\Bigl( W_1^{\pi_1}-W_2^{\pi_2}\Bigr)\,d\mu^M\,d\mu^M \bigg|
&\leq
 \left|  \int_{A\times B}\Bigl(\wt W_1^{\pi_1}(x,y)-\wt W_2^{\pi_2}(x,y)\Bigr)\,d\mu^M\,d\mu^M
 \right|+2\eps
\\
&\leq
 \left| \int f(x)g(y) \Bigl(\wt W_1^{\pi_1}(x,y)-\wt W_2^{\pi_2}(x,y)\Bigr)\,d\mu^M\,d\mu^M
 \right|+4\eps
\\
&\leq
 \left| \int f(x)g(y) \Bigl(\wt W_1^{\pi_1}(x,y)-\wt W_2^{\pi_2}(x,y)\Bigr)\,d\mu_n^M\,d\mu_n^M
 \right|+5\eps
 \\
&\leq
 \left| \int f(x)g(y) \Bigl(W_1^{\pi_1}(x,y)- W_2^{\pi_2}(x,y)\Bigr)\,d\mu_n^M\,d\mu_n^M
 \right|+7\eps
 \\
&\leq
 \|W_1^{\pi_1}-W_2^{\pi_2}
\|_{\square,{\mu_n^M}}+7\eps \leq 8\eps.
\end{align*}
Since $\eps$ was arbitrary, this proves \eqref{eq40}.

For each $M\in\N$ we let $\mu^M$ be a measure as in the previous paragraph.
We may assume the subsequence along which $\mu_n^{M+1}$ converges to
$\mu^{M+1}$ is a subsequence of the subsequence along which $\mu_n^M$
converges to $\mu^M$. This implies that $\mu^{M+1}|_{[0,M]^2}=\mu^M$ for all
$M\in\N$, so there is a measure $\mu$ on $\R_+^2$ such that
$\mu|_{[0,M]^2}=\mu^M$. Furthermore, there is a subsequence of
$(\mu_n)_{n\in\N}$ converging weakly to $\mu$, such that for any $M\in\N$ the
measures $\mu_n|_{[0,M]^2}$ converge weakly to $\mu|_{[0,M]^2}$ along this
subsequence, and as a limit of measures with these properties, the measure
$\mu$ satisfies (ii), as well as
\[
\sup_{A,B\subseteq\R_+^2} \left| \int_{A\times B} W_1^{\pi_1}-W_2^{\pi_2}\,d\mu\,d\mu \right|=0,
\]
where, \emph{a priori}, the supremum is over measurable, bounded subsets
$A,B\subset \R_+^2$. But it is easy to see that if the supremum over these
sets is zero, then the same holds for the supremum over all measurable
subsets $A,B\subseteq \R_+^2$ (use (ii) to conclude that the integrand is in
$L^1$, which means it can be approximated by functions over bounded subsets
of $\R_+^4$). The property (i) of $\mu$ follows.

It remains to prove that $\mu$ satisfies (iii), since (iv) follows
immediately from (iii). Recall that the by definition of the measures
$\mu_n$,
\[
\sup_{A_1,A_2,K}\left|\int_{(A_1\times [K,\infty))\times (A_2\times \R_+)} W_1^{\pi_1}-W_2^{\pi_2}\,d\mu_n\,d\mu_n\right| \rta 0,
\]
where the supremum is over $A_1,A_2\subseteq\R_+$ and $K\geq 0$. Fix any
$\ep>0$, and observe that for all $K>1$ sufficiently large,
\[
\sup_{A_1,A_2}\left|\int_{(A_1\times [K,\infty))\times (A_2\times \R_+)}W_2^{\pi_2}\,d\mu_n\,d\mu_n\right|
\leq \sup_{A_1,A_2}\int_{[K,\infty)\times \R_+}|W_2|\,d\lambda\,d\lambda
 <\ep,
\]
so
\[
\begin{split}
\limsup_{n\rta\infty}
\sup_{A_1,A_2}&\left|\int_{(A_1\times [K,\infty))\times A_2} W_1(\pi_1(x),x')\,d\mu_n(x)\,d\lambda(x')\right|\\
&\qquad\qquad\qquad\qquad=
\limsup_{n\rta\infty}\sup_{A_1,A_2}\left|\int_{(A_1\times [K,\infty))\times (A_2\times \R_+)} W_1^{\pi_1}\,d\mu_n\,d\mu_n\right|
< \ep.
\end{split}
\]
Fix any $A_1,A_2\subseteq\R_+$, and observe from the above that for $K$
sufficiently large,
\[
\limsup_{n\rta\infty}\left|\int_{A_1\times A_2} W_1(x,x')\,d(\lambda-\mu^{1,K}_n)(x)\,d\lambda(x')\right|
< \ep,
\]
where $\mu^{1,K}_n$ is the projection of $\mu_n|_{\R_+\times [0,K]}$ onto the
first coordinate. Choose $\wt W_1\in C_c(\R_+^2)$ such that $\|\wt
W_1-W_1\|_1<\ep$, where $C_c(\R_+^2)$ is the space of continuous and
compactly supported functions on $\R_+^2$ (the existence of such a function
follows again from, for example, \citel{stroock-integration}, Corollary~3.2.15). For all $K>1$ sufficiently large,
\[
\left|\int_{A_1\times A_2} \wt W_1(x,x')\,d(\mu^1-\mu^{1, K})(x)\,d\lambda(x')\right|<\ep,
\]
where $\mu^{1,K}$ (resp.\ $\mu^1$) is the projection of $\mu|_{\R_+\times
[0,K]}$ (resp.\ $\mu$) onto the first coordinate. Next we claim that
$\mu_n^{1,K}|_{[0,K']}$ converges weakly to $\mu^{1,K}|_{[0,K']}$ for any
$K,K'>0$. To see that, we need to show that for any continuous function
$f\colon[0,K']\to\R_+$ the associated integral converges when $n\rta\infty$.
To this end, we approximate the function $(x,x')\mapsto f(x)\1_{x'\in[0,K]}$
(with $f(x)=0$ for $x>K'$) by a function  $g\colon\R_+^2\to\R$, where
$g(x,x')=\wh f(x) \chi(x')$, $\wh f\colon\R_+\to \R_+$ is a continuous
function with support in $[0,K']$ approximating $f$ and satisfying $\|\wh
f\|_\infty\leq\|f\|_\infty$, and $\chi\colon\R_+\to [0,1]$ is a continuous
function with support in $[0,K]$ approximating the indicator function of the
set $[0,K]$. Since $\wh f$ and $\chi$ can be chosen to be arbitrarily close
approximations in the $L^1$ norm and the marginals of $\mu_n$ are given by
Lebesgue measure, this implies the claim. Therefore we can find $n_K\in\N$
depending on $K$, such that for all $n\geq n_K$
\[
\left|\int_{A_1\times A_2} \wt W_1(x,x')\,d(\mu^{1,K}-\mu_n^{1,K})(x)\,d\lambda(x')\right|<\ep.
\]
Combining the above estimates and using the triangle inequality, we get that
for sufficiently large $K$ and $n\geq n_K$,
\begin{equation*}
\begin{split}
\left|\int_{A_1\times A_2} W_1(x,x')\,d(\lambda-\mu^1) (x)\,d\lambda(x')\right|&\leq
\left|\int_{A_1\times A_2} W_1(x,x')-\wt W_1(x,x')\,d(\lambda-\mu^1) (x)\,d\lambda(x')\right|\\
&\quad \phantom{}+ \left|\int_{A_1\times A_2} \wt W_1(x,x')\,d(\lambda-\mu_n^{1,K}) (x)\,d\lambda(x')\right|\\
&\quad \phantom{}+ \left|\int_{A_1\times A_2} \wt W_1(x,x') \,d(\mu_n^{1,K}-\mu^{1,K}) (x)\,d\lambda(x')\right|\\
&\quad \phantom{}+ \left|\int_{A_1\times A_2} \wt W_1(x,x') \,d(\mu^{1,K}-\mu^{1}) (x)\,d\lambda(x')\right|\\
&<4\ep.
\end{split}
\end{equation*}
Since $\ep>0$ was arbitrary this implies that
\begin{equation*}
\left|\int_{A_1\times A_2} W_1(x,x')\,d(\lambda-\mu^1) (x)\,d\lambda(x')\right| = 0.
\end{equation*}
Since $\lambda-\mu^1$ is absolutely continuous with respect to $\lambda$, we
know by the Radon-Nikodym theorem that there is a non-negative function $f$
such that $d(\lambda-\mu^1)(x)=f(x)\,d\lambda(x)$. The Lebesgue
differentiation theorem now says that $W_1(x,x')f(x)=0$ almost everywhere,
which implies (iii).
\end{proof}

\begin{proof}[Proof of Proposition~\ref{prop:dcut-dp}]
Since $\delta_\square(\W_1,\W_2)\leq\delta_1(\W_1,\W_2)$, we only need to
prove that $\delta_\square(\W_1,\W_2)=0$ implies $\delta_1(\W_1,\W_2)=0$ in
order to prove (i). Assume first that the graphons are over $\R_+$, and let
$\mu$ be as in Proposition~\ref{prop27}. Then $W_1^{\pi_1}-W_2^{\pi_2}=0$
$\mu$-almost everywhere. For each $n\in\N$ let $\mu_n$ be some arbitrary
coupling measure on $S_1\times S_2$ such that
$\mu_n|_{[0,n]^2}=\mu|_{[0,n]^2}$. Then
$\lim_{n\rta\infty}\|W_1^{\pi_1}-W_2^{\pi_2}\|_{1,\mu_n}=0$, so
$\delta_1(\W_1,\W_2)=0$. To obtain the result for graphons over general
measure spaces we use Proposition~\ref{prop7}, the triangle inequality, and
the fact that two graphons have distance zero for $\delta_\square$ and
$\delta_1$ if one is a pullback of the other.

For (ii) with graphons over $\R_+$ and $\delta_\square$ defined in terms of
measure-preserving transformations, we will first prove that
$\delta_\square(\W_1,\W_2)=0$ implies $\delta_p(\W_1,\W_2)=0$. This follows
by the exact same argument as in the preceding paragraph, i.e., by using
Proposition~\ref{prop27} to construct a measure $\mu$ and coupling measures
$\mu_n$ on $S_1\times S_2$.

Now we will prove that $\delta_p(\W_1,\W_2)=0$ implies
$\delta_\square(\W_1,\W_2)=0$, still assuming the graphons are over $\R_+$
and that $\delta_p$ is defined in terms of measure-preserving
transformations. By part (i) it is sufficient to show that
$\delta_p(\W_1,\W_2)=0$ implies $\delta_1(\W_1,\W_2)=0$. Fix $\eps>0$ and let
$A_1,A_2\subseteq\R_+$ be such that $\|W_i-W_i\1_{A_i\times A_i}\|_1<\eps/2$
and  $M:= \lambda(A_1)+\lambda(A_2)<\infty$. By H\"older's inequality, for
any isomorphism $\phi\colon\R_+\to\R_+$ and $A:=A_1\cup
\phi^{-1}(A_2)$,
\[
\begin{split}
\|W_1-W_2^\phi \|_1
&\leq
\|(\1-\1_{A\times A}) (W_1-W_2^\phi) \|_1+
\|(W_1-W_2^\phi)\1_{A\times A} \|_1\\
&\leq \eps + \|W_1-W_2^\phi \|_p
\cdot M^{2-2/p}.
\end{split}
\]
Taking the infimum over $\phi$ we see that $\delta_1(\W_1,\W_2) \leq \eps
+M^{2-2/p}\delta_p(\W_1,\W_2)=\eps$.  Since $\eps$ was arbitrary, this shows
that $\delta_1(\W_1,\W_2)=0$.

We get (ii) for non-negative graphons and $\delta_p$ defined in terms of
couplings by using Proposition~\ref{prop10} and
Remark~\ref{rem:alt-def-delta-p}, and to move from graphons over $\R_+$ to
graphons over a general $\sigma$-finite space we use Lemma~\ref{lem:prop7}.
\end{proof}

\section{Measurability properties of graph processes}
\label{sec:measurability}

Recall the definition of a graph process (Definition~\ref{def:graph-process})
and the measurable space of graphs $\BB G$, as well as what it means for two
graph processes to be equal up to relabeling of the vertices
(Definition~\ref{def:relab}). Before stating our main result about the
measurability of the relation of being equal up to relabeling, we state and
prove the following simple lemma.

\begin{lemma}\label{lem:measurability}
Let ${\mcl G}=( G_t)_{t\geq 0}$ be a graph process.
\begin{itemize}
\item[(i)]  Let $V\subset V'\subset\N$ and $E\subset E'\subset
    \binom{\N}{2}$ be finite sets. Then there are increasing sequences
    $(a_k )_{k\in\N}$ and $(b_k)_{k\in\N}$ of real numbers such that
\[
\{t\in\R_+\colon V'\cap V(G_t)=V \text{ and }E'\cap E(G_t)=E\}
=\bigcup_{k\in \N} [a_k,b_k).
\]
Furthermore, for any set of the form $T=\bigcup_{k\in \N} [a_k,b_k)$ with
$a_k,b_k$ as above, the event that $T=\{t\in\R_+\colon V'\cap V(G_t)=V
\text{ and }E'\cap E(G_t)=E\}$ is measurable.

\item[(ii)] For $i=1,2$ let $V_i\subset V_i'\subset\N$ and $E_i\subset
    E_i'\subset \binom{\N}{2}$ be finite sets, let ${\mcl G^i}=(
    G^i_t)_{t\geq 0}$ be a graph process, and let $T_i(\mcl G^i)$ be the
    set of times for which $V_i'\cap V(G^i_t)=V_i$ and  $E_i'\cap
    E(G^i_t)=E_i$. Then the event that $T_1(\mcl G^1)=T_2(\mcl G^2)$ is
    measurable.

\item[(iii)] The event that a specified vertex in $\N$ is isolated for all
    times is measurable.

\item[(iv)] If ${\mcl G}=( G_t)_{t\geq 0}$ is projective, then the birth
    time $t_v\in [0,\infty]$ of any vertex $v\in \N$ is measurable, and
    under the assumptions \eqref{eq19} from
    Section~\ref{sec:W-random-results}, the map defined in \eqref{eq5} is
    measurable, where the $\sigma$-algebra used on the space of measures is
    defined above \eqref{eq5}.
\end{itemize}
\end{lemma}

\begin{proof}
The time set considered in (i) takes the required form since the set of
graphs $\{G\colon V'\cap V(G)=V \text{ and }E'\cap E(G)=E\}$ is an open set
in $\BB G$ and since ${\mcl G}=( G_t)_{t\geq 0}$ is c\`{a}dl\`{a}g. The
measurability claim in (i) follows since the event in question occurs if and
only if the two time sets have the same intersection with $\Q$. The statement
(ii) follows by a similar argument. Both statements (iii) and (iv)
immediately follow from (i).
\end{proof}

\begin{proposition}
The event that two graph processes $(G_t)_{t\geq 0}$ and $(\wh G_t)_{t\geq
0}$ are equal up to relabeling is measurable.
\end{proposition}

\begin{proof}
The proposition is immediate in the case where $\bigcup_{t\geq 0} V( G_t)$ or
$\bigcup_{t\geq 0} V( \wh G_t)$ is finite, since the set of maps
$\phi\colon[n]\to[n]$ is finite, and given any $\phi$ the event that this map
satisfies the requirements of Definition~\ref{def:relab} is measurable by
Lemma~\ref{lem:measurability}.  We may therefore assume that both
$\bigcup_{t\geq 0} V( G_t)$ and $\bigcup_{t\geq 0} V( \wh G_t)$ are infinite.
We may further assume without loss of generality that $\bigcup_{t\geq 0} V(
G_t)=\bigcup_{t\geq 0} V( \wh G_t)=\N$; we may do this upon relabeling the
vertices of both graph processes.

Next, we reduce the proof of the proposition to the case where no vertices
are isolated for all times. For $V\subseteq\N$ let $G_t^V$ denote the induced
subgraph of $G_t$ that has vertex set $V\cap V(G_t)$. Let $V_0\subseteq\N$
(resp.\ $\wh V_0\subseteq\N$) denote the set of vertices for $(G_t)_{t\geq
0}$ (resp.\ $(\wh G_t)_{t\geq 0}$) that are isolated for all times. Then
$(G_t)_{t\geq 0}$ and $(\wh G_t)_{t\geq 0}$ are equal up to relabeling of the
vertices if and only if this property holds for $(G^{V_0}_t)_{t\geq 0}$ and
$(\wh G^{\wh V_0}_t)_{t\geq 0}$ and for $(G^{\N\setminus V_0}_t)_{t\geq 0}$ and
$(\wh G^{\N\setminus \wh V_0}_t)_{t\geq 0}$. To reduce to the case in which no
vertices are isolated for all times, it is sufficient to show measurability
of the event that $(G^{V_0}_t)_{t\geq 0}$ and $(\wh G^{\wh V_0}_t)_{t\geq 0}$ are
equal up to relabeling of the vertices. We say that two vertices $i,j\in V_0$
are equivalent if $\{t\geq 0\,:\,i\in V(G_t^{V_0})\}=\{t\geq 0\,:\,j\in
V(G_t^{V_0})\}$. Equivalence for two vertices $i,j\in \wh V_0$ and for two
vertices $i\in V_0$ and $j\in \wh V_0$ is defined similarly. We observe that
$(G^{V_0}_t)_{t\geq 0}$ and $(\wh G^{\wh V_0}_t)_{t\geq 0}$ are equal up to
relabeling of the vertices if and only if each equivalence class has equal
cardinality in $V_0$ and $\wh V_0$. The latter event is measurable, since for
any two vertices the event that these two vertices are equivalent is
measurable by Lemma~\ref{lem:measurability}. Thus, we can assume that no
vertices are permanently isolated.

To complete the proof, we must determine whether there exists a bijection
$\phi_0\colon\N\to\N$ satisfying the properties of the map $\phi$ in
Definition~\ref{def:relab}, i.e., whether there is a bijective map
$\phi_0\colon\N\to\N$ such that for all times $t\geq 0$, $\phi_0(G_t)=\wh
G_t$.  We will construct such a map by first constructing a sequence of maps
$\phi_n$ defined on a growing sequence of domains, and then using a
subsequence construction to transform them into a map $\phi_0\colon\N\to\N$
with the desired properties. We will show that this construction succeeds if
and only if the two graph processes are equal up to relabeling.

To construct the maps $\phi_n$, we define $A_n^{0}$ to be the set of
injective maps $\phi \colon D \to \N$ such that $D$ is finite,
$\{1,\dots,\lceil n/2 \rceil \} \subseteq D$, and $\{1,\dots,\lfloor n/2
\rfloor \} \subseteq \phi(D)$.  Let $A_n$ to be the set of maps $\phi\in
A_n^{0}$ such that $\phi(G_t^{D})=\wh G_t^{\phi(D)}$ for all $t\geq 0$.  Note
that $A_n$ is non-empty for all $n$ if the two graph processes are equal up
to relabeling (just choose $\phi$ to be a restriction of the bijection
$\phi_0$). Note further that the set $A_n^0$ is countable, and that for each
$\phi\in A_n^{0}$ the event that $\phi\in A_n$ is measurable by
Lemma~\ref{lem:measurability}.

After these preparations, we are ready to construct the map $\phi_0$. First
we define $\phi_0(1)$. For each $j\in\N$ define the event $B_j$ by $B_j=
\bigcap_{n\in\N} B_{j,n}$, where $B_{j,n}$ is the event that there exists a
map $\phi \in A_n$ for which $\phi(1)=j$.  If the graph processes are equal
up to relabeling, then $B_j$ must occur for some $j$. We will prove that
conversely, if $B_j$ occurs for some $j$, then the graph processes are equal
up to relabeling. Since $B_j$ is a countable intersection of measurable
events, this will finish our proof that the event that the two graph
processes are equal up to relabeling is measurable.

To prove that the event $B_j$ implies the existence of a bijection $\phi_0$
such that $\phi_0(G_t)=\wh G_t$ for all $t\geq 0$, we first note that the
occurrence of $B_j$ implies the existence of a sequence  of maps $\phi_n\in
A_n$ such that $\phi_n(1)=j$.  Accordingly, we set $\phi_0(1)=j$. In the
second step of the construction (explained below), we will determine
$\phi_0^{-1}(1)$ by passing to a subsequence for which $\phi_n^{-1}(1)$ is
constant.  More generally, in the $k$th step of the construction, we will
pass to a subsequence to ensure that $\phi_n(i)$ is constant for $1 \le i \le
\lceil k/2 \rceil$ and $\phi_n^{-1}(i)$ is constant for $1 \le i \le \lfloor
k/2 \rfloor$.

We will carry out this construction by induction on $k$.  Suppose that we
have defined $\phi_0(1),\dots,\phi_0(\lceil k/2\rceil)$ and
$\phi_0^{-1}(1),\dots,\phi_0^{-1}(\lfloor k/2\rfloor)$ so that there exists a
sequence $(\phi_n^k)_{n\geq k}$ of maps $\phi_n^k\in A_n$ for which
\begin{align*}
\phi_n^k(i)&=\phi_0(i)
&&\text{for all $n\geq k$ and all $i\leq \lceil k/2\rceil$, and }
\\
(\phi_n^k)^{-1}(i)&=\phi_0^{-1}(i)
&&\text{for all $n\geq k$ and all $i\leq \lfloor k/2\rfloor$. }
\end{align*}
Assume first that $k$ is odd, in which case we need to define $\phi_0^{-1}(
\lfloor (k+1)/2\rfloor)$. Choose $t$ in such a way that $\lfloor
(k+1)/2\rfloor$ is not isolated in $\wh G_t$.  Then $(\phi_n^k)^{-1}( \lfloor
(k+1)/2\rfloor)$ cannot be isolated in $G_t$ either, and since $G_t$ contains
only a finite number of edges, we know there exists a finite set $V$ such
that for all $n\geq k$, $(\phi_n^k)^{-1}( \lfloor (k+1)/2\rfloor)\in V$. But
this implies that we can find a subsequence of $(\phi_n^k)_{{n\geq k}}$ on
which $(\phi_n^k)^{-1}( \lfloor (k+1)/2\rfloor)$ takes a fixed value, which
we use to define $\phi_0^{-1}( \lfloor (k+1)/2\rfloor)$. To conclude we need
to prove the existence of a sequence $\phi_n^{k+1}\in A_n$ satisfying the
induction hypothesis. For $n$ in the subsequence obtained above we define
$\phi_n^{k+1}=\phi_n^{k}$. To turn this subsequence into a sequence
$\phi_n^{k+1}\in A_n$ defined for all $n\geq k+1$, we can simply reuse
elements to fill in any gaps that occur before them, because $A_n \subset
A_m$ for $m < n$.  This completes the proof when $k$ is odd, and the even
case differs only in notation.
\end{proof}

\section{Random Graph Models}
\label{sec3}

The main goal of this appendix is to establish Theorems~\ref{prop11} and
\ref{prop4}. Proposition~\ref{prop2} will be used to prove left convergence
of graphon processes in Section~\ref{sec7}. It will also be applied in the
proof of Theorem~\ref{prop4}(i) in this appendix, where we need to consider
the normalized number of edges in a graphon process.

A result of \citet*[Corollary~2.6]{ls-graphlimits} in the setting of graphons
over probability spaces is closely related to the following proposition.
However, the proofs are different, even if both rely on martingale
techniques. Note that in the course of proving the below proposition we give
an alternative proof of a result of \citet*[Theorem~5.3]{veitchroy} for the
special case of graphs with no self-edges. Recall from Section~\ref{sec7}
that for a simple graph $F$ and a simple graph $G$ we let $\op{inj}(F,G)$
denote the number of injective adjacency preserving maps $\phi\colon V(F)\to
V(G)$.

\begin{proposition}
Let $\W=(W,\scr S)$, where $W\colon S\times S\to[0,1]$ is a symmetric,
measurable (but not necessarily integrable) function, and $\scr
S=(S,\cS,\mu)$ is a $\sigma$-finite measure space. Let $F$ be a simple graph
with vertex set $V(F)=\{1,\dots,k\}$ for $k\geq 2$, such that $F$ has no
isolated vertices. Then a.s.\
\begin{equation}
\lim_{t\rta\infty} t^{-k} \mathrm{inj}(F,G_t(\W))= \int_{S^k}\prod_{(i,j)\in E(F)} W(x_i,x_j)\,dx_1\cdots dx_k,
\label{eq9}
\end{equation}
where both sides should be read as elements of the extended non-negative
reals $[0,\infty]$. \label{prop2}
\end{proposition}

\begin{remark}
Note that the proposition makes no integrability assumptions on $W$. All that
is used is that $W$ is measurable.  As a consequence, the proposition can
deal with situations where, say, the triangle density converges, even though
$G_t(\W)$ has infinitely many edges.  An extreme example of such a behavior
can be obtained by taking $\W=(W,\R_+)$ to be the ``bipartite'' graphon
defined by $W(x,y)=\sum_{i,j\in\N} \1_{2i-2<x<2i-1}\1_{2j-1<y<2j} +
\sum_{i,j\in\N} \1_{2i-1<x<2i}\1_{2j-2<y<2j-1}$, leading to a sequence of
graphs $G_t(\W)$ where every vertex has a.s.\ infinite degree, while all
subgraph frequencies for graphs $F$ that are not bipartite converge to zero.
\end{remark}

\begin{proof}[Proof of Proposition~\ref{prop2}]
We first prove the proposition under the assumption that the right side of
\eqref{eq9} is finite. Throughout the proof we let $G_t:=\wt G_t(\W)$ be the
graphon process generated by $\W$ with isolated vertices. Note that the left
side of \eqref{eq9} is invariant under replacing $G_t(\W)$ with $\wt
G_t(\W)$, because $F$ has no isolated vertices. For each $t>0$ define
$Y_{-t}$ to be the left side of \eqref{eq9} (with $G_t(\W)$ replaced by
$G_t$), i.e.,
\[
Y_{-t}:=t^{-k} \op{inj}(F,G_t)
=t^{-k}\sum_{v_1,\dots,v_k\in V(G_t)}\prod_{(i,j)\in E(F)}\1_{(v_i,v_j)\in E(G_t)}.
\]
As a first step, we will prove that for each $t>0$,
\begin{equation}
\E[Y_{-t}]= \int_{S^k}\prod_{(i,j)\in E(F)} W(x_i,x_j)\,dx_1\cdots\,dx_k.
\label{eq17}
\end{equation}
We may assume that $\mu(S)<\infty$, since we can write $S$ as a union of
increasing sets $S_m$ of finite measure for each $m\in\N$, and by the monotone
convergence theorem it is sufficient to establish \eqref{eq17} with $W$
replaced by $W\1_{S_m\times S_m}$, and with $Y_{-t}$ defined in terms of
graphs where we only consider vertices $v=(t,x)$ for which $x\in S_m$. If
$N:=|V(G_t)|<k$, then $Y_{-t}=0$. If $N\geq k$, then
\[
\E[Y_{-t}\,|\,N]= \frac{1}{t^k} \frac{N!}{(N-k)!} \frac{1}{\mu(S)^k}
\int_{S^k}\prod_{(i,j)\in E(F)} W(x_i,x_j)\,dx_1\cdots\,dx_k,
\]
since we can form $\frac{N!}{(N-k)!}$ ordered sets of size $k$ from $V(G_t)$,
and the probability that a uniformly chosen injective map from $V(F)$ to
$V(G_t)$ is a homomorphism, is given by
\[
\frac{1}{\mu(S)^k}
\int_{S^k}\prod_{(i,j)\in E(F)} W(x_i,x_j)\,dx_1\cdots\,dx_k.
\]
Since $N$ has the law of a Poisson random variable with parameter $t\mu(S)$
we can conclude that \eqref{eq17} holds:
\[
\begin{split}
\E[Y_{-t}] &= \sum_{n=0}^\infty \Pr(N=n) \E[Y_{-t}\,|\,N=n]\\
&= \sum_{n=k}^{\infty} \frac{(t\mu(S))^n}{n!}e^{-t\mu(S)} \frac{1}{t^k}
\frac{n!}{(n-k)!} \frac{1}{\mu(S)^k} \int_{S^k}\prod_{(i,j)\in E(F)} W(x_i,x_j)\,dx_1\cdots\,dx_k\\
&= \int_{S^k}\prod_{(i,j)\in E(F)} W(x_i,x_j)\,dx_1\cdots\,dx_k,
\end{split}
\]
implying in particular that $Y_t$ is integrable for all $t<0$ given that for
now we assumed that the right side of \eqref{eq9} is finite. Note that in
particular this implies that $Y_t$ is a.s.\ finite, even though $W$ may be
such that the event that $G_t$ has infinitely many edges has non-zero
probability.

Let $\wh G_t$ be identical to $G_t$, except that a vertex $v=(t,x)\in V(G_t)$
is labeled only with $x$. In order words, conditioning on a realization of
$\wh G_t$ is equivalent to conditioning on a realization of $G_t$, except
that the time the different vertices were born (i.e., the time they appeared
in the graphon process $(\wt G_t)_{t\geq 0}$) is unknown. Note that since $S$
may have point masses multiple vertices of $\wh G_t$ may have the same label,
but they are still considered to be different. For $t\leq -1$, define
$\cS_{t}$ to be the $\sigma$-algebra generated by $(\wh G_{s})_{s\geq -t}$.
Then $\cS_s\subseteq \cS_t$ for $s\leq t\leq -1$, so $(\cS_t)_{t\leq -1}$ is
a filtration, and $Y_{t}$ is measurable with respect to $\wh G_{-t}$ and
hence $\cS_{t}$; in other words, $(Y_{t})_{t\leq -1}$ is adapted to the
filtration.

Let $t>s>0$. Given any $k$ distinct vertices in $\wh G_t$, the
probability (conditioned on $(\wh G_{t'})_{t'\geq t}$) that all $k$ vertices
are also in $\wh G_s$ is given by $(s/t)^k$. Since $\wh G_s$ is an induced
subgraph of $\wh G_t$,  it follows that
\[
\begin{split}
\E[Y_{-s}\,|\, \cS_{-t}]&=\E[Y_{-s}\,|\,( \wh G_{t'})_{t'\geq t}]\\
&= \frac{1}{s^k} \sum_{v_1,\dots,v_k\in V(\wh G_t)} \prod_{(i,j)\in E(F)}\1_{(v_i,v_j)\in E(G_t)} \cdot \phantom{}\\
& \qquad\qquad\qquad\qquad\qquad\qquad 
\P\big(v_1,\dots,v_k\in V(\wh G_s)\,|\,v_1,\dots,v_k\in V(\wh G_t)\big)\\
&= Y_{-t},
\end{split}
\]
proving that $(Y_{{t}})_{t<0}$ is a backwards martingale. The limit $Y_{-\infty}=\lim_{t\rta\infty,t\in\Q}
Y_{-t}$ exists almost surely \cite[Theorem~7.18]{kallenberg-prob}.
Since $\E[Y_{-t}]<\infty$ for all $t>0$ we know
that a.s., $Y_{-t}<\infty$ for all $t>0$, which implies that a.s.,
$|E(G_t)|<\infty$ for all $t>0$. Therefore $(Y_{t})_{t<0}$ has finitely many
discontinuities in any bounded interval, and is left-continuous with limits
from the right. It follows that $Y_{-\infty}=\lim_{t\rta\infty} Y_{-t}$;
i.e., we do not need to take the limit along rationals.

To complete the proof it is sufficient to prove that the limit $Y_{-\infty}$
is equal to the right side of \eqref{eq17} almost surely. To establish this
it is sufficient to prove that $Y_{-\infty}$ is equal to a deterministic
constant almost surely, since $(Y_{t})_{t<0}$ is uniformly integrable
\cite[Theorem~7.21]{kallenberg-prob}, which implies that $(Y_{t})_{t<0}$
converges to $Y_{-\infty}$ also in $L^1$.

We will use the Kolmogorov 0-1 law \cite[Theorem~1.1.2]{stroock-analyticview}
to deduce this. For any $n\in\N$ define $\mcl V_n:=\{(s,x)\in\mcl
V\,:\,n-1\leq s<n\}$. Let $\cF_n$ be the $\sigma$-algebra generated by the
set $\mcl V_n$ and the edges between the vertex set $\mcl V_n$ and the vertex
set $\bigcup_{1\leq m\leq n} \mcl V_m$. Since the randomness of the edges can
be represented as an infinite sequence of independent uniform random
variables, the $\sigma$-algebras $\cF_n$ can be considered independent even
if the edges considered in $\cF_n$ join vertices in $\mcl V_n$ and $\mcl
V_{m}$ for $m<n$. In order to apply the 0-1 law it is sufficient to prove
that $Y_{-\infty}$ is measurable with respect to the $\sigma$-algebra
generated by $\bigcup_{n\geq n_0} \cF_n$ for all $n_0\in \N$.

Define $Y_{-t,\geq n_0}$ in the same way as $Y_{-t}$, except that instead of
summing over vertices in $V(G_t)$, we sum over vertices in $V(G_t)\cap\mcl
V_{\geq n_0}$, where $\mcl V_{\geq n_0}=\bigcup_{n\geq n_0}\mcl V_n$. Since
$Y_{-t,\geq n_0}$ is  measurable with respect to the $\sigma$-algebra
generated by $\bigcup_{n\geq n_0} \cF_n$,
all we need to show is that for
all $n\geq n_0$, a.s., $Y_{-t}-Y_{-t,\geq n_0}\to 0$ as $t\to\infty$. The
difference between $Y_{-t}$ and $Y_{-t,\geq n_0}$ can then be bounded by
\[
\begin{split}
t^{-k}\sum_{{\substack{v_1,\dots,v_k\in V(G_t)\\ t_i\leq n_0 \text{ for some } i\in [k]}}}
\prod_{(i,j)\in E(F)}\1_{(v_i,v_j)\in E(G_t)},
\end{split}
\]
where $t_i$ is the time label of $v_i$. Conditioned on $v_1,\dots,v_k\in
V(G_t)$, the probability that at least one of them has time label $t_i\leq
t_0$ is bounded by $kt_0/t$. Continuing as in the proof of \eqref{eq17}, we
therefore obtain that
\begin{equation}
0\leq \E[Y_{-t}-Y_{-t,\geq n_0}]\leq k\Bigl(\frac {t_0}t\Bigr)\int_{S^k}\prod_{(i,j)\in E(F)} W(x_i,x_j)\,dx_1\cdots\,dx_k.
\label{eq17a}
\end{equation}
Since the limit $Y_{-\infty}=\lim_{t\to\infty} Y_{-t}$ exists, we can
calculate $Y_{-\infty}$ along any sequence, say the sequence
$(Y_{-n^2})_{n\in \N}$.  The bound \eqref{eq17a} combined with Markov's
inequality and the Borel-Cantelli lemma therefore implies that
$Y_{-\infty}=\lim_{n\to\infty}Y_{-n^2,\geq n_0}$, proving that $Y_{-\infty}$
is measurable with respect to the $\sigma$-algebra generated by
$\bigcup_{n\geq n_0} \cF_n$, as required for the application of the 0-1 law.

This completes the proof of the proposition under the assumption that the
right side of \eqref{eq9} is finite. If the right side is infinite, we note
that for any set $A$ of finite measure \eqref{eq9} holds with $W\1_{A\times
A}$ instead of $W$ on both the left side and the right side. We can make the
right side arbitrarily large by increasing $A$. The left side is monotone in
$A$, and therefore the limit inferior of the left side (with $W$, not
$W\1_{A\times A}$) is larger than any fixed constant, and hence is equal
$\infty$.
\end{proof}

\begin{proof}[Proof of Theorem~\ref{prop4} (i)]
Since the result of the theorem is immediate for $\|W\|_1=0$, we will assume
throughout the proof that $\|W\|_1>0$. Since $\delta_\square(\wt G_t,G_t)=0$
by \eqref{iso-ver-inv}, it is enough to prove the statement for either $(\wt
G_t)_{t\geq 0}$ or $(G_t)_{t\geq 0}$.

If $S$ has finite total mass, then, a.s., $|V(\wt G_t)|$ is finite for each
fixed $t\geq 0$, and conditioned on the size of $|V(\wt G_t)|$, the graph
$\wt G_t$ is  a $\W$-random graph in the sense of the theory of dense graph
convergence.  The results of \citet*{denseconv1} imply that
$\delta_\square(\wt G_t,\wt{\W})\to 0$ and $\|W^{\wt G_t}\|_1\to\|\wt{W}\|_1$
where $\wt\W=(\wt W,\wt{\scr S})$ is obtained from $\W$ by normalizing the
measure to a probability measure (giving, in particular, $\|\wt W\|_1=\|
W\|_1/\pi(S)^2$.) Combined with Lemma~\ref{prop16}, this implies that
$\delta_\square^s(\wt G_t,\W)\to 0$ when $\pi(S)<\infty$.

If $\pi(S)=\infty$, we use Lemma~\ref{lem:prop7}, and the observation that
two graphons generate graphon processes with the same law if one graphon is a
pullback of the other, to reduce the proof to the case $\scr S=(\R_+,\mcl
B,\lambda)$. Given  $0<\ep<1/2$ choose $M>0$ such that
$\|W-W\1_{[0,M]^2}\|_1<\ep\|W\|_1$, and define ${\W}_M$ to be the graphon
$W_M=W\1_{[0,M]^2}$ over $[0,M]$, and  $\wt G_t^M$ to be the induced subgraph
of $\wt G_t$ on the set of vertices $(s,x)$ such that $x\leq M$. Define $\wt
W^{\wt G_t^M,s}:= W^{\wt G_t^M}(\lambda_M\cdot\phantom{},\lambda_M\cdot\phantom{})$ with
$\lambda_M:= M^{-1}\|W\|_1^{1/2}$. In the cut metric $\delta_\square$, the
stretched graphon $\wt W^{\wt G_t^M,s}$ then converges to $\wt
W_M^s:=W_M(\|W\|_1^{1/2}\cdot\phantom{},\|W\|_1^{1/2}\cdot\phantom{})$, again by the convergence
of $\W$-random graphs for $\W$ defined on a probability space.

Furthermore, by Proposition~\ref{prop2} applied to the graph $F$ consisting
of a single edge, we have that a.s., the number of edges in $\wt G_t^M$
divided by $t^2$ converges to $\frac 12\|W \1_{[0,M]^2}\|_1$, so in
particular the time $t_M$ where $\wt G_t^M$ has at least one edge is a.s.\
finite. For the rest of this proof, we will always assume that $t\geq t_M$.

Defining $G_t'$ to be the graph obtained from $\wt G_t$ by removing all
isolated vertices $(s,x)$ from $V(\wt G_t)$ for which $x>M$, we note that by
\eqref{iso-ver-inv}, it is sufficient to prove that
$\delta_\square^s(G_t',\W)\to 0$. Recall that each vertex $v=(s,x)$ of $G_t'$
corresponds to an interval when we define the stretched canonical graphon
$\W^{G_t',s}$ of $G_t'$. Assume the intervals are ordered according the value
of $x$; i.e., if the vertices $v=(s,x)$ and $v'=(s',x')$ satisfy $x<x'$, then
the interval corresponding to $v$ is to the left on the real line of the
interval corresponding to $v'$. Noting that by our assumption $t\geq t_M$,
there exists at last one vertex $v=(s,x)$ in $G_t'$ such that $x\leq M$, we
define the graphon $\wt{\W}^{G_t',s}=(\wt W^{G_t',s},\R_+)$ to be a
``stretched'' version of $\W^{G_t'}$ such that the vertices $v=(s,x)$ for
which $x\in[0,M]$ correspond to the interval $[0,\lambda_M^{-1}]$. In other
words, $\wt W^{G_t',s}=W^{G_t',s}(r_t\cdot\phantom{},r_t\cdot\phantom{})$ for some appropriately
chosen constant $r_t>0$. To calculate $r_t$, we note that
$W^{{G_t'},s}=W^{{G_t'}}(\lambda\cdot\phantom{},\lambda\cdot\phantom{})$ with
$\lambda=\|W^{{G_t'}}\|_1^{1/2}=|V({G_t'})|^{-1}\sqrt{2|E({G_t'})|}$ and $
{\wt W}^{G_t',s}=W^{G_t'}(\lambda'\cdot\phantom{},\lambda'\cdot\phantom{})$ with
$\lambda'=\lambda_M|V(\wt G_t^M)||V({G_t'})|^{-1}$, giving
$r_t=\lambda'/\lambda=\sqrt{\lambda_M^2|V(\wt G_t^M)|^2/(2|E(G_t')|)}$. Since
$|V(\wt G_t^M)|$ is an exponential random variable with expectation $Mt$, and
$|E(G_t')|/t^2=|E(G_t)|/t^2\to \frac 12\|W\|_1$ a.s.\ by Proposition
\ref{prop2}, we have that, a.s., $r_t\to 1$ as $t\to\infty$. By the triangle
inequality, Lemma~\ref{prop15}, and the fact that
$\wt{W}^{G_t',s}|_{[0,\lambda_M^{-1}]^2}=\wt{W}^{{\wt G}_t^M,s}$,
\begin{equation}
\begin{split}
\delta^s_\square(\W,G_t')
\leq &\,\|W^{s}-\wt{W}_M^s\|_1+
\delta_\square\big(\wt{\W}_M^s,\wt{W}^{{\wt G}_t^M,s}\big)
\\
&+ \|\wt{W}^{G_t',s}|_{[0,\lambda_M^{-1}]^2}-\wt{W}^{G_t',s}\|_1 + \delta_\square(\wt{\W}^{G_t',s},\W^{G_t',s}).
\end{split}
\label{eq24}
\end{equation}
The first term on the right side of \eqref{eq24} is bounded by $\eps$ by
assumption, and the second converges to zero as already discussed above. The
third term on the right side of \eqref{eq24} is the product of $r_t^{-2}$ and
the fraction of edges of $G_t$ for which at least one vertex $v=(t,x)$
satisfies $x>M$. By Proposition~\ref{prop2} (applied with the random graphs
$\wt G_t^M$ and $\wt G_t$ and the same simple graph $F$ as above) and
$\lim_{t\rta\infty}r_t=1$ it follows that this term is less than $2\ep$ for
all sufficiently large $t>0$. The fourth term on the right side of
\eqref{eq24} converges to zero by $\lim_{t\rta\infty}r_t=1$ and Lemma
\ref{prop16}. Since $\ep>0$ was arbitrary we can conclude that
$\lim_{t\rta\infty}\delta_\square^s(\W,G_t)=0$.
\end{proof}

\begin{proof}[Proof of Theorem~\ref{prop4} (ii)]
First we will show that the condition $\sum_{n=1}^\infty
\mu(S_n)^{-1}=\infty$ is necessary. We will use proof by contradiction,
and assume $\sum_{n=1}^\infty \mu(S_n)^{-1}<\infty$ and
a.s.-$\lim_{n\rta\infty} \delta_\square^s(\W,G_n)=0$. We will obtain the
contradiction by proving that with positive probability $E(G_n)=\emptyset$
for all $n\in\N$ (which clearly contradicts a.s.-$\lim_{n\rta\infty}
\delta_\square^s(\W,G_n)=0$). By rescaling the measure of $\scr S$ we may
assume without loss of generality that $\|W\|_1=1$. Furthermore, we assume
that $\mu(S)=\infty$ by extending $\scr S$ and $W$ to a space of infinite
measure. Note that the condition $\bigcup_i S_i=S$ will not hold after such an
extension has been done, but we will not use this property in the proof. (The
property $\bigcup_i S_i=S$ is applied only in the second part of the proof,
where we show that the condition $\sum_{n=1}^\infty \mu(S_n)^{-1}=\infty$ is
sufficient.)

First we will prove that there is a random $N\in\N$ such that
$\W^{G_n,s}=\W^{G_N,s}$ (up to interval permutations) for all $n\geq N$.
Since $(|E(G_n)|)_{n\in\N}$ is increasing, in order to do this it is
sufficient to prove that $(|E(G_n)|)_{n\in\N}$ is bounded almost surely, and
by monotone convergence, this in turn follows once we show that
$\sup_{n\in\N} \E[|E(G_n)|]<\infty$. Letting $v_i\in V(G_i)$ denote the
vertex added in step $i\in\N$, and defining $S_0=\emptyset$, we obtain the
desired result:
\[
\begin{split}
\E[|E(G_n)|]
&= \sum_{1\leq i<j\leq n} \Pr[(v_i,v_j)\in E(G_n)]\\
&= \sum_{1\leq i<j\leq n} \frac{1}{\mu(S_i) \mu(S_j)}
\| W\1_{S_i\times S_j}\|_1\\
&\leq \sum_{i',j'=1}^n
\| W\1_{(S_{i'}\setminus S_{i'-1})\times (S_{j'}\setminus S_{j'-1})}\|_1
\sum_{i\geq i',j\geq j'} \frac{1}{\mu(S_i) \mu(S_j)}\\
&\leq \|W\|_1 \left(\sum_{n=1}^\infty \mu(S_n)^{-1} \right)^2\\
&<\infty.
\end{split}
\]
Since $\lim_{n\rta\infty}\delta^s_\square(G_n,\W)=0$ it follows that
$\delta_\square(\W^{G_N,s},\W)=0$.

We saw in the above paragraph that $\delta_\square(\W^{G_N,s},\W)=0$ a.s.\
for some random $N\in\N$. Therefore there is a deterministic step graphon
$\wh{\W}=(\wh W,\R_+)$ with values in $\{0,1\}$ such that
$\delta_\square(\wh{\W},\W)=0$. Since the set $\{D_{\wh W}>0\}$ has finite
measure, by Lemma~\ref{lem:D_W}, the set $A=\{D_W>0\}$ has finite measure as
well.  After changing $W$ on a set of measure $0$, we have $\supp W\subseteq
A\times A$. Note also that by Proposition~\ref{prop:dcut-d1} we have
$\delta_1(\wh W,W)=0$.

For any $n\in\N$ the probability that a feature sampled from the measure
$\mu_n$ is contained in $A$ is given by $\mu(A\cap S_n)/\mu(S_n)\leq
\mu(A)/\mu(S_n)$. Hence the Borel-Cantelli lemma implies that finitely many
vertices in $\bigcup_{n\geq 1}V(G_n)$ have a feature in $A$. Therefore we can
find a deterministic $n_0\in\N$ such that $\Pr(x_n\not\in A \text{ for all }
n\geq n_0)>0$. It follows that with uniformly positive probability
conditioned on $G_{n_0}$, no edges are added to $G_n$ after time $n_0$.

To conclude our proof (i.e., obtain a contradiction by proving that
$E(G_n)=\emptyset$ with positive probability) it is therefore sufficient to
prove that $E(G_{n_0})=\emptyset$ with positive probability. We will do this
by sampling a sequence of graphs $(\wh G_n)_{n\in\N}$ from $\wh{\W}$ which is
close in law to $(G_n)_{n\in\N}$, and use that $E(\wh G_n)=\emptyset$ with
positive probability since $\wh W$ is zero on a certain subdomain since the
graphs we consider have no loops. (Note that our approach would not have
worked if we allowed for loops; if, for example, $\wh{W}|_{[0,1]^2}\equiv 1$
and $S_1,S_2\subset [0,1]$ we would have had $\Pr(E(\wh G_n)=\emptyset)=0$.)
Let $\ep>0$, and recalling that $\delta_1(\wh{\W},\W)=0$, choose a coupling
measure $\wt\mu$ on $S\times\R_+$ such that $\|W^{\pi_1}-\wh
W^{\pi_2}\|_{1,\wt\mu}<\ep$. By using $\wt\mu$ we can sample two coupled
sequences of graphs $(G_n)_{1\leq n\leq n_0}$ and $(\wh G_n)_{1\leq n\leq
n_0}$, such that the two sequences have a law which is close in total
variation distance, $(G_n)_{1\leq n\leq n_0}$ has the law of the graphs in
the statement of the theorem, and $(\wh G_n)_{1\leq n\leq n_0}$ is sampled
similarly as $(G_n)_{n\in\N}$ but with $\wh{\W}$ instead of $\W$. More
precisely, for each $n\in\{1,\dots,n_0\}$ we sample $(x,\wh x)\in
S_n\times\R_+$ from the probability measure
$\mu(S_n)^{-1}\wt\mu|_{S_n\times\R_+}$, we let $x$ (resp.\ $\wh x$) be the
feature of the $n$th vertex of $G_n$ (resp.\ $\wh G_n$), and by using that
$\|W^{\pi_1}-\wh W^{\pi_2}\|_{1,\wt\mu}<\ep$ we can couple $(G_n)_{1\leq
n\leq n_0}$ and $(\wh G_n)_{1\leq n\leq n_0}$ such that for each
$n_1,n_2\in\{1,\dots,n_0\}$ for which $n_1\neq n_2$ we have
\[
\begin{split}
\P\Big(\big\{(x_{n_1},x_{n_2})&\in E(G_{n_0}),
(\wh x_{n_1},\wh x_{n_2})\not\in E(\wh G_{n_0}) \big\} \cup \phantom{}\\
&\qquad\qquad\qquad\big\{(x_{n_1},x_{n_2})\not\in E(G_{n_0}),
(\wh x_{n_1},\wh x_{n_2})\in E(\wh G_{n_0}) \big\} \Big)\\
&\leq  \mu(S_{n_1})^{-1} \mu(S_{n_2})^{-1}
\int_{(S_{n_1}\times\R_+)\times (S_{n_2}\times\R_+)}\left| W^{\pi_1} - \wh W^{\pi_2}\right|\,d\wt\mu\,d\wt\mu\\
&<\mu(S_{n_1})^{-1} \mu(S_{n_2})^{-1}\ep.
\end{split}
\]

Hence the total variation distance between the laws of $(G_n)_{1\leq n\leq
n_0}$ and $(\wh G_n)_{1\leq n\leq n_0}$ is bounded by
$n^2_0\mu(S_1)^{-2}\ep$. Since we can make this distance arbitrarily small by
decreasing $\ep$, in order to complete our proof it is sufficient to prove
that $E(\wh G_{n_0})=\emptyset$ with a uniformly positive probability for all
coupling measures $\wt\mu$. Write $\R_+=\bigcup_{n=0}^N A_n$, such that
$A_0,\dots,A_N$ correspond to the steps of the step function $\wh W$, with,
say, $A_0$ corresponding to the set of all $x$ such that $\int \wh
W(x,y)\,dy=0$. For any choice of $\wt\mu$ we can find a
$k=k_{\wt\mu}\in\{0,\dots,N\}$ such that $\wt\mu(S_1\times A_k)\geq
\mu(S_1)/(N+1)$. Therefore there is a uniformly positive probability that all
the vertices of $\wh G_{n_0}$ have a feature in $A_k$. On this event we have
$E(\wh G_{n_0})=\emptyset$, since $\wh W|_{A_k\times A_k}\equiv 0$ as the
graphs we consider are simple (i.e., they do not have loops). This completes
our proof that the condition $\sum_{n=1}^\infty \mu(S_n)^{-1}=\infty$ is
necessary.

Now we will prove that the condition $\sum_{n=1}^\infty \mu(S_n)^{-1}=\infty$
is sufficient to guarantee that a.s.-$\lim_{n\rta\infty}
\delta_\square^s(\W,G_n)=0$. We will couple $(G_n)_{n\in\N}$ to a graphon
process $(\wt G_t)_{t\geq 0}$ with isolated vertices. Fix $\ep>0$, and choose
$N\in\N$ sufficiently large such that $\|W-W\1_{S_N\times
S_N}\|_\square<\ep$. Sample $(\wt G_t)_{t\geq 0}$, and independently from
$(\wt G_t)_{t\geq 0}$, sample $(G_n)_{1\leq n\leq N}$ as described in the
statement of the theorem. Define $(t_n)_{n\geq N}$ inductively as follows
\[
t_N=0,\qquad t_n = \inf\{t>t_{n-1}\,:\,\text{there exists } x\in S_{n} \text{ such that } (t,x)\in V(\wt G_t) \}.
\]
Note that $t_n \to \infty$ a.s.\ as $n \to \infty$, because the increments
$t_n-t_{n-1}$ are independent and exponentially distributed with mean $\mu(S_n)^{-1}$,
and $\sum_{n=N+1}^\infty \mu(S_n)^{-1}=\infty$ by assumption.
For $n> N$ let $\wh G_n$ be the induced subgraph of $\wt G_{t_n}$ whose
vertex set is
\[
V(\wh G_n)=\{(t,x)\in V(\wt G_{t_n})\,:\, \text{there exists } \wt
n\in\{N+1,\dots,n\}\text{ such that }t=t_{\wt n} \}.
\]
For each $n>N$ let $G_n$
be the union of $G_N$ and $\wh G_n$, such that the edge set of $G_n$ is
given by $E(G_N)\cup E(\wh G_{n})$ in addition to independently sampled
edges between the vertices of $G_N$ and the vertices of $\wh G_{n}$, such
that the probability of connecting vertices with features $x$ and $x'$ is
$W(x,x')$, and such that $G_{n-1}$ is an induced subgraph of $G_n$. It is
immediate that $(G_n)_{n\in\N}$ has the same law as the sequence of graphs
$(G_n)_{n\in\N}$ described in the statement of the theorem.

We will prove that $|E(G_n)\backslash E(\wt G_{t_n})|=o(|E(\wt G_{t_n})|)$
and $|E(\wt G_{t_n})\backslash E(G_n)|<\ep |E(\wt G_{t_n})|$ for all large
$n\in\N$. This is sufficient to complete the proof of the theorem, since
$\ep>0$ was arbitrary, since $\lim_{n\rta\infty}\delta^s_\square(\W,\wt
G_{t_n})=0$ by part (i) of the theorem, and since $\delta_\square^s(G_n,\wt
G_{t_n})\rta 0$ as $n\rta\infty$ and $\ep\rta 0$ by the following argument.
Define $\wt{\W}^{G_n,s}:=(\wt W^{G_n,s},\R_+)$ and $\wt
W^{G_n,s}:=W^{G_n,s}(r_n^{-1}\cdot\phantom{},r_n^{-1}\cdot\phantom{})$ for
$r_n=|E(G_n)|^{1/2}|E(\wt G_{t_n})|^{-1/2}$; i.e., $\wt W^{G_n,s}$ is a
stretched version of $W^{G_n,s}$ defined such that each vertex of $G_n$
corresponds to an interval of length $(2|E(\wt G_{t_n})|)^{-1/2}$. Then each
vertex corresponds to an interval of length $(2|E(\wt G_{t_n})|)^{-1/2}$ both
for $\wt W^{G_n,s}$ and for $W^{\wt G_{t_n},s}$, so by ordering the vertices
appropriately when defining the graphons we have $\|\wt W^{G_n,s}-W^{\wt
G_{t_n},s}\|_1 \leq |E(G_n)\SymmDiff E(\wt G_{t_n})||E(\wt
G_{t_n})|^{-1}=o_n(1)+\ep$. For sufficiently small $\ep>0$ and large $n\in\N$
we have $|r_n-1|<\big||E(G_n)|^{1/2}-|E(\wt G_{t_n})|^{1/2}\big| |E(\wt
G_{t_n})|^{-1/2}<o_n(1)+\ep$, and hence Lemma~\ref{prop16} implies that
$\delta_\square(\W^{G_n,s},\wt{\W}^{G_n,s})<4\ep$ for all sufficiently small
$\ep>0$ and sufficiently large $n\in\N$. Combining the above estimates we get
that for all sufficiently small $\ep>0$ and sufficiently large $n\in\N$,
\[
\delta_\square^s(G_n,\wt G_{t_n})\leq
\delta_\square(\W^{G_n,s},\wt {\W}^{G_n,s}) + \delta_\square(\wt{\W}^{G_n,s},\W^{\wt G_{t_n},s})
\leq 4\ep + \|\wt W^{G_n,s}-W^{\wt G_{t_n},s}\|_1
\leq 6\ep.
\]

First we prove that conditioned on almost any realization of $G_N$,
$|E(G_n)\backslash E(\wt G_{t_n})|=o(|E(\wt G_{t_n})|)$ as $n\to \infty$.
Note that $E(G_n)\backslash E(\wt G_{t_n})$ consists of the edges in
$E(G_N)$, plus independently sampled edges between $V(G_N)$ and $V(\wh G_n)$.
Since $V(\wh G_n)\subset V(\wt G_{t_n})$, we overcount the latter if we
independently sample one edge for each $v\in V(G_N)$ and $v'\in V(\wt
G_{t_n})$, with the probability of an edge between $v$ and $v'$ given by $W$
evaluated at the features of $v$ and $v'$. Defining $\op{deg}(v;\wt G_{t_n})$
to be the number of edges between $v\in V(G_N)$ and $V(\wt G_{t_n})$ obtained
in this way, we thus have
\[
|E(G_n)\backslash E(\wt G_{t_n})|\leq |E(G_N)|+
 \sum_{v\in V(G_N)} \op{deg}(v;\wt G_{t_n}).
\]
By Proposition~\ref{prop2} applied with $F$ being the simple connected graph
with two vertices, $|E(\wt G_{t_n})|=\Theta(t_n^2)$.  In order to prove that
$|E(G_n)\backslash E(\wt G_{t_n})|=o(|E(\wt G_{t_n})|)$ it is therefore
sufficient to prove that, conditioned on almost any realization of $G_N$,
each vertex $v\in V(G_N)$ satisfies $\op{deg}(v;\wt G_{t_n})\leq Ct_n$ for
all sufficiently large $n$ and some $C>0$ depending on the feature of the
vertex. Condition on a realization of $G_N$ such that $\int_S
W(x,y)\,d\mu(y)<\infty$ for all $x\in S$ such that $x$ is the feature of some
vertex in $G_N$. We will prove that if $x\in S$ is the feature of $v\in
V(G_N)$ then a.s.\
\begin{equation}
\lim_{t\rta\infty} Y_{-t}=\int_S W(x,y)\,d\mu(y),\qquad
\text{where $Y_{-t}:=t^{-1}\op{deg}(v;\wt G_t)$ for all $t>0$},
\label{eq39}
\end{equation}
which is sufficient to imply the existence of an appropriate constant $C$.
The convergence result \eqref{eq39} follows by noting that $(Y_t)_{t<0}$ is a
backwards martingale with expectation $\int_S W(x,y)\,d\mu(y)$, which is
left-continuous with right limits at each $t<0$; see the proof of
Proposition~\ref{prop2} for a very similar argument. Hence the Kolmogorov 0-1
law implies \eqref{eq39}. We can conclude that $|E(G_n)\backslash E(\wt
G_{t_n})|=o(|E(\wt G_{t_n})|)$.

Now we prove $|E(\wt G_{t_n})\backslash E(G_n)|<\ep |E(\wt G_{t_n})|$. Let
$\ol G_{t_n}$ be the induced subgraph of $\wt G_{t_n}$ corresponding to
vertices with feature in $S_N$. Then
\[
|E(\wt G_{t_n})\backslash E(G_n)|
\leq |E(\wt G_{t_n})| - |E(\ol G_{t_n})|.
\]
By applying Proposition~\ref{prop2} to each of the graphs $\ol G_{t_n}$ and
$\wh G_{t_n}$, and with $F$ being the simple connected graph on two vertices,
it follows that
\[
\begin{split}
\limsup_{n\rta\infty} |E(\wt G_{t_n})|^{-1}|E(\wt G_{t_n})\backslash E(G_n)|
&\leq \lim_{n\rta\infty}  |E(\wt G_{t_n})|^{-1}(|E(\wt G_{t_n})| - |E(\ol G_{t_n})|)\\
&= \|W\|_1-\|W\1_{S_N\times S_N}\|_1=\|W-W\1_{S_N\times S_N}\|_1<\ep.
\end{split}
\]
\end{proof}

\begin{proof}[Proof of Theorem~\ref{prop11}]
Assume that (i) holds, i.e., $\delta_\square(\W_1,\W_2)=0$. We will prove
that (ii) and (iii) also hold. It is sufficient to prove that (ii) holds,
since (ii) implies (iii).

We first consider the case when $\mu_i(I_i)<\infty$ for $i=1,2$, where
$I_i:=\{x\in S_i\,:\,D_{W_i}(x)>0\}$. Recall that by Proposition
\ref{pro:deg-conv} we have $\mu_1(I_1)=\mu_2(I_2)$, so by restricting the
graphon $\W_i$ to the space $I_i$ for $i=1,2$ we obtain two graphons with cut
distance zero over spaces of finite and equal measure. By definition of
$I_i$, almost surely no vertices of $(\wt G_t)_{t\geq 0}$ will be isolated
for all times, and it is proved that (ii)
holds in, for example, a paper by \citet*[Theorem~8.10]{janson-survey},
who refers to papers by
\citet*{denseconv1}, \citet*{BCL10}, and \citet*{diaconisjanson08} for the original proofs.

Next we consider the case where $\mu_1(I_1)=\mu_2(I_2)=\infty$. We may assume
$\mu_i(S_i\setminus I_i)=0$, since replacing the graphon $\W_i$ by its
restriction to $S_i\setminus I_i$ amounts to removing vertices which are
isolated for all times. Part (i) of Proposition~\ref{prop27} now implies that
we can find a measure $\mu$ such that $W_1^{\pi_1}=W_2^{\pi_2}$ $\mu$-almost
everywhere. By the assumption $\mu_i(S_i\setminus I_i)=0$, part (iv) of the
proposition implies that $\mu$ is a coupling measure. Sampling a graphon
process from $\W_i$ may be done by associating the vertex set with a Poisson
point process on $(S_1\times S_2)\times\R_+$ with intensity
$\mu\times\lambda$, such that each $((x_1,x_2),t)\in (S_1\times
S_2)\times\R_+$ is associated with a vertex with feature $x_i$ appearing at
time $t$.

Now we will prove that (ii) or (iii) imply (i). We will only show that (ii)
implies (i), since we can prove that (iii) implies (i) by the exact same
argument. We assume (ii) holds, and couple $(\wt G_t^{1})_{t\geq 0}$ and
$(\wt G_t^{2})_{t\geq 0}$ such that $\wt G_t^{1}=\wt G_t^{2}$ for all $t\geq
0$. By Theorem~\ref{prop4}(i) we know that
$\lim_{t\rta\infty}\delta_\square(\W_i,\W^{\wt G_t^i})=0$. Since $\W^{\wt
G_t^1}=\W^{\wt G_t^2}$ for all $t\geq 0$ it follows by the triangle
inequality that $\delta_\square(\W_1,\W_2)=0$, so (i) holds.
\end{proof}

\section{Compactness}
\label{sec:compact}

In this appendix we will establish Theorem~\ref{prop1}.

\begin{lemma}
Let $(\W_n)_{n\in\N}$ and $(\wt\W_n)_{n\in\N}$ be two sequences of graphons,
with $\W_n=(W_n,\scr S_n)$, $\scr S_n=(S_n,\cS_n,\mu_n)$, $\wt{\W}_n=(\wt
W_n,\wt{\scr S}_n)$, and $\wt{\scr S}_n=(\wt S_n,\wt{\cS}_n,\wt\mu_n)$, such
that there are measure-preserving transformations $\phi_n\colon S_n\to\wt
S_n$ for which $\lim_{n\rta\infty}\|W_n-\wt W_n^{\phi_n}\|_1=0$. Furthermore,
assume that either (i) $\phi_n$ is a bimeasurable bijection, or (ii) $S_n=\wt
S_n\times[0,1]$, where $[0,1]$ is equipped with Lebesgue measure, and
$\phi_n\colon S_n\to\wt S_n$ is the projection map. Then $(\W_n)_{n\in\N}$
has uniformly regular tails iff $(\wt\W_n)_{n\in\N}$ has uniformly regular
tails. \label{prop25}
\end{lemma}

\begin{proof}
Let $(\ep_n)_{n\in\N}$ be a sequence of positive real numbers converging to
zero, such that $\|W_n-\wt W_n^{\phi_n}\|_1<\ep_n$ for all $n\in\N$. First
assume $(\wt{\W}_n)_{n\in\N}$ has uniformly regular tails. Given any $\ep>0$
let $M>0$ be such that for all $n\in\N$ we can find $\wt U_n\in\wt{\cS}_n$
satisfying $\wt\mu_n(\wt U_n)<M$ and $\|\wt W_n-\wt W_n\1_{\wt U_n\times \wt
U_n}\|_1<\ep/2$. Define $U_n:=\phi^{-1}_n(\wt U_n)$. Since $\phi_n$ is
measure-preserving, $\mu_n(U_n)=\mu_n(\wt U_n)<M$. By first using $\|W_n-\wt
W_n^{\phi_n}\|_1<\ep_n$ (which implies that $\|(W_n-\wt
W_n^{\phi_n})\1_{U_n\times U_n}\|_1<\ep_n$) and then using that $\phi_n$ is
measure-preserving we get
\[
\begin{split}
\|W_n-&W_n\1_{U_n\times U_n}\|_1\\
&\leq \|W_n-\wt W_n^{\phi_n}\|_1
+\|\wt W_n^{\phi_n}-\wt W_n^{\phi_n}\1_{U_n\times U_n}\|_1
+\|(\wt W_n^{\phi_n}-W_n)\1_{U_n\times U_n}\|_1\\
&\leq \|\wt W_n^{\phi_n}-\wt W_n^{\phi_n}\1_{U_n\times U_n}\|_1+2\ep_n\\
&=\|\wt W_n-\wt W_n\1_{\wt U_n\times \wt U_n}\|_1+2\ep_n\\
&<\ep/2+2\ep_n.
\end{split}
\]
The right side is less than $\ep$ for all sufficiently large $n\in\N$.
Therefore $(\W_n)_{n\in\N}$ has uniformly regular tails.

Next assume $(\W_n)_{n\in\N}$ has uniformly regular tails. We consider the
two cases (i) and (ii) separately. In case (i) it is immediate from the above
result that $(\wt{\W}_n)_{n\in\N}$ has uniformly regular tails, since $\|\wt
W_n-W_n^{\phi^{-1}_n}\|_{1}<\ep_n$. Now consider case (ii). Given any $\ep>0$
let $M>0$ be such that for all $n\in\N$ we can find $U_n\in\cS_n$ satisfying
$\mu_n(U_n)<M/2$ and $\|W_n- W_n\1_{U_n\times U_n}\|_1<\ep/5$. Define $\wt
U_n$ by
\[
\wt U_n := \left\{x\in \wt S_n\,:\,
\int_0^1 \1_{(x,s)\in U_n}\,ds> \frac 12
\right\},
\]
and define $U'_n:=\phi_n^{-1}(\wt U_n)$. Note that $\wt U_n$ is a measurable
set since $(x,s)\mapsto \1_{(x,s)\in U_n}$ is measurable. Then $\wt\mu_n(\wt
U_n)<M$, since
\[
\begin{split}
\mu_n(U_n)&=\int_{\wt S_n} \int_0^1 \1_{(x,s)\in U_n}\,ds\,d\wt\mu(x)
\geq \int_{\wt U_n} \int_0^1 \1_{(x,s)\in U_n}\,ds\,d\wt\mu(x)
\geq \int_{\wt U_n}\frac 12\,d\wt\mu(x)
= \frac 12 \wt\mu_n(\wt U_n).
\end{split}
\]
Next we will argue that
\begin{equation}
\| \wt W_n- \wt W_n\1_{\wt U_n\times \wt U_n}\|_1
\leq 2\|\wt W_n^{\phi_n}- \wt W_n^{\phi_n}\1_{U_n\times U_n}\|_1.
\label{eq34}
\end{equation}
If $(x,x')\in (\wt S_n\times \wt S_n)\backslash (\wt U_n\times \wt U_n)$ it
holds by the definition of $\wt U_n$ that
\begin{equation*}
\int_0^1 \int_0^1 \1_{((x,s),(x',s'))\in (S_n\times S_n)\backslash (U_n\times U_n)}\,ds'\,ds
= 1-\int_0^1 \1_{(x,s)\in U_n}\,ds\int_0^1 \1_{(x',s)\in U_n}\,ds
\geq \frac 12,
\end{equation*}
which implies \eqref{eq34} by
\begin{equation*}
\begin{split}
\|\wt W_n-&\wt W_n\1_{\wt U_n\times \wt U_n}\|_1
= \int_{\wt S_n\times \wt S_n}\,d\mu(x)\,d\mu(x')\left|\wt W_n(x,x')\right|
\1_{(x,x')\in (\wt S_n\times \wt S_n)\backslash (\wt U_n\times \wt U_n)}
\\
&\leq 2\int_{\wt S_n\times \wt S_n}\,d\mu(x)\,d\mu(x')\bigg(\left|\wt W_n(x,x')\right|
 \1_{(x,x')\in (\wt S_n\times \wt S_n)\backslash (\wt U_n\times \wt U_n)}
 \\
 &\qquad\qquad\cdot\int_0^1 \int_0^1
\1_{((x,s),(x',s'))\in (S_n\times S_n)\backslash (U_n\times U_n)}\,ds'\,ds\bigg)
\\
& \leq 2\int_{\wt S_n\times \wt S_n}\,d\mu(x)\,d\mu(x')\left|\wt W_n(x,x')\right|
 \int_0^1 \int_0^1
\1_{((x,s),(x',s'))\in (S_n\times S_n)\backslash (U_n\times U_n)}\,ds'\,ds
\\
&=2\left\|\wt W_n^{\phi_n}
\1_{(S_n\times S_n)\backslash (U_n\times U_n)}\right\|_1\\
&=2\|\wt W_n^{\phi_n}- \wt W_n^{\phi_n}\1_{U_n\times U_n}\|_1.
\end{split}
\end{equation*}
Using that $\phi_n$ is measure-preserving, the triangle inequality, that
$\|\wt W_n^{\phi_n}-W_n\|_1<\ep_n$, and the estimate \eqref{eq34} we get
\[
\begin{split}
\|\wt W_n-\wt W_n\1_{\wt U_n\times \wt U_n}\|_1
&\leq 2\| \wt W_n^{\phi_n}- \wt W_n^{\phi_n}\1_{U_n\times U_n}\|_1\\
&\leq 2\|W_n-W_n\1_{U_n\times U_n}\|_1+4\ep_n
<2\ep/5+4\ep_n.
\end{split}
\]
The right side is less than $\ep$ for all sufficiently large $n\in\N$, and
thus $(\wt{\W}_n)_{n\in\N}$ has uniformly regular tails.
\end{proof}

\begin{proof}[Proof of Theorem~\ref{prop1}]
First we will prove that every $\delta_\square$-Cauchy sequence has uniformly
regular tails. Let $(\mcl W_n)_{n\in\N}$ with $\mcl W_n=(W_n,\scr S_n)$ be a
$\delta_\square$-Cauchy sequence of graphons, i.e.,
$\lim_{n,m\rta\infty}\delta_\square(\W_n,\W_m)\rta 0$. By Lemma~\ref{prop25}
we may assume without loss of generality that $\scr S_n$ is atomless for all
$n\in\N$. By Lemmas~\ref{prop6} and~\ref{prop9} we can find graphons
$\wt{\W}_n=(\wt W_n,\R_+)$ and measure-preserving maps $\psi_n\colon  S_n\to
\R_+$ such that $W_n=(\wt W_n)^{\psi_n}$. Since
$\delta_\square(\wt{\W}_n,\W_n)=0$, $(\wt{\W}_n)_{n\in\N}$ is a Cauchy
sequence. Given any $\ep>0$ choose $N\in\N$ such that
$\delta_\square(\wt{\W}_N,\wt{\W}_n)< \ep/4$ for all $n\geq N$. For each
$n\leq N$ let $M_n\in\R_+$ be such that $\|\wt W_n-\wt
W_n\1_{[0,M_n]^2}\|_1<\ep/3$, and define $M:=\sup_{n\leq N} M_n<\infty$. To
prove that $(\wt{\W}_n)_{n\in\N}$ has uniformly regular tails it is
sufficient to prove that for each $n\geq N$ we can find a Borel-measurable
set $\wt A_n \subset\R_+$ such that
\begin{equation}
\lambda
(\wt A_n)\leq M,\qquad
\|\wt W_n-\wt W_n\1_{\wt A_n\times \wt A_n} \|_1<\ep.
\label{eq16}
\end{equation}
We can clearly find an appropriate set $\wt A_n$ for $n=N$; indeed, we can
find a set $\widetilde A_N\subset\R_+$ such that the second bound holds with
$\eps/3$ instead of $\eps$. By Proposition~\ref{prop10}(c) we can find
isomorphisms $\phi_n\colon \R_+\to \R_+$ such that $\|\wt
W_N-\wt W^{\phi_n}_n\|_{\square}<\ep/3$ for all $n\geq N$. Define $\wt
A_n=\phi_n(\wt A_N)$, and note that
\[
\|\wt W_n-\wt W_n\1_{\wt A_n\times \wt A_n} \|_\square
=\|\wt W_n^{\phi_n}-\wt W_n^{\phi_n}\1_{\wt A_N\times \wt A_N} \|_\square
\leq
\|\wt W_N-\wt W_N\1_{\wt A_N\times \wt A_N} \|_\square+\frac {2\eps}3
< \eps.
\]
Observing that for non-negative graphons the cut norm is equal to the $L^1$
norm, this gives that \eqref{eq16} is satisfied and $(\wt{\W}_n)_{n\in\N}$
has uniformly regular tails. Defining $A_n:=\psi^{-1}(\wt A_n)$, we have
$\mu(A_n)<M$ and $\|W_n-W_n\1_{A_n\times A_n}\|_1=\|\wt W_n-\wt W_n\1_{\wt
A_n\times \wt A_n}\|_1<\ep$. Hence $(\W_n)_{n\in\N}$ has uniformly regular
tails.

Now we will prove that uniform regularity of tails implies subsequential
convergence for $\delta_\square$. We consider some sequence of graphons
$(\W_n)_{n\in\N}$ with uniformly regular tails, and will prove that the
sequence is subsequentially convergent for $\delta_\square$ towards some
graphon $\W$. By Lemma~\ref{prop25} we may assume without loss of generality
that $\scr{S}_n$ is atomless for all $n\in\N$, and by trivially extending
$\W_n$ to a graphon over a space of infinite total mass if needed, we may
assume that $\mu_n(S_n)=\infty$. Recall the definition of a partition of a
measurable space, which was given as part of the discussion before the
statement of Proposition~\ref{prop13}. We will prove that we can find
increasing sequences $(m_k)_{k\in\N}$ and $(M_k)_{k\in\N}$ with values in
$\N$, such that for each $k,n\in\BB N$ there is a partition $\mcl P_{n,k}$ of
$S_n$ and a graphon $\mcl W_{n,k}=(W_{n,k},\R_+)$ such that the following
hold:
\begin{itemize}
\item[(i)] We have $\mcl P_{n,k}=\{I_{n,k}^i\,:\,i=0,\dots,m_k\}$, where
    $\mu_n(S_n\setminus I_{n,k}^0)=M_k$ and $\mu_n(I_{n,k}^i)=M_k/m_k$ for
    $i\in\{1,\dots,m_k\}$.
\item[(ii)] We have $\delta_\square(\mcl W_n,\mcl W_{n,k})<1/k$ for all $n\in\N$.
\item[(iii)] For each $i_1,i_2\in\{1,\dots,m_k\}$ the value of $W_{n,k}$ on
    $([i_1-1,i_1)\times[i_2-1,i_2))M_k/m_k$ is constant and equal to the
    value of $(W_{n})_{\mcl P_{n,k}}$ on $I_{n,k}^{i_1}\times
    I_{n,k}^{i_2}$. On the complement of $[0,M_k]^2$, we have $W_{n,k}=0$.
\item[(iv)] The partition $\mcl P_{n,k+1}$ refines the partition $\mcl
    P_{n,k}$. We number the elements of the partition $\mcl P_{n,k+1}$ to
    be consistent with the refinement. More precisely, defining $r_k:=
    (M_k/m_{k})/(M_{k+1}/m_{k+1})\in\N$ to be the ratio of the partition
    sizes in the two partitions, we have
    $I_{n,k}^i=\bigcup_{j=(i-1)r_k+1}^{ir_k} I_{n,k+1}^j$ for every $i$ with $0<i\leq m_k$.
\end{itemize}
Partitions $\mcl P_{n,k}$ and graphons $\W_{n,k}$ satisfying (i)--(iv) exist
by the following argument. By the assumption of uniformly regular tails, for
each $k\in\N$ we can find an $M_k\in\N$ such that for appropriate sets
$I_{n,k}^0$ satisfying $\mu_n(S_n\setminus I_{n,k}^0)=M_k$ we have
$\|W_n-W_n\1_{(S_n\setminus I_{n,k}^0)\times (S_n\setminus
I_{n,k}^0)}\|_1<1/(3k)$ for all $n\in\N$. By Lemmas~\ref{prop6} and
\ref{prop9}, for each $n,k\in\N$ the graphon $(W_n|_{(S_n\setminus
I_{n,k}^0)\times (S_n\setminus I_{n,k}^0)},S_n\setminus I_{n,k}^0)$ is a
pullback of a graphon $\wt W_{n,k}=(\wt W_{n,k},[0,M_k])$ by a
measure-preserving transformation $\varphi_{n,k}$. By applying Szemer\'edi
regularity for equitable partitions to $\wt W_{n,k}$ (see, for example, the
paper of \citel{lp1}, Lemma~3.3) we can find appropriate $m_k\in\N$ and
partitions $\wt{\mcl P}_{n,k}$ of $[0,M_k]$ such that $\|(\wt W_{n,k}-(\wt
W_{n,k})_{\wt{\mcl P}_{n,k}})\|_\square<1/(3k)$. Then the pullback of $(\wt
W_{n,k})_{\wt{\mcl P}_{n,k}}$ along $\varphi_{n,k}$ equals $( W_{n})_{\mcl
P_{n,k}}$ for an appropriate partition of $S_n$ satisfying (i), and
\[
\begin{split}
\|W_n-(W_n)_{\mcl P_{n,k}}\|_\square
&\le \|W_n-W_n\1_{(S_n\setminus I_{n,k}^0)\times (S_n\setminus I_{n,k}^0)}\|_\square\\
&\quad\phantom{} + \|(W_n-(W_n)_{\mcl P_{n,k}})\1_{(S_n\setminus I_{n,k}^0)\times (S_n\setminus I_{n,k}^0)}\|_\square\\
&\quad\phantom{} + \|(W_n)_{\mcl P_{n,k}}\1_{(S_n\setminus I_{n,k}^0)\times (S_n\setminus I_{n,k}^0)}-(W_n)_{\mcl P_{n,k}}\|_\square\\
&=
\|W_n-W_n\1_{(S_n\setminus I_{n,k}^0)\times (S_n\setminus I_{n,k}^0)}\|_\square
+ \|\wt W_{n,k}^{\varphi_{n,k}}-(\wt W_{n,k})_{\wt{\mcl P}_{n,k}}^{\varphi_{n,k}}
\|_\square\\
&\quad \phantom{}+ \left\|\left(W_n\1_{(S_n\setminus I_{n,k}^0)\times (S_n\setminus I_{n,k}^0)}
-W_n\right)_{\mcl P_{n,k}}\right\|_\square
\\
&< 1/k.
\end{split}
\]
Define $\mcl W_{n,k}$ as described in (iii), and note that all requirements
(i)--(iv) are satisfied since $\delta_\square(((W_n)_{\mcl P_{n,k}},\scr
S_n),\mcl W_{n,k})=0$.

By compactness, for each $k\in\N$ there exists a step function $U_k\colon
\R_+^2\to[0,1]$ (with support in $[0,M_k]^2$) such that $(W_{n,k})_{n\in\N}$
converges pointwise and in $L^1$ along a subsequence towards $U_k$. We may
assume the subsequence along which $(W_{n,k+1})_{n\in\N}$ converges is
contained in the subsequence along which $(W_{n,k})_{n\in\N}$ converges. Note
that for each $i_1,i_2\in\{1,\dots,m_k\}$ the function $U_k$ is constant on
$([i_1-1,i_1]\times[i_2-1,i_2])M_k/m_k$. Furthermore, observe that if
$k,k'\in\N$ and $k'\geq k$, the value of $U_k$ at
$([i_1-1,i_1]\times[i_2-1,i_2])M_k/m_k$ is equal to the average of $U_{k'}$
over this set. Define the graphon $\cU_k$ by $\mcl U_k:=(U_k,\R_+)$.

Choose $M>1$, and then choose $k'$ such that $M_k\geq M_{k'}\geq M$ for all
$k\geq k'$. Let $(X,Y)$ be a uniformly random point in $[0,M_{k'}]^2$. By the
observations in the preceding paragraph $(U_k(X,Y))_{k\geq k'}$ is a
martingale. Hence the martingale convergence theorem implies that the limit
$\lim_{k\rta\infty} U_k(X,Y)$ exists a.s. Since $M$ was arbitrary it follows
that there is a set $E\subset\R_+^2$ of measure zero outside of which
$(U_k)_{k\geq k'}$ converges pointwise. Define the graphon $\mcl U:=(U,\R_+)$
as follows. For any $(x_1,x_2)\in \R_+^2\backslash E$ define
$U(x_1,x_2):=\lim_{k\rta\infty} U_k(x_1,x_2)$, and for any $(x_1,x_2)\in E$
define $U(x_1,x_2):=0$. Since the functions $U_k$ are uniformly bounded,
martingale convergence also implies that $U_k|_{[0,M_{\ell}]^2}$ converges to
$U|_{[0,M_{\ell}]^2}$ in $L^1$ for each ${\ell}\in\N$.

Next we will show that $\lim_{k\rta\infty}\|U_k-U\|_1=0$. Since
$\lim_{k\rta\infty}\|(U_k-U)\1_{[0,M_\ell]^2}\|_1=0$ for each $\ell\in\N$ it
is sufficient to prove that $
\|U_k\1_{\R_+^2\setminus[0,M_\ell]^2}\|_1<1/(3\ell)$ for all $k,\ell\in\N$
for which $k>\ell$. This follows by Fatou's lemma and the
inequality
\[
\begin{split}
\int_{\R_+^2\setminus[0,M_\ell]^2} |W_{n,k}|
&= \int_{[0,M_k]^2\setminus[0,M_\ell]^2} |W_{n,k}|
\leq \int_{(S_n\setminus I_{n,k}^0)^2 \setminus (S_n\setminus I_{n,\ell}^0)^2} |W_{n}|_{\mcl P_{n,k}}\\
&= \int_{(S_n\setminus I_{n,k}^0)^2 \setminus (S_n\setminus I_{n,\ell}^0)^2} |W_{n}|
<1/(3\ell).
\end{split}
\]

By the result of the preceding paragraph
\[
\limsup_{k\rta\infty}\delta_\square(\mcl U_k,\mcl U)\leq
\limsup_{k\rta\infty}\|U_k-U\|_1=0,
\]
and we conclude the proof by applying the triangle inequality
to obtain
\[
\liminf_{n\rta\infty}\delta_\square(\mcl U,\mcl W_n)\leq
\limsup_{k\rta\infty} \liminf_{n\rta\infty} \Big(\delta_\square(\mcl U,\mcl U_k)+\delta_\square(\mcl U_k,\mcl W_{n,k})+\delta_\square(\mcl W_{n,k},\mcl W_n)\Big)=0.
\]
\end{proof}

\section{Basic Properties of Metric Convergent Sequences of Graphs}
\label{sec8}

In this appendix we will establish Propositions~\ref{prop13} and
\ref{prop:unbdd_avg_deg}. First we prove a lemma saying that for a set of
graphs with uniformly regular tails we may assume the sets $U$ in Definition
\ref{defn2} correspond to sets of vertices.

\begin{lemma}
Let $\mcl G$ be a set of graphs with uniformly regular tails. For every
$\ep>0$ there is an $M>0$ such that for each $G\in\mcl G$ we can find a set
$U\subset\R_+$ corresponding to a set of vertices for $G$ such that
$\|W^{G,s}-W^{G,s}\1_{U\times U}\|_1<\ep$ and $\lambda(U)<M$. \label{prop26}
\end{lemma}

\begin{proof}
Since $\mcl G$ has uniformly regular tails we can find an $M>0$ such that
for each $G\in\mcl G$ there is a set $\wt U\subset\R_+$ (not necessarily
corresponding to a set of vertices for $G$) such that
$\|W^{G,s}-W^{G,s}\1_{\wt U\times \wt U}\|_1<\ep/2$ and $\lambda(\wt
U)<M/2$. Recall that each vertex $i\in V(G)$ corresponds to an interval
$I_i\subset\R_+$ for the stretched canonical graphon $W^{G,s}$, such that
$\lambda(I_i)$ is proportional to the weight of the vertex. Given a set $\wt
U\subset\R_+$ as above, define
\[
U:=\bigcup_{i\in\mcl I} I_i,\quad\text{where}\quad
\mcl I:= \{i\in V(G)\,:\, 2\lambda(I_i\cap \wt U)>\lambda(I_i) \}.
\]
The lemma now follows by observing that $\lambda(U)\leq 2\lambda(\wt U)<M$
and
\[
\begin{split}
\|W^{G,s}-W^{G,s}\1_{ U\times U}\|_1
&=\sum_{i,j\in V(G)\,:\,(i,j)\not\in\mcl I\times\mcl I } \beta_{i,j} \lambda(I_i)\lambda(I_j)\\
&\leq 2\sum_{i,j\in V(G)\,:\,(i,j)\not\in\mcl I\times\mcl I } \beta_{ij} \big(\lambda(I_i)\lambda(I_j)-\lambda(I_i\cap \wt U)\lambda(I_j\cap \wt U)\big)\\
&\leq 2\|W^{G,s}-W^{G,s}\1_{ \wt U\times \wt U}\|_1\\
&<\ep.
\end{split}
\]
\end{proof}

\begin{proof}[Proof of Proposition~\ref{prop13}]
Define $M_n:=\inf\{M>0\,:\,\supp(W^{G_n,s})\subseteq[0,M]^2 \}$. If
$(G_n)_{n\in\N}$ is sparse, then $\liminf_{n\rta\infty} M_n=\infty$. By
Lemma~\ref{prop26}, if $(G_n)_{n\in\N}$ has uniformly regular tails there
exists an $M'>0$ such that if we order the vertices of $G_n$ appropriately
when defining the canonical graphon $W^{G_n}$ of $G_n$, then
$\|W^{G_n,s}\1_{[0,M']^2}\|_1>1/2$ for all  $n\in\N$.

The graphons $\W^{G_n,r}$ and $\W^{G_n,s}$ are related by $W^{G_n,r}= \wt
M_n^2 W^{G_n,s}(\wt M_n\cdot\phantom{},\wt M_n\cdot\phantom{})$ for some $\wt M_n\geq M_n$ (with
$\wt M_n=M_n$ if $G_n$ has no isolated vertices; if $G_n$ has isolated
vertices corresponding to the end of the interval $[0,1]$ for the canonical
graphon $W^{G_n}$ we will have $\wt M_n>M_n$). If $\lim_{n\rta\infty} \wt
M_n=\infty$ and $\|W^{G_n,s}\1_{[0,M']^2}\|_1>1/2$ for all $n\in\N$, then
\begin{equation}
\|W^{G_n,r}\1_{[0,a_n]^2}\|_1>1/2
\qquad\text{and}\qquad
\lim_{n\rta\infty} a_n=0,\qquad
\text{where } a_n:= \min\big(M' \wt M_n^{-1},1\big).
\label{eq36}
\end{equation}
The proof of (i) is complete if we can prove that \eqref{eq36} implies that
$(G_n)_{n\in\N}$ is not uniformly upper regular. Assume the opposite, and let
$K\colon (0,\infty)\to (0,\infty)$ and $(\eta_n )_{n\in\N}$ be as in the
definition of uniform upper regularity. Let $\mcl P_n$ be a partition of
$\R_+$ such that one of the parts is $[0,a'_n]$, where $a'_n\geq a_n$ is
chosen as small as possible such that $[0,a'_n]$ corresponds to an integer
number of vertices of $G_n$ for the canonical graphon. Then
$\lim_{n\rta\infty} a'_n=0$ since $\lim_{n\rta\infty} a_n=0$ and
$\lim_{n\rta\infty}V(G_n)=\infty$. By the first part of \eqref{eq36} it
follows that $(W^{G_n,r})_{\mcl P_n}>K(1/2)$ on $[0,a'_n]^2$ for all
sufficiently large $n$; hence for all sufficiently large $n$,
\[
\|(W^{G_n,r})_{\mcl P_n}\1_{|(W^{G_n,r})_{\mcl P_n}|\geq K(1/2)}\|_1
\geq \|(W^{G_n,r})_{\mcl P_n}\1_{[0,a'_n]^2}\|_1
= \|W^{G_n,r}\1_{[0,a'_n]^2}\|_1
>1/2.
\]
We have obtained a contradiction to the assumption of uniform upper
regularity, and thus the proof of (i) is complete.

Defining $\rho_n:=\rho(G_n)$ (recall the definition of $\rho$ in the
beginning of Section~\ref{sec2}), we have
$W^{G_n,s}=W^{G_n}(\rho_n^{1/2}\cdot\phantom{},\rho_n^{1/2}\cdot\phantom{})$ and
$W^{G_n,r}=\rho_n^{-1}W^{G_n}$. If $(G_n)_{n\in \N}$ is dense and has
convergent edge density the following limit exists and is positive:
$\rho:=\lim_{n\rta\infty}\rho_n>0$. It follows by Lemma~\ref{prop16} (resp.\
Lemma~\ref{prop15}) that $(G_n)_{n\in\N}$ is a Cauchy sequence for
$\delta_\square^s$ (resp.\ $\delta_\square^r$) iff $(\W^{G_n})_{n\in\N}$ is a
Cauchy sequence for $\delta_\square$, since for any $n,m\in\N$,
\[
\begin{split}
\Big|\delta_\square^s(G_n,G_m)&-\delta_\square\big((W^{G_n}(\rho^{1/2}\cdot\phantom{},\rho^{1/2}\cdot\phantom{}),\R_+),
(W^{G_m}(\rho^{1/2}\cdot\phantom{},\rho^{1/2}\cdot\phantom{}),\R_+)\big)\Big|
\\
&\le \delta_\square\big(W^{G_n,s},(W^{G_n}(\rho^{1/2}\cdot\phantom{},\rho^{1/2}\cdot\phantom{}),\R_+)\big)
\\
&\quad\phantom{}+ \delta_\square\big((W^{G_m}(\rho^{1/2}\cdot\phantom{},\rho^{1/2}\cdot\phantom{}),\R_+),W^{G_m,s}\big)\\
&
\rta\, 0
\end{split}
\]
as $n,m\rta\infty$, and a similar estimate holds with $\delta_\square^r$
instead of $\delta_\square^s$. This completes the proof of the first
assertion of (ii).

To prove the second assertion of (ii) consider the following sequence of
dense graphs $(G_n)_{n\in\N}$, which is a Cauchy sequence for
$\delta_\square^s$ but not for $\delta_\square^r$. For odd $n$ let $G_n$ be a
complete simple graph on $n$ vertices, and for even $n$ let $G_n$ be the
union of a complete graph on $n/2$ vertices and $n/2$ isolated vertices. This
sequence converges to $\W_1:=(\1_{[0,1]^2},\R_+)$ for $\delta_\square^s$, but
does not converge for $\delta_\square^r$.

Conversely, the following sequence of dense graphs $(G_n)_{n\in\N}$ is a
Cauchy sequence for $\delta_\square^r$ but not for $\delta_\square^s$. For
odd $n$ let $G_n$ be a complete graph on $n$ vertices, and for even $n$ let
$G_n$ be an Erd\H{o}s-R\'{e}nyi graph with edge probability $1/2$. This
sequence converges to $\W_1$ for $\delta_\square^r$, but does not converge
for $\delta_\square^s$.
\end{proof}

\begin{proof}[Proof of Proposition~\ref{prop:unbdd_avg_deg}]
We will assume throughout the proof that the graphs have no isolated
vertices, since the case of $o|E(G_n)|$ vertices clearly follows from this
case. We assume the average degree of $(G_n)_{n\in\N}$ is bounded above by
$d\in\N$, and want to obtain a contradiction. When defining the canonical
stretched graphon $W^{G_n,s}$ of $G_n$, each vertex of $G_n$ corresponds to
an interval of length $1/\sqrt{2|E(G_n)|}$. Since $|E(G_n)|/|V(G_n)|\leq d/2$
by assumption, the vertices of $G_n$ correspond to an interval of length
$|V(G_n)|/\sqrt{2|E(G_n)|} \geq\sqrt{2|E(G_n)|}/d$, which is too stretched
out to be compatible with uniformly regular tails.  Explicitly, given that
$G_n$ has no isolated vertices it follows that for any $M>0$ and any Borel
set $I\subset \R_+$ satisfying $\lambda(I)<M$,
\[
\int_{(\R_+\backslash I)\times \R_+} W^{G_n,s} \geq \frac{\sqrt{2|E(G_n)|}/d-M}{\sqrt{2|E(G_n)|}}.
\]
By the assumption that $\lim_{n\rta\infty}|E(G_n)|=\infty$, the right side of
this equation is greater than $1/(2d)$ for all sufficiently large $n\in\N$.
Since $M>0$ was arbitrary, this is not compatible with $G_n$ having uniformly
regular tails, which together with Theorem~\ref{prop1} gives a contradiction.
\end{proof}

\section{Exchangeability of Graphon Processes}
\label{sec6}

The main goal of this appendix is to prove Theorem~\ref{prop5}.

\begin{lemma}
Let $(\wt G_n)_{n\in\N}$ be a sequence of simple graphs with uniformly regular
tails, such that $|E(\wt G_n)|<\infty$ for all $n\in\N$ and
$\lim_{n\rta\infty}|E(\wt G_n)|=\infty$. Fix $d\in\N$, and for each $n\in\N$
let $\wh G_n$ be an induced subgraph of $\wt G_n$ where all or some of the
vertices of degree at most $d$ are removed. Then $\lim_{n\rta\infty} |E(\wh
G_n)|/|E(\wt G_n)|=1$. \label{prop12}
\end{lemma}

\begin{proof}
We wish to prove that $\ep:=\limsup_{n\rta\infty}\ep_n= 0$, where
$\ep_n:=1-|E(\wh G_n)|/|E(\wt G_n)|$. By taking a subsequence we may assume
$\ep=\lim_{n\rta\infty}\ep_n$. We will carry out the proof by contradiction,
and assume $\ep>0$.  By definition of $\wh G_n$ there are at least $\ep_n
|E(\wt G_n)|$ edges of $\wt G_n$ which have at least one endpoint of degree
at most $d$. Hence there are at least $\ep_n |E(\wt G_n)|/d$ vertices with
degree between $1$ and $d$. In the canonical stretched graphon of $\wt G_n$,
each vertex corresponds to an interval of length $(2|E(\wt G_n)|)^{-1/2}$.
Hence the total length of the intervals corresponding to vertices of degree
between $1$ and $d$ is at least $2^{-1/2} \ep_n |E(\wt G_n)|^{1/2} d^{-1}$,
which tends to infinity as $n\rta\infty$. It follows that for each $M>0$ and
any sets $U_n \subset \R_+$ of measure at most $M$,
\[
\|W^{\wt G_n,s}-W^{\wt G_n,s}\1_{U_n\times U_n}\|_1
\geq \big(2^{-1/2}\ep_n |E(\wt G_n)|^{1/2} d^{-1}-M\big) \cdot (2|E(\wt G_n)|)^{-1/2},
\]
which is at least $\ep (2^{3/2} d)^{-1}$ when $n$ is sufficiently large.
Thus $(\wt G_n)_{n\in \N}$ does not have uniformly regular tails, and we
have obtained the desired contradiction.
\end{proof}

We will now prove Theorem~\ref{prop5}. Note that we use a result of
\citet*[Theorem~9.25]{kallenberg-exch} for part of the argument, a result
which is also used by \citet*{veitchroy}, but that we use it to prove
Theorem~\ref{prop5}, which characterizes exchangeable random graphs that have
uniformly regular tails, while \citet*{veitchroy} use it to characterize
exchangeable random graphs that have finitely many edges for each finite
time, but which do not necessarily have uniformly regular tails (a notion not
considered by \citel{veitchroy}).

\bigskip 

\begin{proof}[Proof of Theorem~\ref{prop5}]
First assume $(\wt G_t)_{t\geq 0}$ is a graphon process generated by
$\W_\alpha$ with isolated vertices, where $\alpha$ is a random variable. We
want to prove that $(\wt G_t)_{t\geq 0}$ has uniformly regular tails, and
that the measure $\xi$ is exchangeable. Regularity of tails is immediate from
Theorems~\ref{prop4}(i) and~\ref{prop1}. Exchangeability is immediate by
observing that the Poisson random measure $\mcl V$ on $\R_+\times S$ defined
in the beginning of Section~\ref{sec:W-random-results} is identical in law to
$\{(\phi(t),x)\,:\, (t,x)\in\mcl V \}$ for any measure-preserving
transformation $\phi\colon\R_+\to \R_+$ (in particular, for the case when
$\phi$ corresponds to a permutation of intervals).

To prove the second part of the theorem assume that $\xi$ is jointly
exchangeable and that $(\wt G_t)_{t\geq 0}$ has uniformly regular tails. By
joint exchangeability of $\xi$ it follows from the representation theorem for
jointly exchangeable random measures on $\R_+^2$ \cite[Theorem~9.24]{kallenberg-exch}
that a.s.
\begin{equation}
\begin{split}
\xi =&\, \sum_{i,j} f(\alpha,x_i,x_j,\zeta_{\{i,j\}}) \delta_{t_i,t_j}
+ \beta\lambda_D+\gamma\lambda^2\\
&+\sum_{j,k} \left(
g(\alpha,x_j,\chi_{j,k})\delta_{t_j,\sigma_{j,k}} + g'(\alpha,x_j,\chi_{j,k})\delta_{\sigma_{j,k},t_j}\right)\\
&+\sum_j \left( h(\alpha,x_j)(\delta_{t_j}\otimes\lambda) + h'(\alpha,x_j)(\lambda\otimes\delta_{t_j}) \right)\\
&+\sum_k \left( l(\alpha,\eta_k)\delta_{\rho_k,\rho'_k} +l'(\alpha,\eta_k)\delta_{\rho'_k,\rho_k}
\right),
\end{split}
\label{eq6}
\end{equation}
for some measurable functions $f\geq 0$ on $\R_+^4$, $g,g'\geq 0$ on
$\R_+^3$, and $h,h',l,l'\geq 0$ on $\R_+^2$, a set of independent uniform
random variables $(\zeta_{\{i,j\}})_{i,j\in \N}$ with values in $[0,1]$,
independent unit rate Poisson processes $(t_j,x_j)_{j\in\N}$ and
$(\sigma_{i,j},\chi_{i,j})_{j\in\N}$ on $\R_+^2$ for $i\in\N$ and
$(\rho_j,\rho'_j,\eta_j)_{j\in\N}$ on $\R_+^3$, an independent set of random
variables $\alpha,\beta,\gamma\geq 0$, and $\lambda$ (resp.\ $\lambda_D$)
denoting Lebesgue measure on $\R_+$ (resp.\ the diagonal $x_1=x_2\geq 0$).

By the definition \eqref{eq5} in Section~\ref{sec:W-random-results} of $\xi$ as a sum of point masses, all the
terms in \eqref{eq6} involving Lebesgue measure must be zero; i.e., except on
an event of probability zero, $\beta=\gamma=0$ and
$h(\alpha,x_j)=h'(\alpha,x_j)= 0$ for all $j\in\N$. Recall that by
\eqref{eq19} each vertex can be uniquely identified with the time $t\geq 0$
when it appeared in the graph, and that each point mass $\delta_{t,t'}$ with
$t,t'\geq 0$ represents an edge between the two vertices associated with $t$
and $t'$. Almost surely no two of the random variables
$\rho_k,\rho'_k,t_i,t_j,\sigma_{j,k}$ for $i,j,k\in\N$ have the same value,
and hence the functions $f,g,g',l,l'$ take values in $\{0,1\}$ almost
everywhere. Furthermore, since the graphs $\wt G_t$ are undirected, we have
$g=g'$ and $l=l'$, and $f$ is symmetric in its second and third input
argument.

First we will argue that the subgraphs $\wh G_t$ of $\wt G_t$ corresponding
to the terms
\[
f(\alpha,x_i,x_j,\zeta_{\{i,j\}})\delta_{t_i,t_j}
\]
have the law of a graphon process with isolated vertices generated by some
(possibly random) graphon $\mcl W$. Condition on the realization of $\alpha$,
and define the function $W_\alpha\colon \R_+^2\to[0,1]$ by
\[
W_\alpha(x,x') := \Pr( f(\alpha,x,x',\zeta_{\{i,j\}})=1\,|\,\alpha)\qquad \text{for all } x,x'\in\R_+.
\]
It follows that, conditioned on $\alpha$ such that $W_\alpha\in L^1$, $(\wh
G_t)_{t\geq 0}$ has the law of a graphon process generated by $\mcl
W_\alpha=(W_\alpha,\R_+)$. To conclude we need to prove that $W_\alpha\in
L^1$ almost surely, which will be done in the next two paragraphs.

First we will argue that $(\wh G_t)_{t\geq 0}$ has uniformly regular tails.
Since no two of the random variables $\rho_k,\rho'_k,t_i,t_j,\sigma_{j,k}$
have the same value for $i,j,k\in\N$, each point mass
$\delta_{\rho_k,\rho'_k}$ or $\delta_{\rho'_k,\rho_k}$ of $\xi$ corresponds
to an isolated edge, i.e., an edge between two vertices each of degree one,
and each point mass $\delta_{\sigma_{j,k},t_j}$ or
$\delta_{t_j,\sigma_{j,k}}$ of $\xi$ corresponds to an edge between the
vertex $t_j$ in $\wh G_t$ and a vertex of degree one; i.e.,
$\sum_{k\in\N}(\delta_{t_j,\sigma_{j,k}}+\delta_{\sigma_{j,k},t_j})$
corresponds to a star centered at the vertex associated with $t_j$. Note that
$\wh G_t$ and $\wt G_t$ satisfy the conditions of Lemma~\ref{prop12} with
$d=1$. Hence $\lim_{t\rta\infty} |E(\wh G_t)|/|E(\wt G_t)|=1$, and since $\wt
G_t$ has uniformly regular tails this implies that $\wh G_t$ must also have
uniformly regular tails.

We assume that $W_\alpha$ is not almost surely integrable, and will derive a
contradiction. We condition on $\alpha$ such that $W_\alpha\not\in L^1$, and
to simplify notation we will write $W$ instead of $W_\alpha$. Let $\wh
G^{+}_t$ (resp.\ $\wh G^{-}_t$) be the induced subgraph of $\wh G_t$
consisting of the vertices for which the feature $x$ satisfies $x\in
I:=\{x'\in\R_+\,:\,\int_{\R_+} W(x,x')\,dx'\geq 1\}$ (resp.\ $x\not\in I$).
Let $\W^{+}=(W^{+},\R_+)$ and $\W^{-}=(W^{-},\R_+)$ denote the corresponding
graphons (we will see shortly that they are integrable), i.e.,
$W^{+}=W\1_{I\times I}$ and $W^{-}=W\1_{I^c\times I^c}$. Since $|E(\wh
G_t)|<\infty$ a.s.\ for each $t\geq 0$ the measure $\xi$ is locally finite
a.s. We deduce from this that $\lambda(I)<\infty$ and $\|W^-\|_1 <\infty$
\cite[Theorem~9.25, (iii) and (iv)]{kallenberg-exch}. By applying
Proposition~\ref{prop2} this implies further that
\[
\lim_{t\rta\infty}t^{-2} |E(\wh G^{+}_t)|=\frac 12 \|W^+\|_1 \le \frac 12 \lambda(I)^2<\infty,
\]
\[
\lim_{t\rta\infty}t^{-2} |E(\wh G^{-}_t)|=\frac 12 \|W^-\|_1<\infty,
\]
and
\[
\lim_{t\rta\infty}t^{-2} |E(\wh G_t)|=\infty.
\]
It follows that if $\wt E(\wh G_t):=E(\wh G_t)\backslash (E(\wh G_t^{+})\cup
E(\wh G_t^{-}))$ is the set of edges having one endpoint in $V(\wh G_t^{+})$
and one endpoint in $V(\wh G_t^{-})$, we have
\begin{equation}
\lim_{t\rta\infty}|\wt E(\wh G_t)|/|E(\wh G_t)|=1.
\label{eq33}
\end{equation}
For the stretched canonical graphon $\W^{\wh G_t,s}$ the edges $\wt E(\wh
G_t)$ correspond to $A:=(J_t\times J_t^c)\cup(J_t^c\times J_t)\subset
\R_+^2$, where $J_t\subset\R_+$ corresponds to $V(\wh G_t^+)$. Since $|V(\wh
G_t^+)|=\Theta(t)$, we have $\lambda(J_t)=|V(\wh G_t^+)| (2|E(\wh
G_t)|)^{-1/2}=o_t(1)$. By \eqref{eq33} and $\|W^{\wh G_t,s}\|_1=1$, we have
$\lim_{t\rta\infty}\|W^{\wh G_t,s}\1_{A^c}\|_1=0$. Since
$\lambda(J_t)=o_t(1)$ and $W^{\wh G_t,s}$ takes values in $[0,1]$, we have
$\lim_{t\rta\infty}\|W^{\wh G_t,s}\1_{A\cap U_t^2}\|_1=0$ for all sets $U_t
\subset \R_+$ of bounded measure. We have obtained a contradiction to the
hypothesis of uniform regularity of tails, since
\[
\lim_{t\rta\infty}\|W^{\wh G_t,s}\1_{U_t^2}\|_1
\leq \lim_{t\rta\infty}\|W^{\wh G_t,s}\1_{U_t^2\cap A}\|_1
+ \lim_{t\rta\infty}\|W^{\wh G_t,s}\1_{A^c}\|_1 = 0.
\]

To complete the proof that $(\wt G_t)_{t\geq 0}$ has the law of a graphon
process with isolated vertices, we need to argue that a.s.
\begin{equation}
l(\alpha,\eta_k) =
g(\alpha,x_j,\chi_{j,k})
=0
\qquad \text{for all } k,j\in\N.
\label{eq7}
\end{equation}
Let $N_t\in\N_0$ denote the number of edges associated with terms of $\xi$ of
the form $\delta_{\rho_k,\rho'_k}+\delta_{\rho'_k,\rho_k}$, and let $\wt N_t$
denote the number of edges associated with terms of $\xi$ of the form
$\delta_{\sigma_{j,k},t_j}+\delta_{t_j,\sigma_{j,k}}$. Since $\wh G_t$ and
$\wt G_t$ satisfy the conditions of Lemma~\ref{prop12} with $d=1$, and since
Lemma~\ref{prop2} implies that a.s.-$\lim_{t\rta\infty} |E(\wt
G_t)|/t^2=\frac 12 \|W\|_1$, it follows that a.s.\
\begin{equation}
\lim_{t\rta\infty}N_t/t^2=0\quad\text{and}\quad
\lim_{t\rta\infty}\wt N_t/t^2=0.
\label{eq8}
\end{equation}
We will prove \eqref{eq7} by contradiction, and will consider each term
separately. First assume $\lambda(\supp (l(\alpha,\phantom{}\cdot\phantom{}))>0$ with positive
probability. Conditioned on a realization of $\alpha$ such that
$p:=\lambda(\supp (l(\alpha,\phantom{}\cdot\phantom{}))>0$, the random variable $N_t$ is a
Poisson random variable with expectation $t^2p$. Hence
$\lim_{t\rta\infty}N_t/t^2=p$, which is a contradiction to \eqref{eq8}. It
follows that $\lambda(\supp (l(\alpha,\phantom{}\cdot\phantom{}))=0$ a.s., and thus
$l(\alpha,\eta_k)=0$ for all $k\in\N$ a.s.

Now assume $\lambda(\supp (g(\alpha,\phantom{}\cdot\phantom{},\phantom{}\cdot\phantom{})))>0$ with positive
probability. Then there exists $\ep>0$ such that with positive probability
there is a set $I\subset\R_+$ (depending on $\alpha$) satisfying
$\lambda(I)=\ep$, and such that for all $x\in I$ it holds that
$\lambda(I_x)>\ep$, where $I_x :=\{x'\in\R_+\,:\, g(\alpha,x,x')=1 \}$.
Consider the Poisson point process $(t_j,x_j)_{j\in\N}$ corresponding to the
graphon process $\wt G_t$ with isolated vertices. The number of points
$(t_j,x_j)\in [0,t]\times I$ evolves as a function of $t$ as a Poisson
process with rate $\ep>0$; hence the number of such points divided by $t$
converges to $\ep$ a.s. For any given pair $(t_j,x_j)\in [0,t]\times I$ the
number of points $(\sigma_{i,j},\chi_{i,j})\in [0,t]\times I_{x_j}$ for the
Poisson point process $(\sigma_{i,j},\chi_{i,j})_{i\in\N}$ has the law of a
Poisson random variable with intensity greater than $t\ep$. Hence,
\[
t^{-2}\lim_{t\rta\infty} \wt N_t
= t^{-2}\lim_{t\rta\infty} \sum_{j,k\,:\,t_j,\sigma_{j,k}\leq t}
g(\alpha,x_j,\chi_{j,k}) > \ep^2.
\]
This contradicts \eqref{eq8}, and thus completes our proof that $(\wt
G_t)_{t\geq 0}$ has the law of a graphon process with isolated vertices.
\end{proof}

\begin{remark}
In our proof above we observed that the assumption of exchangeability alone
is not sufficient to prove that $(\wt G_t)_{t\geq 0}$ has the law of a
graphon process with isolated vertices. More precisely, without this
assumption we might have $W\not\in L^1$ and the measure might also consist of
the terms containing $g,g'$ and $l,l'$. We observed in the proof that the
terms containing $l,l'$ correspond to isolated edges, and that the terms
containing $g,g'$ correspond to ``stars'' centered at a vertex in the graphon
process. It is outside the scope of this paper to do any further analysis of
these more general exchangeable graphs.
\end{remark}

\section{Left Convergence of Graphon Processes}
\label{sec5}

In this appendix we will prove Proposition~\ref{prop3}. The following lemma
will imply part (ii) of the proposition.

\begin{lemma}
Let $\W$ be a bounded, non-negative graphon, and assume that $h(F_k,\mcl
W)<\infty$ for a star $F_k$ with $k$ leaves.  Then $h(F,\mcl W)<\infty$ for
all simple, connected graphs $F$ of maximal degree at most $k$.
\label{lem:h-bd}
\end{lemma}

\begin{proof}
We first note that if $h(F_k,\mcl W)=\int  D_W(x)^k\, d\mu(x)<\infty$ for a
star with $k$ leaves, then the same holds for all stars with at most $k$
leaves, since we know that $D_W\in L^1(S)$ by our definition of a graphon.
Also, using that $W$ is bounded, we assume without loss of generality that
$F$ is a tree $T$ of maximal degree $\Delta\leq k$.

Designate one of the vertices, $r$, as the root of the tree, and choose a
vertex $u_1$ such that no other vertex is further from the root.  If $u_1$
has distance $1$ from $r$, then $T$ is a star and there is nothing to prove.
Let $u$ be $u_1$'s grandparent, let $v$ be its parent, and let $u_2,\dots,
u_s$ for $ 1\leq s\leq \Delta -1$ be its siblings. Note that by our
assumption on $u_1$, all the siblings $u_1,\dots,u_s$ are leaves.
Furthermore, if their grandparent $u$ is the root and the root has no other
children, then $T$ is again a star, so we can rule that out as well.

If we remove the edge $uv$ from $T$, we obtain two disjoint trees $T_1$ and
$T_2$, and as just argued, the one containing $u$ is a tree with at least $2$
vertices and maximal degree at most $\Delta$, while the second one is a star,
again of maximal degree at most $\Delta$. Because $h(T,\mcl W)\leq
\|W\|_\infty h(T_1,\mcl W) h(T_2,\mcl W)$, the lemma now follows by
induction.
\end{proof}

\begin{proof}[Proof of Proposition~\ref{prop3}]
We will start by proving (i). Fix some simple connected graph $F$ with $k$
vertices. By Proposition~\ref{prop2} applied with $F$ and the simple
connected graph on two vertices, respectively,
\[
\lim_{t\rta\infty}t^{-k}\mathrm{inj}(F,G_t)=\|W\|^{k/2}_1 h(F,\cW),\qquad
2\lim_{t\rta\infty}t^{-2} |E(G_t)|= \|W\|_1,
\]
and hence
\begin{equation}
\lim_{t\rta\infty} |2E(G_t)|^{-k/2} \mathrm{inj}(F,G_t)=h(F,\cW),
\label{eq26}
\end{equation}
proving (i).

(ii) Since $h(F,\W)=\int D_w^k$ if $F$ is a star with $k$ leaves, we can use
Lemma~\ref{lem:h-bd} to conclude that $h(F,\W)<\infty$ for every simple
connected graph $F$ with at least two vertices.  Express
$\mathrm{hom}(F,G_t)$ as $\mathrm{hom}(F,G_t)=\sum_{\Phi}
\mathrm{inj}(F/\Phi,G_t)$, where we sum over all equivalence relations $\Phi$
on $V(F)$. By Proposition~\ref{prop2} applied with $F/\Phi$, we have
\[
\lim_{t\rta\infty}|2E(G_t)|^{-k/2}\mathrm{inj}(F/\Phi,G_t)=0
\]
 unless $\Phi$
is the equivalence relation for which the number of equivalence classes
equals $|V(F)|$. Hence the estimate \eqref{eq26} holds with $\mathrm{hom}$ in
place of $\mathrm{inj}$, which completes the proof of (ii).

Next we will prove (iii). Let $F$ be a simple connected graph with at least
three vertices, and assume $d\in\N$ is such that the degree of the vertices
of $G_n$ is bounded by $d$. We may assume $G_n$ has no isolated vertices,
since $h(F,G_n)$ is invariant under adding or deleting isolated vertices.
Under the assumption of no isolated vertices, we have $|E(G_n)|\geq
|V(G_n)|/2$. By boundedness of degrees, $\op{hom}(F,G_n)\leq
|V(G_n)|d^{|V(F)|-1}$. Combining these estimates, $h(F,G_n) \leq
|V(G_n)|^{1-|V(F)|/2} d^{|V(F)|-1}$, from which the desired result follows.

Now we will prove (iv). We first construct an example of a sequence of graphs
$(G_n)_{n\in\N}$ which converges for the stretched cut metric
$\delta_\square^s$, but which is not left convergent. Let $(\wt
G_n)_{n\in\N}$ be a sequence of simple dense graphs with $|V(\wt
G_n)|\rta\infty$ that is convergent in the $\delta_\square$ metric, and hence
also in the $\delta_\square^s$ metric.

Define $G_n:=\wt G_n$ for even $n$, and for odd $n$ let $G_n$ be the union
of $\wt G_n$ and $|E(\wt G_n)|^{7/8}$ vertices of degree one, which are all
connected to the same uniformly random vertex of $\wt G_n$. Then
$(G_n)_{n\in\N}$ converges for $\delta_\square^s$ with the same limit as
$(\wt G_n)_{n\in\N}$, since $\wt G_n$ is an induced subgraph of $G_n$ and
$|E(\wt G_n)|/|E(G_n)|\rta 1$. On the other hand, $(G_n)_{n\in\N}$ is not
left convergent, since if $F$ is the simple connected graph with three
vertices and two edges, then $\op{hom}(F,G_n)= \Omega(|E(\wt G_n)|^{14/8})$
for odd $n$ and hence $h(F,G_n)\to\infty$ along sequences of odd $n$, while
$h(F,G_n)$ converges to a finite number by the fact that dense graph
sequences which are convergent in the cut metric are left convergent.

Finally we will provide a counterexample in the reverse direction; i.e., we
will construct a sequence of graphs $(G_n)_{n\in\N}$ which is left
convergent, but does not converge for the stretched cut metric. Let $(\wt
G_n)_{n\in\N}$ be left convergent and satisfy $\lim_{n\rta\infty}|E(\wt
G_n)|=\infty$, and let $G_n$ be the union of $\wt G_n$ and $|E(\wt G_n)|$
isolated edges. Then $(G_n)_{n\in\N}$ is left convergent, since
$\op{hom}(F,G_n)=\op{hom}(F,\wt G_n)+2|E(\wt G_n)|$ when $F$ is the simple
connected graph on two vertices, and $\op{hom}(F,G_n)=\op{hom}(F,\wt G_n)$
when $F$ has at least three vertices. On the other hand, $(G_n)_{n\in\N}$ is
not convergent for $\delta_\square^s$, since it does not have uniformly
regular tails.
\end{proof}


\vskip 0.2in
\begin{thebibliography}{37}
\providecommand{\natexlab}[1]{#1}
\providecommand{\url}[1]{\texttt{#1}}
\expandafter\ifx\csname urlstyle\endcsname\relax
  \providecommand{\doi}[1]{doi: #1}\else
  \providecommand{\doi}{doi: \begingroup \urlstyle{rm}\Url}\fi

\bibitem[Aldous(1981)]{aldous}
D.~J.~Aldous.
\newblock Representations for partially exchangeable arrays of random
  variables.
\newblock \emph{J.~Multivariate Anal.}, 11\penalty0 (4):\penalty0 581--598,
  1981.
\newblock \doi{10.1016/0047-259X(81)90099-3}.

\bibitem[Bickel and Chen(2009)]{bickelchen}
P.~J.~Bickel and A.~Chen.
\newblock A nonparametric view of network models and {N}ewman-{G}irvan and
  other modularities.
\newblock \emph{Proc.\ Natl.\ Acad.\ Sci.\ USA}, 106\penalty0 (50):\penalty0
  21068--21073, 2009.
\newblock \doi{10.1073/pnas.0907096106}.

\bibitem[Bickel et~al.(2011)Bickel, Chen, and Levina]{BCL11}
P.~J.~Bickel, A.~Chen, and E.~Levina.
\newblock The method of moments and degree distributions for network models.
\newblock \emph{Ann.~Statist.}, 39\penalty0 (5):\penalty0 2280--2301, 2011.
\newblock \doi{10.1214/11-AOS904}.

\bibitem[Billingsley(1999)]{billingsley-book}
P.~Billingsley.
\newblock \emph{Convergence of Probability Measures}.
\newblock John Wiley \& Sons, Inc., New York, second edition, 1999.
\newblock \doi{10.1002/9780470316962}.

\bibitem[Bollob{\'a}s et~al.(2007)Bollob{\'a}s, Janson, and Riordan]{bjr07}
B.~Bollob{\'a}s, S.~Janson, and O.~Riordan.
\newblock The phase transition in inhomogeneous random graphs.
\newblock \emph{Random Structures Algorithms}, 31\penalty0 (1):\penalty0
  3--122, 2007.
\newblock \doi{10.1002/rsa.20168}.

\bibitem[Bollob{\'a}s and Riordan(2009)]{br-09}
B.~Bollob{\'a}s and O.~Riordan.
\newblock Metrics for sparse graphs.
\newblock In \emph{Surveys in Combinatorics 2009}, volume 365 of \emph{London
  Math.~Soc.~Lecture Note Ser.}, pages 211--287. Cambridge Univ.\ Press,
  Cambridge, 2009.
\newblock \doi{10.1017/CBO9781107325975.009}.

\bibitem[Borgs et~al.(2015)Borgs, Chayes, Cohn, and Ganguly]{w-estimation}
C.~Borgs, J.~T.~Chayes, H.~Cohn, and S.~Ganguly.
\newblock {Consistent nonparametric estimation for heavy-tailed sparse graphs}.
\newblock Preprint, \arXiv{1508.06675}, 2015.

\bibitem[Borgs et~al.(2017)Borgs, Chayes, Cohn, and Holden]{multiscale2}
C.~Borgs, J.~T.~Chayes, H.~Cohn, and N.~Holden.
\newblock {In preparation}, 2017.

\bibitem[Borgs et~al.(2014{\natexlab{a}})Borgs, Chayes, Cohn, and Zhao]{lp1}
C.~Borgs, J.~T.~Chayes, H.~Cohn, and Y.~Zhao.
\newblock {An $L^p$ theory of sparse graph convergence I: limits, sparse random
  graph models, and power law distributions}.
\newblock Preprint, \arXiv{1401.2906}, 2014{\natexlab{a}}.

\bibitem[Borgs et~al.(2014{\natexlab{b}})Borgs, Chayes, Cohn, and Zhao]{lp2}
C.~Borgs, J.~T.~Chayes, H.~Cohn, and Y.~Zhao.
\newblock {An $L^p$ theory of sparse graph convergence II: LD convergence,
  quotients, and right convergence}.
  \newblock To appear in Ann.\ Probab., \arXiv{1408.0744}, 2014{\natexlab{b}}.

\bibitem[Borgs et~al.(2013)Borgs, Chayes, Kahn, and
  Lov{\'a}sz]{left-right-conv}
C.~Borgs, J.~Chayes, J.~Kahn, and L.~Lov{\'a}sz.
\newblock Left and right convergence of graphs with bounded degree.
\newblock \emph{Random Structures Algorithms}, 42\penalty0 (1):\penalty0 1--28,
  2013.
\newblock \doi{10.1002/rsa.20414}.

\bibitem[Borgs et~al.(2010)Borgs, Chayes, and Lov{\'a}sz]{BCL10}
C.~Borgs, J.~Chayes, and L.~Lov{\'a}sz.
\newblock Moments of two-variable functions and the uniqueness of graph limits.
\newblock \emph{Geom.~Funct.~Anal.}, 19\penalty0 (6):\penalty0 1597--1619,
  2010.
\newblock \doi{10.1007/s00039-010-0044-0}.

\bibitem[Borgs et~al.(2006)Borgs, Chayes, Lov{\'a}sz, S{\'o}s, and
  Vesztergombi]{graph-hom}
C.~Borgs, J.~Chayes, L.~Lov{\'a}sz, V.~T.~S{\'o}s, and K.~Vesztergombi.
\newblock Counting graph homomorphisms.
\newblock In \emph{Topics in discrete mathematics}, volume~26 of
  \emph{Algorithms Combin.}, pages 315--371. Springer, Berlin, 2006.
\newblock \doi{10.1007/3-540-33700-8_18}.

\bibitem[Borgs et~al.(2008)Borgs, Chayes, Lov{\'a}sz, S{\'o}s, and
  Vesztergombi]{denseconv1}
C.~Borgs, J.~T.~Chayes, L.~Lov{\'a}sz, V.~T.~S{\'o}s, and K.~Vesztergombi.
\newblock Convergent sequences of dense graphs {I}: {S}ubgraph frequencies,
  metric properties and testing.
\newblock \emph{Adv.~Math.}, 219\penalty0 (6):\penalty0 1801--1851, 2008.
\newblock \doi{10.1016/j.aim.2008.07.008}.

\bibitem[Caron and Fox(2014)]{caron-fox}
F.~Caron and E.~B.~Fox.
\newblock {Sparse graphs using exchangeable random measures}.
\newblock Preprint, \arXiv{1401.1137}, 2014.

\bibitem[Chatterjee(2015)]{C15}
S.~Chatterjee.
\newblock Matrix estimation by universal singular value thresholding.
\newblock \emph{Ann.~Statist.}, 43\penalty0 (1):\penalty0 177--214, 2015.
\newblock \doi{10.1214/14-AOS1272}.

\bibitem[Choi et~al.(2012)Choi, Wolfe, and Airoldi]{CWA}
D.~S.~Choi, P.~J.~Wolfe, and E.~M.~Airoldi.
\newblock Stochastic blockmodels with a growing number of classes.
\newblock \emph{Biometrika}, 99\penalty0 (2):\penalty0 273--284, 2012.
\newblock \doi{10.1093/biomet/asr053}.

\bibitem[{\c{C}}{\i}nlar(2011)]{cinlar-book}
E.~{\c{C}}{\i}nlar.
\newblock \emph{Probability and Stochastics}, volume 261 of \emph{Graduate
  Texts in Mathematics}.
\newblock Springer, New York, 2011.
\newblock \doi{10.1007/978-0-387-87859-1}.

\bibitem[Diaconis and Janson(2008)]{diaconisjanson08}
P.~Diaconis and S.~Janson.
\newblock Graph limits and exchangeable random graphs.
\newblock \emph{Rend.~Mat.~Appl.~(7)}, 28\penalty0 (1):\penalty0 33--61, 2008.

\bibitem[Frieze and Kannan(1999)]{FK99}
A.~Frieze and R.~Kannan.
\newblock Quick approximation to matrices and applications.
\newblock \emph{Combinatorica}, 19\penalty0 (2):\penalty0 175--220, 1999.
\newblock \doi{10.1007/s004930050052}.

\bibitem[Gao et~al.(2015)Gao, Lu, and Zhou]{rateopt}
C.~Gao, Y.~Lu, and H.~H.~Zhou.
\newblock Rate-optimal graphon estimation.
\newblock \emph{Ann.~Statist.}, 43\penalty0 (6):\penalty0 2624--2652, 2015.
\newblock \doi{10.1214/15-AOS1354}.

\bibitem[Halmos(1974)]{halmos}
P.~R.~Halmos.
\newblock \emph{Measure Theory}, volume~18 of \emph{Graduate Texts in
  Mathematics}.
\newblock Springer, New York, 1974.
\newblock \doi{10.1007/978-1-4684-9440-2}.

\bibitem[Herlau et~al.(2016)Herlau, Schmidt, and M{\o}rup]{kallenberg-block}
T.~Herlau, M.~N.~Schmidt, and M.~M{\o}rup.
\newblock Completely random measures for modelling block-structured sparse
  networks.
\newblock In D.~D.~Lee, M.~Sugiyama, U.~V.~Luxburg, I.~Guyon, and R.~Garnett,
  editors, \emph{Advances in Neural Information Processing Systems 29}, pages
  4260--4268. Curran Associates, Inc., 2016.

\bibitem[Hoff et~al.(2002)Hoff, Raftery, and Handcock]{HRH02}
P.~D.~Hoff, A.~E.~Raftery, and M.~S.~Handcock.
\newblock Latent space approaches to social network analysis.
\newblock \emph{J.~Amer.~Statist.~Assoc.}, 97\penalty0 (460):\penalty0
  1090--1098, 2002.
\newblock \doi{10.1198/016214502388618906}.

\bibitem[Hoover(1979)]{hoover}
D.~Hoover.
\newblock Relations on probability spaces and arrays of random variables.
\newblock Preprint, Institute for Advanced Study, Princeton, NJ, 1979.

\bibitem[Janson(2013)]{janson-survey}
S.~Janson.
\newblock \emph{Graphons, Cut Norm and Distance, Couplings and Rearrangements},
  volume~4 of \emph{New York Journal of Mathematics Monographs}.
\newblock State University of New York, University at Albany, Albany, NY, 2013.

\bibitem[Janson(2016)]{janson16}
S.~Janson.
\newblock {Graphons and cut metric on {$\sigma$-finite} measure spaces}.
\newblock Preprint, \arXiv{1608.01833}, 2016.

\bibitem[Kallenberg(2002)]{kallenberg-prob}
O.~Kallenberg.
\newblock \emph{Foundations of Modern Probability}.
\newblock Springer, New York, second edition, 2002.

\bibitem[Kallenberg(2005)]{kallenberg-exch}
O.~Kallenberg.
\newblock \emph{Probabilistic Symmetries and Invariance Principles}.
\newblock Springer, New York, 2005.
\newblock \doi{10.1007/0-387-28861-9}.

\bibitem[Klopp et~al.(2017)Klopp, Tsybakov, and Verzelen]{oracle}
O.~Klopp, A.~B.~Tsybakov, and N.~Verzelen.
\newblock Oracle inequalities for network models and sparse graphon estimation.
\newblock \emph{Ann.~Statist.}, 45\penalty0 (1):\penalty0 316--354, 2017.
\newblock \doi{10.1214/16-AOS1454}.

\bibitem[Lov{\'a}sz and Szegedy(2006)]{ls-graphlimits}
L.~Lov{\'a}sz and B.~Szegedy.
\newblock Limits of dense graph sequences.
\newblock \emph{J.~Combin.~Theory Ser.~B}, 96\penalty0 (6):\penalty0 933--957,
  2006.
\newblock \doi{10.1016/j.jctb.2006.05.002}.

\bibitem[Lov{\'a}sz and Szegedy(2007)]{ls-analyst}
L.~Lov{\'a}sz and B.~Szegedy.
\newblock Szemer\'edi's lemma for the analyst.
\newblock \emph{Geom.~Funct.~Anal.}, 17\penalty0 (1):\penalty0 252--270, 2007.
\newblock \doi{10.1007/s00039-007-0599-6}.

\bibitem[Rohe et~al.(2011)Rohe, Chatterjee, and Yu]{roheetal}
K.~Rohe, S.~Chatterjee, and B.~Yu.
\newblock Spectral clustering and the high-dimensional stochastic blockmodel.
\newblock \emph{Ann.~Statist.}, 39\penalty0 (4):\penalty0 1878--1915, 2011.
\newblock \doi{10.1214/11-AOS887}.

\bibitem[Stroock(2011{\natexlab{a}})]{stroock-analyticview}
D.~W.~Stroock.
\newblock \emph{Probability Theory: An Analytic View}.
\newblock Cambridge University Press, Cambridge, second edition,
  2011{\natexlab{a}}.

\bibitem[Stroock(2011{\natexlab{b}})]{stroock-integration}
D.~W.~Stroock.
\newblock \emph{Essentials of Integration Theory for Analysis}, volume 262 of
  \emph{Graduate Texts in Mathematics}.
\newblock Springer, New York, 2011{\natexlab{b}}.
\newblock \doi{10.1007/978-1-4614-1135-2}.

\bibitem[Veitch and Roy(2015)]{veitchroy}
V.~Veitch and D.~M.~Roy.
\newblock The class of random graphs arising from exchangeable random measures.
\newblock Preprint, \arXiv{1512.03099}, 2015.

\bibitem[Wolfe and Olhede(2013)]{WO}
P.~J.~Wolfe and S.~C.~Olhede.
\newblock Nonparametric graphon estimation.
\newblock Preprint, \arXiv{1309.5936}, 2013.

\end{thebibliography}
\end{document}